\documentclass[11pt]{amsart}

\usepackage{xcolor}
\usepackage[draft,inline,nomargin]{fixme}
\usepackage[top=1in,bottom=1in,left=1in,right=1in]{geometry}

\usepackage{multirow}
\usepackage{array}
\usepackage{booktabs}

\usepackage{cellspace}
\usepackage{enumerate}
\usepackage{amsthm}
\usepackage{thmtools}
\usepackage{amsmath}
\usepackage{amsfonts,amssymb}
\usepackage{mathrsfs}
\usepackage{graphicx}
\usepackage{hhline}
\usepackage{blkarray}
\usepackage[unicode=true,colorlinks=true,linkcolor=blue,urlcolor=blue, citecolor=blue]{hyperref}

\usepackage[all,arc,curve,color,frame]{xy}
\usepackage{colordvi}
\usepackage{vmargin}
\usepackage{pgf,tikz}
\usetikzlibrary{arrows}

\usepackage{algpseudocode}

\usepackage{pdflscape}
\usepackage{afterpage}
\usepackage{caption}

\usepackage{amscd}
\usepackage{xcolor}
\usepackage{ifpdf} 
\usepackage[all]{xy}
\ifpdf
\else
\fi

\numberwithin{equation}{section}
\numberwithin{figure}{section}
\numberwithin{table}{section}

\makeatletter
\@namedef{subjclassname@2010}{%
 \textup{2010} Mathematics Subject Classification}
\makeatother

\declaretheorem[style=plain,parent=section]{theorem}

\declaretheorem[style=plain,sibling=theorem]{lemma}
\declaretheorem[style=plain]{claim}
\declaretheorem[style=plain,sibling=theorem]{proposition}

\declaretheorem[style=definition,sibling=theorem]{definition}

\declaretheorem[style=definition, qed=\hfill $\diamond$, sibling=definition]{example}
\declaretheorem[style=remark,sibling=theorem]{remark}

\usepackage[boxed,vlined,lined,algoruled,algosection]{algorithm2e}

\renewcommand{\algorithmautorefname}{Algorithm}

\newcommand{\RR}{\ensuremath{\mathbb{R}}}
\newcommand{\PP}{\ensuremath{\mathbb{P}}}
\newcommand{\resK}{\ensuremath{\widetilde{K}}}
\newcommand{\CC}{\ensuremath{\mathbb{C}}}
\newcommand{\NN}{\ensuremath{\mathbb{N}}}
\newcommand{\ZZ}{\ensuremath{\mathbb{Z}}}
\newcommand{\QQ}{\ensuremath{\mathbb{Q}}}
\newcommand{\ep}{\varepsilon}
\newcommand{\trop}{\text{trop}}
\newcommand{\Trop}{\text{Trop}}
\newcommand{\val}{\operatorname{val}}
\newcommand{\charF}{\operatorname{char}}

\newcommand{\init}{\operatorname{in}}
\newcommand{\ww}{\omega}
\newcommand{\support}{\operatorname{supp}}
\newcommand{\TP}{\mathbb{T}}
\newcommand{\TPr} {\ensuremath{\TP\PP}}
\newcommand{\PS}{\CC\{\!\{t
\}\!\}}
\newcommand{\relint}{\ensuremath{\operatorname{rel\,int}}}
\newcommand{\rspanone}{\ensuremath{\RR\!\cdot\!\mathbf{1}}}

\newcommand{\Gr}{\operatorname{\mathbb{G}{r}}}
\DeclareMathOperator {\cf}{coeff}
\DeclareMathOperator {\an}{\operatorname{an}}
\DeclareMathOperator {\partua}{\operatorname{part}}
\DeclareMathOperator {\defua}{\operatorname{def}}
\DeclareMathOperator {\tht}{\operatorname{ht}}

\newcommand \dumb{\includegraphics[scale=0.1]{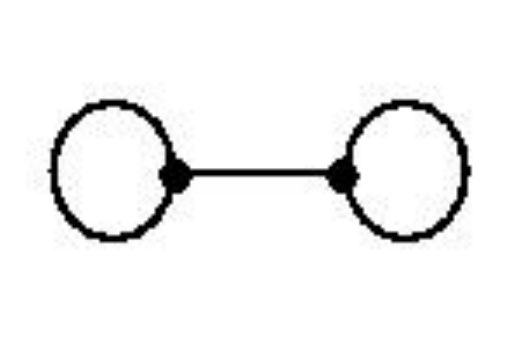}}
\newcommand \figTheta{\includegraphics[scale=0.07]{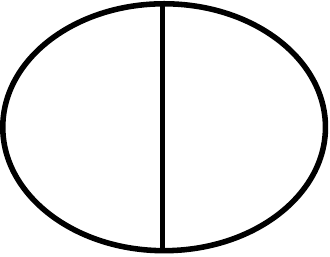}}

\newcommand \sage{\texttt{Sage}}
\newcommand \MacT{\texttt{Macaulay2}}
\newcommand \singular {\texttt{Singular}}

\title{Tropical geometry of genus two curves}

\author [Maria Angelica Cueto and Hannah Markwig]
{
Maria Angelica Cueto
 and Hannah Markwig${}^{\S}$
}

\thanks{${\S}$ \emph{Corresponding author}}

\date{\today}

\begin{document}

\begin{abstract}
We exploit three classical characterizations of smooth genus two curves to study their tropical and analytic counterparts. First, we provide a combinatorial rule to  determine the dual graph of each algebraic curve and the metric structure on its minimal Berkovich skeleton. Our main tool is the description of genus two curves via hyperelliptic covers of the projective line with six branch points. Given the valuations of these six points and their differences, our algorithm provides an explicit harmonic 2-to-1 map to a metric tree on six leaves.
Second, we use tropical modifications to produce a faithful tropicalization in dimension three starting from a planar hyperelliptic embedding.

Finally, we consider the moduli space
of abstract genus two tropical curves and translate the classical Igusa invariants characterizing isomorphism classes of genus two algebraic curves into the tropical realm.
While these tropical Igusa functions do not yield coordinates in the tropical moduli space, we propose an alternative set of invariants that provides new length data.
\end{abstract}

\maketitle

\section{Introduction}\label{sec:introduction}

Algebraic smooth genus two curves defined over an algebraically closed non-Archimedean valued field $K$, with residue field $\resK$ of $\charF{\widetilde{K}}\neq 2$ can be studied from three perspectives:
\begin{enumerate}[(i)]
\item as a planar curve defined by a (dehomogeneized) hyperelliptic equation:
  \begin{equation}\label{eq:hyperelliptic}
    y^2 = u\prod_{i=1}^6 (x-\alpha_i)\;;
  \end{equation}
\item as a $K$-point of the  space  ${M_2}$ of smooth genus two curves;
  \item as a hyperelliptic cover of $\PP^1_K$ with six simple branch points $\alpha_1,\ldots, \alpha_6\in \PP^1_{K}$.
\end{enumerate}
The hyperelliptic cover is determined, up to isomorphism, by a choice of six branch points, i.e., by a $K$-point in the space ${M_{0,6}}$ of smooth rational curves with six marked points.

The top row in~\autoref{fig:diagram} contains the three relevant spaces and maps between them.  The first and third characterizations are related by a projection to the $x$-coordinate and a forgetful map that disregards the planar embedding of the curve induced by~\eqref{eq:hyperelliptic}.

The present paper exploits the aforementioned description to
characterize the tropical and Berkovich non-Archimedean analytic
counterparts of smooth genus two curves. It relies on known comparison
methods between the moduli of (stable) algebraic and abstract tropical
curves via the vertical tropicalization maps
from~\autoref{fig:diagram}~\cite{ACPModuli,cap:13,cha:12}. Such curves
come in seven combinatorial types, and they form a poset under
degenerations. Their associated Berkovich skeleta are obtained as dual
metric graphs to the central fiber of a semistable regular model of
each input curve over the valuation ring $K^{\circ}$ of
$K$~\cite{bak.pay.rab:16, tyo:12}. Each vertex in the graph is
assigned the genus of the corresponding irreducible component as its
weight.  The induced poset of skeleta is depicted on the left
of~\autoref{fig:M2BarAndMumford}. The good reduction case is the only
smooth one and it corresponds to Type (VII). The \emph{tropical moduli
  space} of abstract genus two tropical curves $M_2^{\trop}$ is obtained
as the image of ${M_2}(K)$ under the tropicalization map~\cite[Theorem
  1.2.1]{ACPModuli}. It has the structure of a stacky fan with seven
cones, each labeled by a type and isomorphic to an orthant of
dimension equal to the number of edges on the
skeleton~\cite{ACPModuli,cav.et.at:17,cha:12}. We dicuss this space in
more detail in~\autoref{sec:trop-moduli-spaces}.

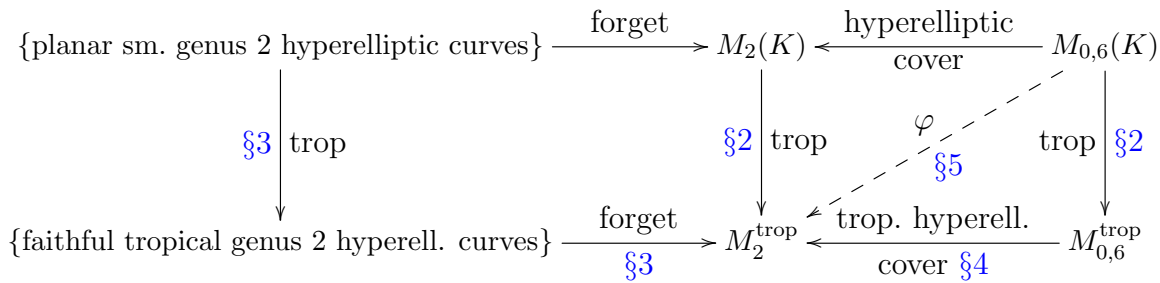
\begin{figure}
\centering  \xymatrix{\{\text{planar sm.~genus 2 hyperelliptic curves}\}\ar[rr]^-{\text{\large{forget}\normalsize}}
    \ar[dd]^-{\text{\large{$\trop$}\normalsize}}_-{\text{\large{\textcolor{blue}{\S}\ref{sec:faithf-trop-trop}}\normalsize}}  && {M_2}(K)\ar[dd]^-{\text{\large{$\trop$}\normalsize}}_-{\text{\large{\textcolor{blue}{\S}\ref{sec:trop-moduli-spaces}\normalsize}}}
    &  & & {M_{0,6}}(K)\ar[lll]_-{\text{\large{hyperelliptic}\normalsize}}^-{\text{\large{cover}\normalsize}} \ar[dd]_-{\text{\large{$\trop$}\normalsize}}^-{\text{\large{\textcolor{blue}{\S}\ref{sec:trop-moduli-spaces}}\normalsize}} \ar@{-->}[ddlll]^-{\text{\large{\textcolor{blue}{\S}\ref{sec:class-thm}}\normalsize}}_-{\text{\large{$\varphi$}\normalsize}}\\
      &&  &&& \\
    \{\text{faithful tropical  genus 2 hyperell. curves}\} \ar[rr]^-{\text{\large{forget}\normalsize}}_-{\text{\large{\textcolor{blue}{\S}\ref{sec:faithf-trop-trop}}\normalsize}} &&  M_{2}^{\trop} &&& M_{0,6}^{\trop} \ar[lll]_-{\text{\large{trop.~hyperell.}\normalsize}}^-{\text{\large{cover~\textcolor{blue}{\S}\ref{sec:trop-hyper-covers}}\normalsize}}}
  \caption{Three ways to represent genus two curves, their relations, and their tropical analogues.\label{fig:diagram}}
  \end{figure}

The tropical moduli space $M_{0,6}^{\trop}$ of rational tropical curves with six marked points is the space of phylogenetic trees on six leaves of Billera-Holmes-Vogtmann~\cite{BHV2001}. It is realized as the image of ${M_{0,6}}(K)$ under the vertical tropicalization map in~\autoref{fig:diagram}, i.e., by taking coordinatewise negative valuations of all $K$-points of  ${M_{0,6}}$ embedded in the toric variety defined by the pointed fan $M_{0,6}^{\trop}\subset \RR^9$. This map and the combinatorial structure of $M_{0,6}^{\trop}$ are also discussed in~\autoref{sec:trop-moduli-spaces}.

As in the algebraic case, abstract genus two tropical  curves are hyperelliptic: they admit a tropical hyperelliptic cover of a metric tree with six markings, given by a  2-to-1 harmonic map branched at all six legs of the tree~\cite{bak.nor:09, cha:12}. We review this construction in~\autoref{sec:trop-hyper-covers}. The tropical covers turn the right square of~\autoref{fig:diagram} into a commuting diagram, but the assignment is not explicit: it requires prior knowledge of each Berkovich skeleton.
We bypass this difficulty by factoring the right square of the diagram through the  map $\varphi$.  The assignment depends  on the  valuations of the  points $\alpha_1,\ldots, \alpha_6\in K^*$  and their differences:
\begin{equation}\label{eq:deWAndDs}
  \ww_i :=-\val(\alpha_i)\;
  \text{ for } i=1,\ldots, 6,\; \text{ and } d_{ij}:=-\val(\alpha_i-\alpha_j) \quad \text{for }i < j, \text{ if } \ww_i=\ww_j.
    \end{equation}
Here is our first main result, which we discuss in~\autoref{sec:class-thm}:
\begin{theorem}\label{thm:admisscovers}
  Each point  in $M_2^{\trop}$ together with an explicit harmonic 2-to-1 map to a metric tree in $M_{0,6}^{\trop}$ is determined by the ordering of  the quantities $\ww_i$ and $d_{ij}$ (see~\autoref{tab:CombAndLengthData}).
\end{theorem}

\noindent
For example,  the two maximal
cells in $M^{\trop}_{2}$ correspond to the orders
$\ww_1\!<\!\ww_2\!<\!\ww_3\!<\!\ww_4\!<\!\ww_5\!<\!\ww_6$ (the dumbbell graph (I)) and
$\ww_1\!<\!\ww_2\!<\!\ww_3\!\leq\! \ww_4\!<\!\ww_5\!<\!\ww_6$ with $d_{34}\!<\!\ww_3$ (the theta graph (II)). They are realized as 2-to-1 harmonic covers of the caterpillar and  snowflake trees as shown in~\autoref{fig:M2BarAndMumford}. Similar results were obtained earlier by Ren-Sam-Sturmfels~\cite[Table 3]{ren.sam.stu:14}  but with very different methods.

Our proof of~\autoref{thm:admisscovers} is sketched in the right of~\autoref{fig:M2BarAndMumford}. Starting from $\TPr^1$,  tropical modifications of $\TPr^1$ at the locations of the points $\ww_i$ dictated by the quantities $d_{ij}$  allow us to construct the target metric trees. The source curve and the map are determined by the tropical Riemann-Hurwitz formula~\cite{buc.mar:15}.
\autoref{pr:classificiationRefTypes} provides a list of seven regions in ${M_{0,6}}(K)$ that surject onto $M_2^{\trop}$. Algorithms~\ref{alg:ineq} and~\ref{alg:sep} take six arbitrary points in $(K^*)^6$ and return a linear change of coordinates of $\PP^1$ that sends these six points to one of these seven witness regions.
The same techniques will lead to a natural extension of~\autoref{thm:admisscovers} to the tropical hyperelliptic locus in $M_g^{\trop}$ for any $g\geq 2$.

\begin{figure}[htb]
  \centering
  \includegraphics[scale=0.41]{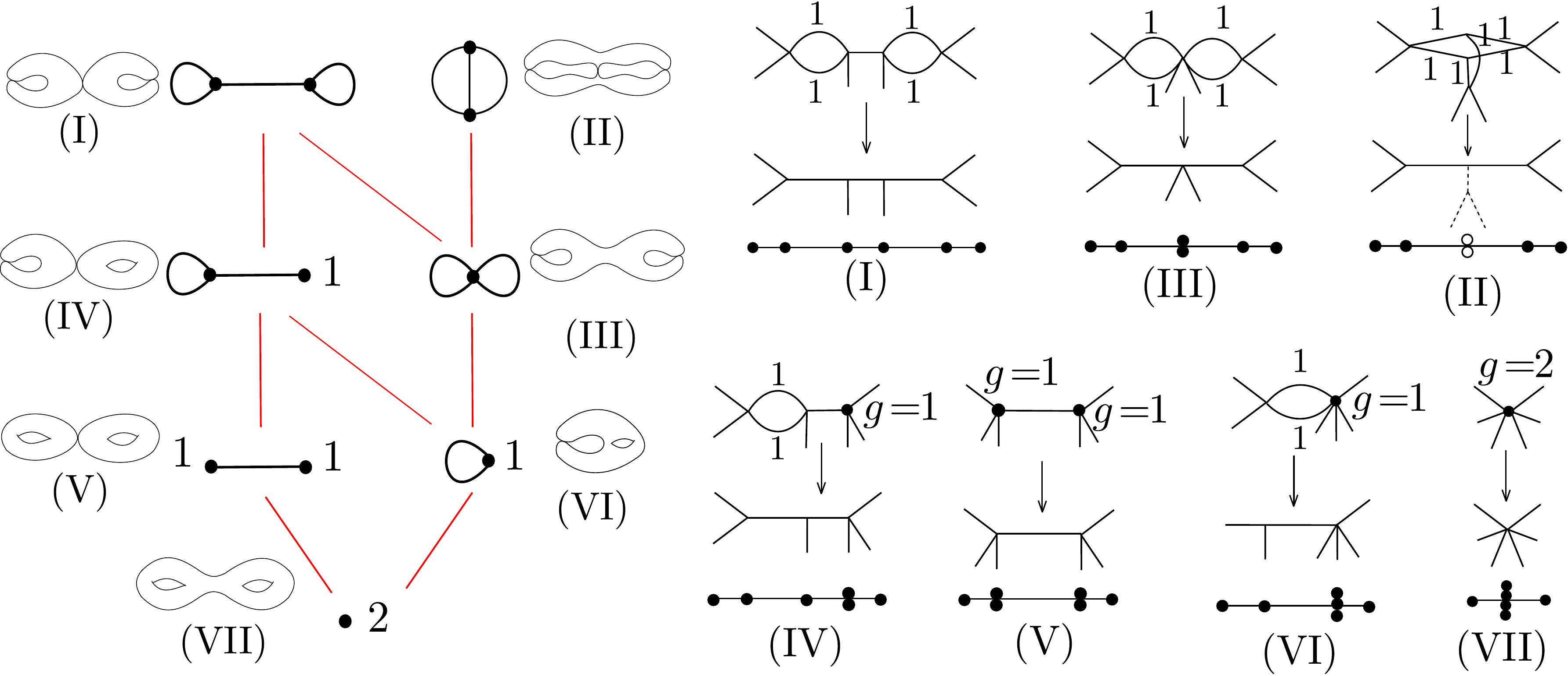}
\caption{From left to right: poset
  of  stable genus two curves, and their weighted dual graphs encoding the genus and intersections of all components;  harmonic 2-to-1 covers of tropical lines with six legs for each type, and ordering of the valuations of the six branch points.  All edge weights in the source curve equal two or one (indicated). All vertices in the source curves have genus zero, unless otherwise indicated.
  The unfilled points in type (II) share initial terms and yield a dashed branch on the metric tree.\label{fig:M2BarAndMumford}}
\end{figure}

\smallskip

The left side of~\autoref{fig:diagram} involves embedded
tropicalizations. Given the hyperelliptic
equation~\eqref{eq:hyperelliptic} defining a smooth genus two curve
$\mathcal{X}$, the tropical plane curve $\Trop\,\mathcal{X}\subset \RR^2$ is the dual
complex of the Newton subdivision of $\mathcal{X}$. An explicit calculation shown
in~\autoref{tab:nonFaithfulness} proves that the planar tropicalization  is always a tree, so it  does
not reflect the genus of our algebraic curve.  Thus, outside Types (V)
and (VII), the minimal Berkovich skeleton of $\mathcal{X}^{\an}$ will not map
isometrically to a subgraph of $\Trop\,\mathcal{X}$ under the hyperelliptic
tropicalization map $\trop\colon \mathcal{X}^{\an}\to \Trop\,\mathcal{X}$.  The
forgetful map on the bottom left of~\autoref{fig:diagram} is analogous
to the retraction map of $\mathcal{X}^{\an}$ onto the minimal Berkovich
skeleton: it shrinks all unbounded edges of the tropical curve and
contracts edges adjacent to one-valent vertices if they correspond to a
rational initial degeneration of $\mathcal{X}$. The map is further described
in~\autoref{sec:faithf-trop-trop}, and it will only be defined if the
tropicalization is faithful.

Faithful tropicalizations are a powerful tool to study non-Archimedean curves through combinatorial means~\cite{bak.pay.rab:16}.  In~\cite{cue.mar:16}, we proposed a program for effectively producing faithfulness for curves over non-Archimedean fields, starting in genus one. 
Our second main result shows that similar methods  can be used to faithfully re-embed genus two curve in three-space in a uniform fashion. The explicit construction is the subject of~\autoref{sec:faithful-reemb} and it relies on the notion of tropical modifications, which we review in~\autoref{sec:faithf-trop-trop}.
\begin{theorem}\label{thm:faithfulRe-embedding}
  Outside Types (V) and (VII), the na\"ive tropicalization induced by the hyperelliptic equation  can be repaired in dimension three by adding one 
  equation of the form $z-f(x,y)$ where  $f$ is linear
  in $y$ and quadratic in $x$. The re-embedded tropical curve contains an isometric copy of the minimal Berkovich skeleton (see~\autoref{tab:nonFaithfulness} and~\autoref{fig:5-7Naive}). 
\end{theorem}
\noindent A precise formula for $f(x,y)$ can be found in~\eqref{eq:liftF}. An alternative refinement of this polynomial, denoted by $\tilde{f}(x,y)$ in~\eqref{eq:liftFtilde} will sometimes be used to simplify the combinatorics. 

In concrete computations, it is always desirable to bound  the ambient dimension required to achieve faithful tropicalizations on  minimal skeleta. In genus two, Wagner~\cite{wag:17} showed that, under certain length restrictions,  any Mumford curve (curves with totally degenerate reduction, namely Types (I), (II) and (III)) can be embedded faithfully in dimension three. Starting from the Schottky uniformization~\cite{GP} of the given Mumford curve, his techniques involve tropical Jacobians, together with an explicit description of the Abel-Jacobi map and they apply not only to the minimal Berkovich skeleta but also to unbounded subgraphs of extended skeleta.

\autoref{thm:faithfulRe-embedding} recovers the same dimension bound for every curve of genus two  where the curve is given by its hyperelliptic equation. In addition to contributing a larger class of curves where the same bound can be attained, our techniques have the additional advantage of extending to the whole hyperelliptic locus in any genus. Generalizations of this result to  extended skeleta are also treated in~\autoref{sec:faithful-reemb}.

\begin{remark}[\textbf{Algorithmic faithful tropicalization in genus $2$}]\label{rem:algf}
Theorems~\ref{thm:admisscovers}
and~\ref{thm:faithfulRe-embedding} can be combined with Algorithms~\ref{alg:ineq} and~\ref{alg:sep} to produce an explicit algorithm that inputs a hyperelliptic equation of the curve $\mathcal{X}$ and outputs a faithful tropicalization. Indeed, starting from the six branch points $\alpha_1,\ldots, \alpha_6$ of the cover, we use 
  Algorithms~\ref{alg:ineq} and~\ref{alg:sep} to construct an automorphism of the projective line that places the branch points in one of the seven special configurations described in~\autoref{tab:CombAndLengthData}. This step recovers the type of the Berkovich skeleton of $\mathcal{X}^{\an}$. With this knowledge, after shifting two of the branch points to be the origin and the  point at infinity via~\autoref{lm:0InftyBranchPoints}, we can pick the appropriate function $f(x,y)$ (which depends on the branch points) that gives the faithful embedding for the minimal Berkovich skeleton by~\autoref{thm:faithfulRe-embedding}. As a result, we obtain an explicit projective model for the input curve $\mathcal{X}$ in dimension three where we detect the topological type of its Berkovich analytification through its embedded tropicalization. In case we wish to recover faithfulness on the extended skeleta we must refine our choice of $f(x,y)$ and perform further linear re-embeddings. These refined methods are type-dependent. We explain them in detail in~Subsections \ref{ss:Dumb}--~\ref{ss:TypeVandVII}.
\end{remark}

A second motivation for Theorems~\ref{thm:admisscovers}
and~\ref{thm:faithfulRe-embedding} and the explicit description of the
diagonal map $\varphi$ from~\autoref{fig:diagram} originates in the
invariant theory of ${M_2}$~\cite{Igusa01} and the search for a coordinate
system for $M_2^{\trop}$.  Defining complete sets of tropical
invariants for each cell in the tropical hyperelliptic locus from
their algebraic counterparts is challenging already in small genera.
The genus one case is well-understood. The $j$-invariant has its
tropical analog: the tropical $j$-invariant. It arises as the expected
negative valuation of the $j$-invariant by using the
conductor-discriminant formula for Weierstrass
equations~\cite{KMM08}. This tropical invariant defines a piecewise
linear function on the space of smooth tropical plane cubics (i.e., the identity on $M_1^{\trop}$) and it is crucial in tropical
enumerative geometry of genus one curves~\cite{ker.mar:09}.

In the algebraic setting, the isomorphism classes of curves of
genus two are determined by the three (absolute) \emph{Igusa
invariants}~\cite{Igusa01}. They can be expressed as rational functions
on all pairwise differences of the six ramification points~\cite{gor.lau:12}. From a computational perspective, they can be viewed as a coordinate-dependent interpretation of the top row in~\autoref{fig:diagram}. We refer to~\autoref{sec:Igusa} for the precise definitions.

Any point on a maximal cell in $M_2^{\trop}$ is determined by three edge lengths: $L_0,L_1$ and $L_2$ in~\autoref{fig:allSkeletaAndMetrics}. In analogy with recent work of Helminck~\cite{hel:16}, our third main result relates these three numbers to the tropicalization of the Igusa invariants, but confirms that these classical invariants are not well suited for tropicalization:
\begin{theorem}\label{thm:IgusaTropical}
  The tropicalization of the Igusa invariants $j_1, j_2$ and $j_3$ are piecewise linear
  functions in $M^{\trop}_2$, with domains of linearity given by the seven cones in $M^{\trop}_2$. They do not form a complete set of
  invariants in $M^{\trop}_2$ since ${j_i}_{\figTheta}^{\trop}\!=\!L_1\!+\!L_0\!+\!L_2$ for
  all $i=1,2,3$, whereas ${j_1}_{\dumb}^{\trop}=L_1\!+\!12L_0\!+\!L_2$, and
  ${j_2}_{\dumb}^{\trop} \! = {j_3}_{\dumb}^{\trop}\!=\!L_1\!+\!8L_0\!+\!L_2$ whenever $\charF \resK \neq 2,3$.

  Replacing $j_3$ by the new invariant $j_4=j_2-4j_3$ induces a piecewise linear function on $M^{\trop}_2$ with ${j_4}_{\figTheta}^{\trop}= L_0+L_1 +L_2 - \min\{L_0, L_1, L_2\}$, and ${j_4}_{\dumb}^{\trop}=\!L_1\!+\!8L_0\!+\!L_2$ when $\charF \resK \neq 2,3$. The tropicalization of the invariants $\{j_1,j_2, j_4\}$  recovers two of the three edge lengths on each point in the tropical moduli space.  Similar formulas hold if $\charF \resK =3$.
\end{theorem}

The ill-behavior of the Igusa invariants under tropicalizations is
similar to a phenomenon occurring in the ring of symmetric
polynomials: power sums will never yield a complete set of tropical
invariants. Indeed, their valuation only captures the root with
lowest valuation. In turn, the elementary symmetric functions enable
us to recover the valuation of all roots. \autoref{thm:IgusaTropical}
manifests again the non-faithfulness of the hyperelliptic embedding
and shows that faithfulness should be viewed as the natural
replacement for the tropical Igusa invariants. It remains an
interesting challenge to find three new algebraic invariants on $M_2$
inducing tropical coordinates on each cell of $M_2^{\trop}$.

\section*{Supplementary material}\label{sec:suppl-mater}
 Many results in this paper rely on calculations performed with \texttt{Singular}~\cite{DGPS} (including its
\texttt{tropical.lib} library~\cite{tropicalLib}), \MacT~\cite{M2}, \texttt{Polymake}~\cite{polymake}
and \sage~\cite{sage}.  We have created supplementary files so that the reader can reproduce all the claimed assertions done via explicit computations and numerical examples. The files are available at:
\begin{center}
  \url{https://people.math.osu.edu/cueto.5/tropicalGeometryGenusTwoCurves/}
\end{center}

  \noindent
  In addition to all \sage\ scripts, the website contains all input and output files both as  \sage\ object files and in plain text.
  We have also included the supplementary files  on the latest
  \texttt{arXiv}
  submission of this paper. They can be obtained by downloading the source.

  \section{Tropical moduli spaces}\label{sec:trop-moduli-spaces}
  In this section, we introduce the objects in the center and right of~\autoref{fig:diagram} involving abstract tropical curves and their moduli spaces.
  
\begin{definition}  
An \emph{abstract tropical curve} is a connected 
{metric graph} consisting of the data of a triple $\Gamma= (G,g,\ell)$ where $G=(V,E,L)$ is a connected graph $G$ with vertices $V$, edges $E$ and unbounded legs $L$ (called \emph{markings}), together with a weight function $g\colon V\to \ZZ_{\geq 0}$ on vertices and a length function $\ell\colon E\to \RR_{>0}$ on edges. Legs are considered to have infinite length. In the absence of legs, we say the curve has no markings.
The genus of a metric graph $\Gamma$ equals
\begin{equation}\label{eq:genus}
\operatorname{genus}(\Gamma):=  b_1(\Gamma) + \sum_{v\in V} g(v),
\end{equation}
where $b_1(\Gamma) = |E|-|V|+1$ is the first Betti number of the graph $G$. A genus zero curve is called \emph{rational}: it corresponds to a metric tree with constant weight function $g\equiv 0$.

An \emph{isomorphism} of a tropical curve is an automorphism of the underlying graph $G$ that respects both the length and weight functions.
The \emph{combinatorial type} of a tropical curve is obtained by disregarding the metric structure, i.e.\ it is given by $(G,g)$.
\end{definition}
The set of all tropical curves with a given a combinatorial type $(G,g)$ can be parameterized by the quotient of an open cone $\RR_{>0}^{E}$ under the action of automorphisms of $G$ that preserve the weight function $g$. Cones corresponding to different combinatorial types can be glued together by collapsing edges and adjusting the genus function accordingly. Such operations keep track 
of possible degenerations of the algebraic curves.~\autoref{fig:M2BarAndMumford} describes this process for unmarked genus two curves. In this way, the tropical moduli space $M_{g,n}^{\trop}$ (respectively, $M_g^{\trop}$) of $n$-marked (respectively, unmarked) curves of genus $g$ inherits the structure of an abstract cone complex. For more details on tropical moduli spaces of curves, we refer to~\cite{ACPModuli, cav.et.at:17, cha:12,GKM,mik:07}.

In this paper, we focus on two examples: $M_{0,6}^{\trop}$ and $M_2^{\trop}$. The first is the space of rational tropical curves with six markings. Up to relabeling of the markings, the moduli space $M_{0,6}^{\trop}$ has two top-dimensional cells, corresponding to the snowflake and caterpillar trees on six leaves. The second object of interest is the space of genus two tropical curves with no marked legs.~\autoref{fig:M2BarAndMumford} shows the labeling of the two top-dimensional cones: the dumbbell and theta graphs, indicated by Types (I) and (II).

\smallskip

The connection between moduli spaces of stable marked curves and their counterparts in tropical geometry has been studied on various occasions~\cite{ACPModuli,gib.mac:10,ren.sam.stu:14}. The spaces $M_{0,n}$ can be identified with a quotient of the open orbit of the cone over the Grassmannian of planes by the torus $(K^*)^n$ and tropicalized thereafter, as in~\cite{ren.sam.stu:14}. In turn, $M_{0,n}^{\trop}$ becomes the space of trees on $n$ leaves~\cite{TropGrass,CompactificationsTori} where we assign length zero to all leaf edges, as we now explain.

Up to an automorphism of $\PP^1$ we may assume that our marked points exclude $(1:0)$ and $(0:1)$, so we identify them with a tuple in $\underline{\alpha} \in (K^*)^n$. The torus $(K^*)^n$ acts on $\Gr_0(2,n)$ by $\underline{t}\star (p_{ij})_{i,j}=(t_it_jp_{ij})_{i,j}$.
In particular, we get an isomorphism
  \begin{equation}\label{eq:PlEmb}
    \Phi\colon M_{0,n}\stackrel{\simeq}\longrightarrow
    \Gr_0(2,n)/(K^*)^n\subset (K^*)^{\binom{n}{2}}/(K^*)^n \qquad \Phi(\underline{\alpha})  = (\alpha_i-\alpha_j)_{1\leq i<j\leq n},
  \end{equation}
  The  space $\overline{M}_{0,n}$ of stable rational curves with $n$ marked points is the tropical compactification of $M_{0,n}$ induced by $M_{0,n}^{\trop}\!:=\Trop\,\Gr_0(2,n)/\RR^n\subset \RR^{\binom{n}{2}}/\RR^n$~\cite[Theorem 5.5]{CompactificationsTori}.   Here, $\RR^n\subset \RR^{\binom{n}{2}}$ is the image of the linear map $\underline{\alpha}\mapsto (\alpha_i+\alpha_j)_{i,j}$. This is precisely the lineality space of $\Trop\,\Gr_0(2,n)$.  It is generated by the $n$ cut-metrics~\cite{TropGrass}.

  The lattice spanned by the cut-metrics has index two in its saturation in $\ZZ^{\binom{n}{2}}$. For this reason, a factor of $1/2$ must be added when considering lattice lengths on the space of trees (see~\cite[Section 3.1]{gro:16}.)
In particular, when $n=6$, the tropicalization map sends a tuple
$\underline{\alpha}$ of six distinct points in $K^*$ to the pairwise \emph{half}-distances
between the legs of the corresponding tree on six leaves:
  \begin{equation}\label{eq:tropMap}
  \trop\colon {M_{0,6}}(K)\to M_{0,6}^{\trop}\subset \RR^{15}/\RR^6 \qquad \trop(\underline{\alpha})=(-\val(\alpha_i-\alpha_j))_{1\leq i<j\leq 6}.
    \end{equation}

  All seven combinatorial types of trees with six leaves are depicted in the right of~\autoref{fig:M2BarAndMumford}. The poset structure of all labeled seven cells matches that of stable genus two curves and their tropical counterparts. Furthermore, the space $M_2^{\trop}$ can be constructed from $M_{0,6}^{\trop}$ via tropical hyperelliptic covers as in~\autoref{sec:trop-hyper-covers}. Indeed, starting from a metric tree $T$ with six leaves, there is a unique tropical hyperelliptic cover of it by a tropical curve $\Gamma$ of genus two with six legs. Our genus two abstract tropical curve will be obtained as the image of $\Gamma$ under the \emph{tropical forgetful map} that contracts all legs and, in turn, all edges adjacent to one-valent vertices of genus zero~\cite{PlaneModuli2015}. This identification describes the commuting right square of~\autoref{fig:diagram}, as proved in~\cite[Theorem 5.3]{ren.sam.stu:14}.

  The tropicalization map $\trop\colon {M_2}(K)\to M_2^{\trop}$ factors through $\trop\colon {M_2}^{\an}\twoheadrightarrow M_2^{\trop}$~\cite[Theorem 1.2.1]{ACPModuli}. Under this map, abstract tropical curves correspond to the minimal Berkovich skeleta: metrized dual graphs of central fibers of semistable regular models of a smooth curve over the valuation ring $K^{\circ}$~\cite{bak.pay.rab:16, tyo:12}.

\begin{figure}[htb]
  \centering
  \includegraphics[scale=0.3]{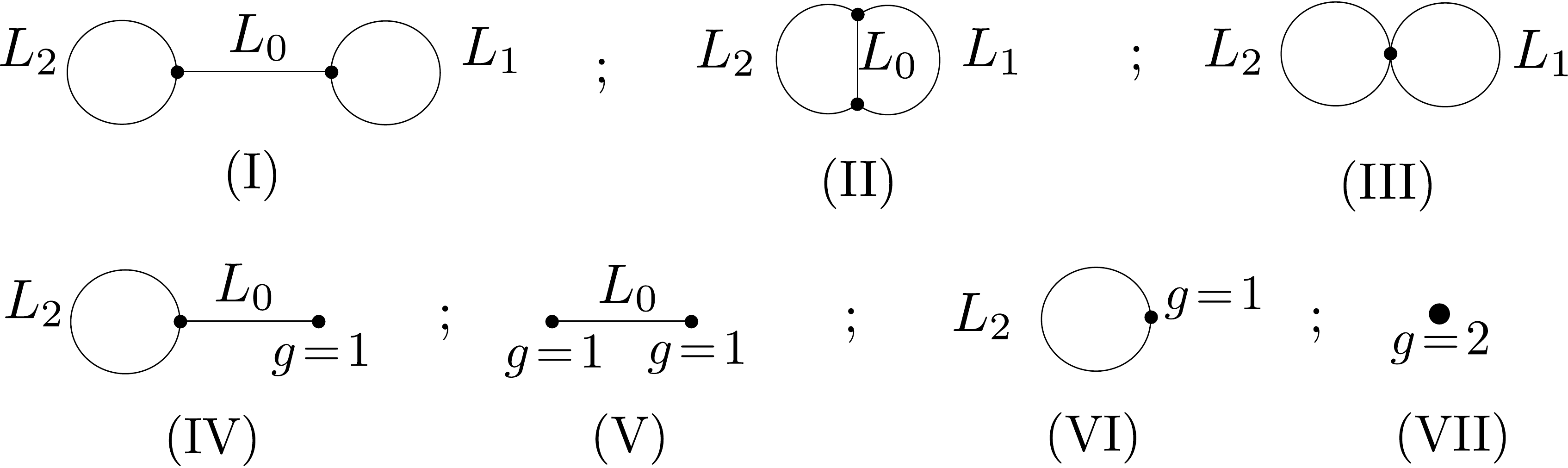}
  \caption{Minimal skeleta of genus two curves obtained by applying the forgetful map to the double covers in the right of~\autoref{fig:M2BarAndMumford}.}\label{fig:allSkeletaAndMetrics}
\end{figure}

  \section{Faithful tropicalization, skeleta and  tropical modifications}
  \label{sec:faithf-trop-trop}

In this section, we discuss {embedded tropicalizations of curves} and their relation to abstract tropical curves and their moduli. Embedded tropical curves are determined by the negative valuations of all $K$-points on a curve $\mathcal{X}$ inside the multiplicative split torus $(K^*)^n$~\cite[Chapter 3]{MSBook}: they are balanced weighted graphs in $\RR^n$ with rational slopes. While this approach is computationally advantageous due to its connection to Gr\"obner degenerations~\cite{gfan} it also poses a major challenge: tropicalization in this setting strongly depends on the embedding. Furthermore, certain features of an abstract tropical curve can be lost under a given choice of coordinates. For example, the na\"ive tropicalization of a genus two hyperelliptic plane curve induced by~\eqref{eq:hyperelliptic} is a graph $\Gamma\subset\RR^2$ with $b_1(\Gamma)=0$.

The connection to Berkovich non-Archimedean spaces~\cite{berkovichbook} initiated by Payne~\cite{analytifAndTrop} hands us a way to overcome this coordinate-dependency: a faithful tropicalization is the best candidate to reflect relevant geometric properties of the algebraic curve~\cite{bak.pay.rab:16}. An embedding $\mathcal{X}\subset (K^*)^n$ induces a \emph{faithful tropicalization} if $\Trop\,\mathcal{X}$ contains an isometric copy of the minimal Berkovich skeleton of $\mathcal{X}^{\an}$ under the tropicalization map $\trop\colon \mathcal{X}^{\an}\to \Trop\,\mathcal{X}$.
 The latter can be obtained from a given (extended) skeleton by contracting it to its minimal expression~\cite{BPRContempMath}.

  Just as in the abstract setting, faithful tropicalizations induced by $\mathcal{X}\subset (K^*)^n$ admit a tropical forgetful map to $M_g^{\trop}$, where $g$ is the arithmetic genus of $\mathcal{X}$. In order to do so, we must endow the rational weighted balanced graph $\Gamma=\Trop\,\mathcal{X}$ in $\RR^n$ with  a weight function on its vertices. This can be achieved by means of an extended Berkovich skeleton $\Sigma(\mathcal{X})$ coming from a semistable model of $\mathcal{X}$ with a horizontal divisor (i.e.\ the closure of a divisor of the generic fiber in the model) that is compatible with $\Gamma$~\cite{gub.rab.wer:16,gub.rab.wer:15}. Indeed, to each vertex $v$ in $\Gamma$ we assign the sum of the genera of all semistable vertices of $\Sigma(\mathcal{X})$ mapping to $v$ under $\trop\colon \Sigma(\mathcal{X})\to \Trop\,\mathcal{X}$.  The semistable vertices correspond exactly to the components of the central fiber~\cite{BPRContempMath}, so we weigh them with the genus of the associated component. 

  For planar tropicalizations, a similar ad-hoc rule can be put in practice. If we let $\Gamma$ be the dual complex of the Newton subdivision of the corresponding curve, each vertex of $\Gamma$ gets assigned the number of interior lattice points of its dual polygon. This quantity is the genus of the initial degeneration of the curve induced by the vertex  minus the number of nodes (assuming it is nodal). However, unless our planar embedding is faithful (which only occurs for Types (V) and (VII)), we will not be able to define a forgetful map on the tropical side (by collapsing all legs and weight zero one-valent vertices, as we did in the abstract case) that recovers the image of the Berkovich skeleton under tropicalization.

In the algebraic setting, the forgetful map sending planar genus two smooth hyperelliptic curves  to points in ${M_2}(K)$ is surjective if we allow the curves to be defined over valued field extensions $L|K$. Since the forgetful map on the associated tropical plane curves is only defined for Types (V) and (VII), faithfulness becomes an essential property to define the  left square in~\autoref{fig:diagram}. A similar behavior  in genus three and four was encountered by Brodsky-Joswig-Morrison-Sturmfels~\cite[Theorems 5.1 and 7.1]{PlaneModuli2015}.
\autoref{sec:faithful-reemb} and~\autoref{tab:nonFaithfulness} give explicit effective methods for producing \emph{faithful re-embeddings} of smooth planar genus two curves in a suitable torus. The main technique involved is tropical modifications of $\RR^n$ along tropical divisors~\cite{BLdM11,IMS09,MikhalkinICM}, which we now recall.

\begin{definition}\label{def:modif}
Fix a tropical polynomial $F$ defining a piecewise linear function 
\begin{equation*}\label{eq:tropPoly}
  F\colon \RR^n \to \TP=\RR \cup \{-\infty\} \quad F(\underline{X})=\max_{\beta\in \ZZ_{\geq 0}^n} \{C_{\beta} + \beta_1X_1 + \ldots +\beta_n X_n\} \text{ in }\TP[X_1,\ldots, X_n].
\end{equation*}
 The graph of $F$ is a rational polyhedral complex of pure dimension $n$. Unless $F$ is linear, the bend locus of $F$ has codimension 1. At each break codimension-one cell $\sigma$, we attach a new cell $\widetilde{\sigma}$ spanned by $\sigma$ and  $-e_{n+1}:=(0,\ldots, 0, -1)$. The result is a pure rational polyhedral complex in $\RR^{n+1}$. We call it the \emph{tropical modification of $\RR^n$ along $F$}.
\end{definition}
It will often be useful to consider \emph{polynomial lifts} of $F$, namely
 \begin{equation}\label{eq:liftf}
   f(\underline{x}) = \sum_{\beta \in \support(F)} c_{\beta} \underline{x}^{\beta}\in K[x_1,\ldots, x_n] \qquad
   \text{ where }\quad \support(F):=\{\beta \colon C_{\beta}\neq -\infty\}
 \end{equation}
 satisfies $\trop(f)(\underline{X}) := \max_{\beta} \{-\val(c_{\beta}) +\beta_1X_1 +\ldots + \beta_nX_n\} \!=\! F(\underline{X})$ as functions on $\RR^n$.

 By the Structure Theorem~\cite[Proposition 3.1.6]{MSBook}, any polynomial lift $f$ of $F$ will allow us to turn the tropical modification of $\RR^n$ along $F$ into a weighted balanced complex, since it will be supported on the tropical hypersurface $\Trop\,V(f)$.  In turn, any tropical hypersurface $\Trop\,V(g)$ in $\RR^n$ can be modified along $F$ in a similar fashion and the attached cells can be endowed with suitable multiplicities to turn the resulting complex into a balanced one. For precise multiplicity formulas, we refer to~\cite[Construction 3.3]{AllermanRau}.
\begin{example}
  The leftmost map in~\autoref{fig:modifypoints} describes the tropical modification of $\RR$ along the tropical function $F=\max\{X,-\val(\alpha_2)\}=\trop(x-\alpha_2)$. The result is a tropical line in $\RR^2$ with vertex $(-\val(\alpha_2), -\val(\alpha_2))$. All its tropical multiplicities equal 1. A higher dimensional instance can be found in Example \ref{ex:thetaC}.
\end{example}

Tropical modifications can be used to define re-embeddings of irreducible plane curves $\mathcal{X}$~\cite{BLdM11,cue.mar:16,IMS09}. This technique is also known as tropical refinement in parts of the literature. Consider a tropical polynomial $F\in \TP[X,Y]$ and a lift $f$. Given a defining equation $g(x,y)$ for $\mathcal{X}$, the tropicalization of the ideal
\begin{equation}\label{eq:Igf}
  I_{g,f}:=\langle g, z-f\rangle\subset K[x^{\pm},y^{\pm},z^{\pm}]
\end{equation}
is a tropical curve in the modification of $\RR^2$ along $F$. For almost all lifts $f$, $\Trop\,V(I_{g,f})$ coincides with the modification of $\Trop\,V(g)$ along $F$, i.e.\ we only bend $\Trop\,V(g)$ so that it fits the graph of $F$ and attach suitable weighted downward legs. However, for some special choices of lifts $f$, the cells of $\Trop\,V(I_{g,f})$ in the downward cells of the modification of $\RR^n$ along $F$ become more interesting. Such choices are determined by the  initial degenerations of $g$ along the bend locus of $F$. More details can be found in~\autoref{sec:faithful-reemb}.
\smallskip

In addition to linear tropical polynomials, which were the main players in~\cite{cue.mar:16}, our main focus in~\autoref{sec:faithful-reemb} will be  modifications of $\RR^2$ along tropical polynomials of the form
\begin{equation}\label{eq:F}
  F=\max\{Y , A + X , B +2X\} = \trop(f)\qquad \text{ for }A, B\in \RR.
\end{equation}
The tropical surface $\Trop\,V(f)$ consists of six two-dimensional cells $\sigma_1,\ldots, \sigma_6$, as depicted in~\autoref{fig:ThetaModif}. They are defined by the following systems of linear equations and inequalities:
  \begin{equation}
  \label{eq:cellsModif}
   \begin{minipage}{0.45\textwidth}
 \[\begin{aligned}
    \sigma_1:= &\{Z = X+A\geq Y, X\leq A-B\}, \\
    \sigma_2:=&\{Z=2X+B\geq Y, X\geq A-B\},\\
    \sigma_3:=&\{Z=Y\geq X+A, 2X+B\},
  \end{aligned}\]
    \end{minipage}
   \begin{minipage}{0.45\textwidth}
 \[\begin{aligned}
   \sigma_4:=& \{Z,Y \leq 2A-B, X=A-B\}, \\
  \sigma_5:=&\{Y=2X+B\geq Z, X\geq A-B\},\\
  \sigma_6:=&\{Y=X+A\geq Z, X\leq A-B\}.
\end{aligned}\]
        \end{minipage}
    \end{equation}
Just as it happened in the linear case~\cite[Lemma 2.2]{cue.mar:16}, the choice of $F$ in~\eqref{eq:F} allows us to  recover $\Trop\, V(I_{g,f})$ in $\RR^3$ from the three coordinate projections.  This property will be exploited in~\autoref{sec:faithful-reemb} to certify faithfulness by planar computations.
\begin{lemma}\label{lm:projections}
Given an irreducible curve $\mathcal{X}\subset (K^*)^2$ defined by a polynomial $g\in K[x,y]$ and a polynomial lift $f(x,y)=y-a x- bx^2 \in K[x,y]$ of the tropical polynomial $F$ from \eqref{eq:F}, the tropicalization induced by the ideal $I_{g,f}=\langle g, z-f\rangle\subset K[x^{\pm}, y^{\pm}, z^{\pm}]$ is completely  determined by the tropical plane curves $\Trop\,V(g)$, $\Trop\, V(I_{g,f}\cap K[x^{\pm},z^{\pm}])$, and $\Trop\, V(I_{g,f}\cap K[y^{\pm},z^{\pm}])$.  \end{lemma}
\begin{proof} Since coordinate projections are monomial maps, functoriality ensures that the three coordinate projections of $\Trop\, V(I_{g,f})$ are supported on the three tropical plane curves in the statement. The tropical space curve is completely determined by its intersection with the relative interiors of the six maximal cells of $\Trop\,V(f)$. By construction, each open cell $\sigma_i^\circ$ maps to a  two-dimensional open region under two out of the three projections. The precise choices are indicated on~\autoref{fig:ThetaModif}. Note that overlaps occur only in the $YZ$-projection between two pairs of cells: $(\sigma_1,\sigma_4)$ and $(\sigma_4,\sigma_6)$. 
  
The tropical multiplicities in all coordinate projections let us recover the support of $\Trop\,V(I_{g,f})$ along the bend locus
from the generalized push-forward formula for multiplicities of
Sturmfels--Tevelev in the non-constant coefficients
case~\cite[Corollary 7.3]{bak.pay.rab:16}.
\end{proof}
\begin{example}\label{ex:thetaC}
  Consider the smooth genus two curve in $(K^*)^2$ defined over $\PS$ by  
  \[g(x,y)=y^2 - x(x-(3 t^5)^2)(x-(11 t^2 + 5 t^7)^2)(x-(11 t^2)^2)(x+(1+t^2)^2),
  \]
  the tropical polynomial $F=\max\{Y,-4+X, 2X\}$ and its lift $f(x,y)=y-(1+t^2)(11t^2+5t^7)(11t^2)\, x+(1+t^2)\,x^2$. The tropicalization induced by $I_{g,f}\subset K[x^{\pm}, y^{\pm},z^{\pm}]$ is depicted in the left of~\autoref{fig:ThetaModif} and it lies in the tropical surface in $\RR^3$ obtained by modifying $\RR^2$ along $F$. We reconstruct the tropical curve from the three coordinate projections shown on the right of the picture, accounting for additivity of multiplicities and the two \emph{false crossings} on the $YZ$-projection. The na\"ive plane tropicalization agrees with the $XY$-projection. The Berkovich skeleton is a theta graph. For further details we refer to~\autoref{ex:ExTypeII}.
\end{example}
\begin{figure}
  \centering
  \includegraphics[scale=0.37]{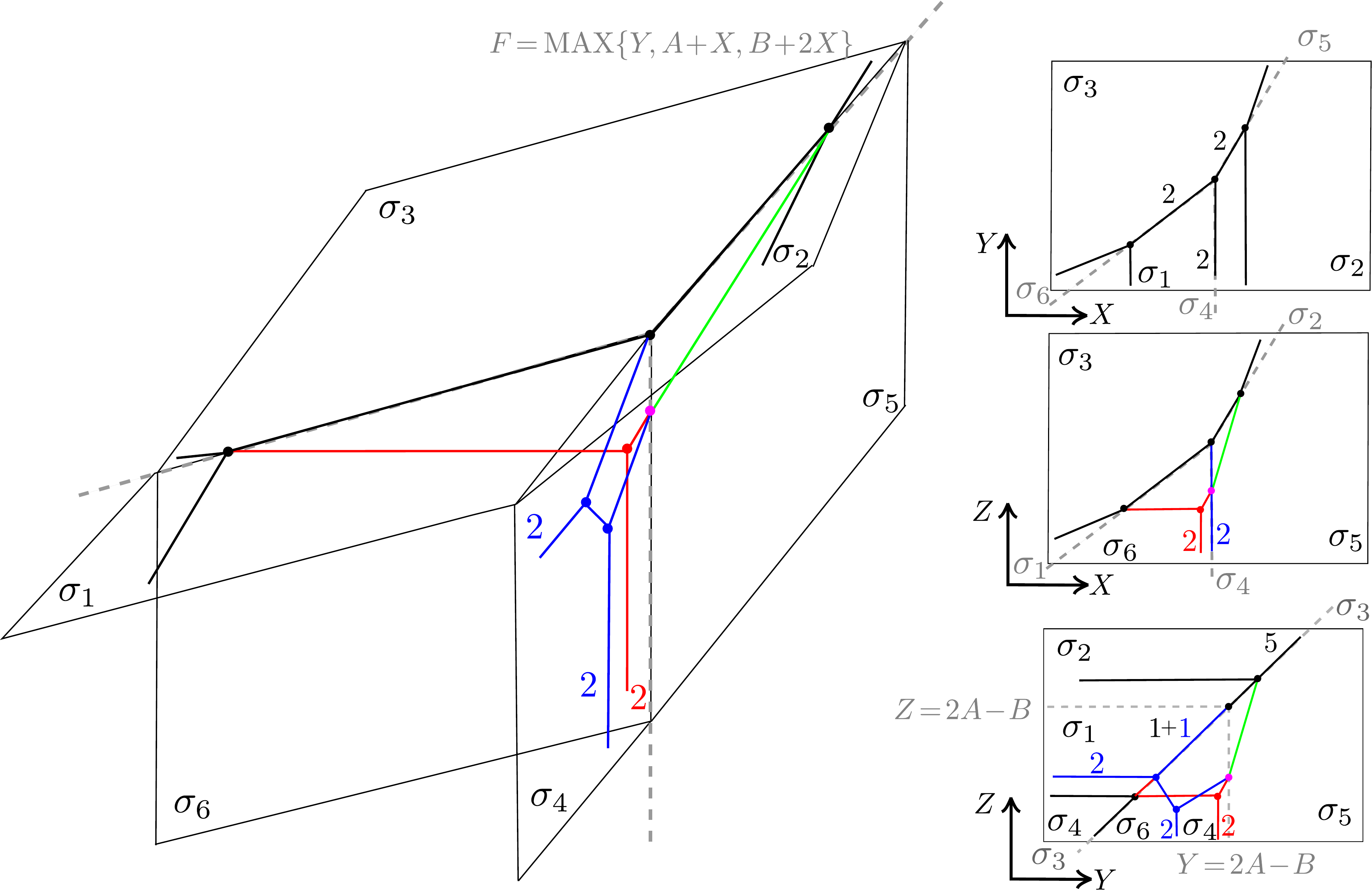}
  \caption{A tropical modification of $\RR^2$  and its  coordinate projections.\label{fig:ThetaModif}}
\end{figure}

  \section{Tropical hyperelliptic covers of metric trees}\label{sec:trop-hyper-covers}
  Algebraic genus two curves are hyperelliptic and hence can be realized as the source curve of a 2-to-1 cover of the projective line branched at six points. The analogous results for tropical hyperelliptic genus $g$ curves and metric trees with $n=2g+2$ legs and genus zero vertices was first established by Baker-Norine~\cite{bak.nor:09} and Chan~\cite{cha:13}, and later generalized to admissible covers and harmonic morphisms by Caporaso~\cite{cap:14} and Cavalieri-Markwig-Ranganathan~\cite{cav.mar.ran:16}. We restrict the exposition to our case of interest.
  \begin{definition}A map $\pi\colon \Gamma \to \Gamma'$ is a morphism of metric graphs if $\pi$ sends the vertices of $\Gamma$ to vertices of $\Gamma'$, and the edges (respectively, legs) of $\Gamma$ to edges (respectively, legs) of $\Gamma'$ in a piecewise fashion with integral slopes.
  \end{definition}
  \begin{remark}\label{rm:pLmetricMorphism} Assume the morphism $\pi$ sends an edge $e$ of $\Gamma$ with length $\ell(e)$ onto an edge $e'$ of $\Gamma'$ of length $\ell(e')$. We may write the map $\pi_{|_e}$ as $h\colon [0,\ell(e)]\twoheadrightarrow [0,\ell(e')]$ with $h(t)=w(e)t$ for some $w(e)\in \ZZ_{>0}$. By construction, $w(e)=\ell(e')/\ell(e)$.
Similarly, the map $\pi$ restricted to a leg $e$ of $\Gamma$ equals $h\colon [0,\infty)\twoheadrightarrow [0, \infty)$ with $h(t)=w(e)t$ for some $w(e)\in \ZZ_{>0}$. 
  \end{remark}

  \begin{definition} A map $\pi\colon \Gamma \to \Gamma'$ of metric graphs is \emph{harmonic} if for each vertex $v$ of $\Gamma$ and any edge $e'$ adjacent to $\pi(v)$, the number
    \begin{equation}\label{eq:localDegree}
    d_v:=\sum_{\substack{e\in E(\Gamma)\\
        v\in e,\pi(e)=e'}} w(e)
    \end{equation}
    does not depend on the choice of edge $e'$. We call $d_v$ the \emph{local degree} of the map $\pi$ at $v$. The \emph{degree} of  $\pi$ is the sum over all local degrees in the fiber of any vertex $v'\in \Gamma'$.
  \end{definition}
 
  \begin{definition} A \emph{tropical hyperelliptic cover} of a metric tree $T$  by a metric graph $\Gamma$ is a surjective degree two harmonic map $\pi\colon \Gamma \to T$ of metric graphs satisfying the local Riemann-Hurwitz conditions at each $v$ vertex of $\Gamma$:
  \begin{equation}\label{eq:RHGenus0}
    2-2g(v) = 2d_v - \#\{e\ni v: \ww(e)=2\}.
  \end{equation}
  \end{definition}

  \begin{definition} A \emph{branch point} of a hyperelliptic cover $\pi\colon \Gamma \to T$ of a genus zero metric tree $T$  is a leg or edge of $T$ which is covered by a leg or edge $e$ of $\Gamma$  with weight $w(e)=2$.
    \end{definition}

Since we are interested in metric graphs $\Gamma$ of genus two, we are restricted to covers of trees $T$ with precisely six leaves. Each vertex of $T$ has valency between three and six. The following technical lemma describes the local behavior of a hyperelliptic cover  $\Gamma \to T$.

\begin{lemma}\label{lm:localCovers} There are precisely five tropical hyperelliptic covers of a single genus zero vertex with valency between three and six  with source curve a vertex of genus at most two.
\end{lemma}
\begin{proof}

  We let $v'$ be the vertex in the target curve and fix a covering vertex $v$ on  the source curve. The result follows by analyzing all possible combinations of genus $g(v)$ and valency of $v'$.   
  Replacing  each  value of $g_v=0,1$, or $2$ in~\eqref{eq:RHGenus0} yields all cases in~\autoref{fig:localCovers}.
\end{proof}
\begin{figure}[htb]
  \centering
   \includegraphics[scale=0.3]{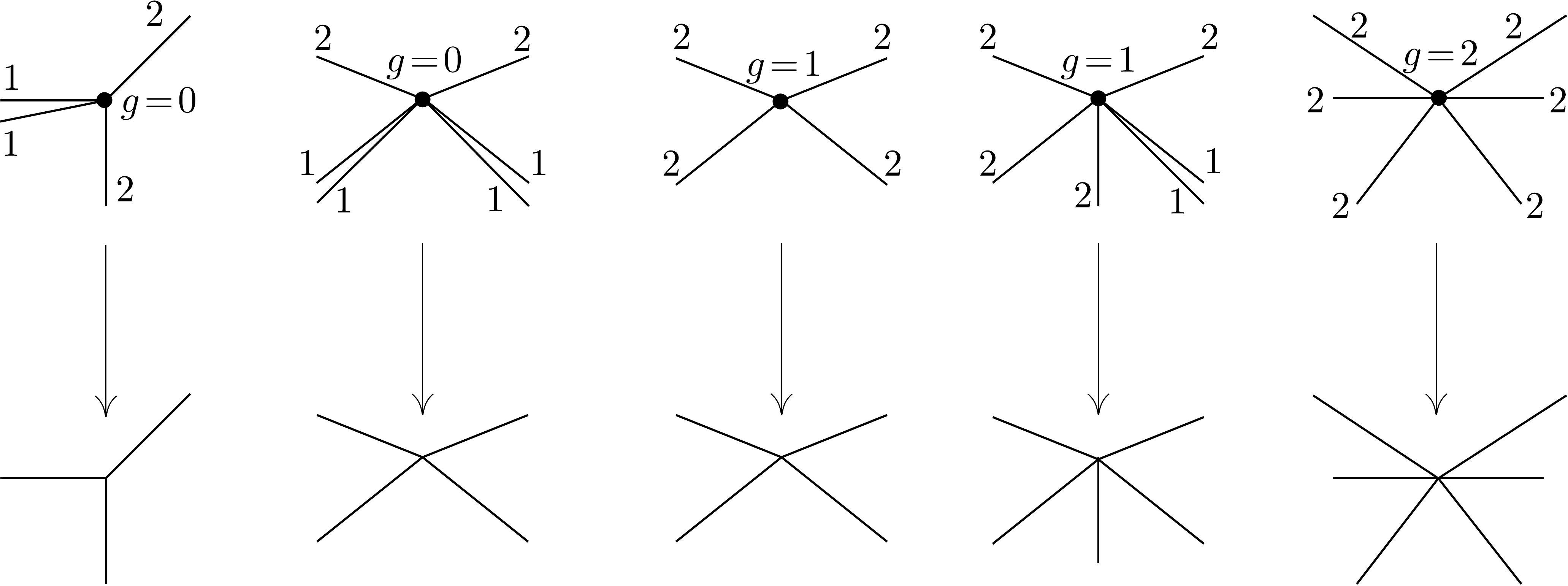}
  \caption{All possible degree two covers of a single genus zero vertex with valency between three and six by a single vertex of genus up to two.}\label{fig:localCovers}
\end{figure}

Our main result in this section describes the combinatorics of hyperelliptic covers of trees on six leaves. It implies that the poset structures on $M_{0,6}^{\trop}$ and $M_2^{\trop}$ agree, as shown in~\cite[Theorem 5.3]{ren.sam.stu:14}. Unlike the latter, our proof is elementary and uses the local tropical Riemann-Hurwitz conditions~\eqref{eq:RHGenus0}. The general hyperelliptic case is treated in~\cite[Lemma 2.4]{bobrch:17}. Superhyperelliptic curves are discussed in~\cite{brhe:17}:
\begin{proposition}\label{pr:CombCovering}
  Each tree on six leaves is covered by exactly one genus two graph with six legs  via a harmonic 2-to-1 map branched at all six leaf edges as in~\autoref{fig:M2BarAndMumford}. 
\end{proposition}
\begin{proof}
  The leaf edges on the trees are branch points, hence they must be covered by legs of weight two. \autoref{lm:localCovers} characterizes the local behavior at each vertex of the tree.
These two facts uniquely determine the combinatorial type of the graph and the cover itself.
\end{proof}

\begin{remark}\label{rm:edgeLengths} Following~\autoref{rm:pLmetricMorphism}, the  length of an edge $e$ in $\Gamma$ covering an edge $e'$ in $\Gamma'$ satisfies $ \ell(e) = \ell(e')/\ww(e)$. 
In particular, when two weight-one edges in $\Gamma$ form a loop that covers a single edge $e'$ in $\Gamma'$, then the loop has length $2\,\ell(e')$.
\end{remark}

\section{The Classification Theorem and the diagonal map ${M_{0,6}}(K)\to M_2^{\ensuremath{\trop}}$}\label{sec:class-thm}
Throughout this section, we let $\alpha_1,\ldots, \alpha_6$ be six distinct points in $K^*$ defining an element of ${M_{0,6}}(K)$ via the six marking  $(1:\alpha_1), \ldots, (1:\alpha_6)$ in $\PP^1$. We consider the diagonal map
\begin{equation}\label{eq:diagVarphi}
  \varphi\colon {M_{0,6}}(K)\to M_2^{\trop}
\end{equation}
from~\autoref{fig:diagram} sending a smooth rational curve $\mathcal{X} \in {M_{0,6}}(K)$ to the minimal Berkovich skeleton $\varphi(\mathcal{X})$ of the unique hyperelliptic
curve covering $\mathcal{X}$ with branching at $(1:\alpha_1), \ldots, (1:\alpha_6)$, as in \autoref{fig:M2BarAndMumford}.  This map is well-defined since it only depends on the
equivalence class of $\underline{\alpha}:=(\alpha_1, \ldots, \alpha_6)$ in
$(K^*)^6$ up to isomorphism. 
Combining~\autoref{tab:CombAndLengthData} with Algorithms~\ref{alg:ineq} and~\ref{alg:sep} will completely determine $\varphi$. Furthermore, this characterization depends solely on  the relative order of the negative valuations of the entries of $\underline{\alpha}$ and some of their differences, as in~\eqref{eq:deWAndDs}. 
As discussed in \autoref{rem:algf}, results in this section can be used to take an arbitrary genus two curve given by a hyperelliptic equation to one of the seven forms corresponding to the seven cones in $M_2^{\trop}$.

Since $K = \overline{K}$ is non-trivially valued by assumption, it follows that the valued group of $K$ is dense in $\RR$~\cite[Lemma 2.1.12]{MSBook}. As a consequence, we can construct a splitting of the valuation map~\cite[Lemma 2.1.15]{MSBook}. Inspired by the canonical splitting for the Puiseux series field, we write it as $\gamma\mapsto t^{\gamma}$. We use this notion to define initial forms in  $K^*$:
\begin{definition}\label{def:init_t}
  Given  a splitting $\gamma\mapsto t^{\gamma}$ of the valuation on $K$, we define the \emph{initial form}  $\init(\alpha)$ of any $\alpha\in K^*$  as the class of $\alpha \,t^{-\val(\alpha)}$ in the residue field $\resK$ of $K$ obtained as the quotient of the valuation ring by its maximal ideal.
\end{definition}

We let $\underline{\omega}:= (\ww_1,\ldots, \ww_6)\in \RR^6$ be the  weight vector from \eqref{eq:valn4Param} associated to $\underline{\alpha}$ and assume $\ww_1\leq \ww_2\leq \ldots\leq \ww_6$.
Whenever there is a tie between $\ww_i$ and $\ww_{i+1}$ and the corresponding initial forms of $\alpha_i$ and $\alpha_{i+1}$ agree, we consider the valuation of the difference $\alpha_i-\alpha_{i+1}$ and notice that $d_{i,i+1}:=-\val(\alpha_i-\alpha_{i+1}) <\ww_i=\ww_{i+1}$  if  $\init(\alpha_i)=\init(\alpha_{i+1})$. In this situation, we replace the $(i+1)$-st. entry of $\underline{\omega}$ by $d_{i,i+1}$.

\smallskip
As a first step towards a complete classification of the image of $\varphi$ and its domains of linearity, we construct
seven regions in the space of branch points whose
associated trees have different combinatorial types:
\begin{equation}\label{eq:setsOmega}
  \Omega^{(i)} := \{\underline{\alpha} \in {M_{0,6}}(K)\colon \text{weight }\underline{\omega} \in \RR^6\text{ satisfies conditions (i) in~\autoref{tab:CombAndLengthData}}\}, 
\end{equation}
for $i\in \{\text{I,} \ldots, \text{VII}\}$. 
Even though these sets do not cover all tuples of distinct points in $(K^*)^6$ we show that they parameterize all seven cones in $M_2^{\trop}$ and the harmonic maps from the metric graphs in $M_2^{\trop}$ to $M_{0,6}^{\trop}$ given in~\autoref{fig:M2BarAndMumford}. Here is the precise statement:

\begin{proposition}\label{pr:classificiationRefTypes}
  For each $i\in \{$\emph{I,}$\ldots$,\emph{VII}$\}$, the diagonal map $\varphi$ from~\eqref{eq:diagVarphi} restricted to $\Omega^{(i)}$ parameterizes the cone  of Type (i) in $M_2^{\trop}$ and induces a hyperelliptic cover of a tree in $M_{0,6}^{\trop}$ by an abstract tropical curve of Type (i) in $M_2^{\trop}$.
   Furthermore, the metrics on both objects are completely determined by piecewise functions on the weight vectors $\underline{\omega}$ of points in each $ \Omega^{(i)}$  as in the second and fourth column of~\autoref{tab:CombAndLengthData}.
\end{proposition}

\begin{proof} Starting from a tuple $\underline{\alpha}\in \Omega^{(i)}$ viewed as a marking on $\PP^1$, we consider the smooth rational curve $\mathcal{X}$ in ${M_{0,6}}$  and the associated weight vector $\underline{\omega}\in \RR^6$. 
  Our goal is to determine the combinatorial type of the tree  $\Trop\, \mathcal{X}$ and to express its metric structure in terms of $\underline{\omega}$. We do so by analyzing each of the seven sets $\Omega^{(i)}$ separately. By~\autoref{pr:CombCovering} we can label each tree by the type of the genus two metric graph $\Gamma$ covering it.  The edge length formulas on $\Gamma$ indicated on the last column of~\autoref{tab:CombAndLengthData}  are obtained directly from the metric structure on each tree using~\autoref{rm:edgeLengths}. It is important to emphasize that the tropical Pl\"ucker map will give the half-distance vector on the tree, as we saw in~\autoref{sec:trop-moduli-spaces}.

In what remains, we discuss the second column of the table.  The combinatorial type of each tree is determined by the isomorphism $\Phi\colon M_{0,6}\stackrel{\simeq}\longrightarrow \Gr_0(2,6)/(K^*)^6\subset (K^*)^{15}/(K^*)^6$ from~\eqref{eq:PlEmb} and the four-point conditions (i.e., the tropical 3-term Pl\"ucker relations \cite[Lemma 4.3.6]{MSBook}) on $-\val(\Phi(\underline{\alpha}))\in \RR^{15}/\RR^6$.  We use the lexicographic order on $\RR^{15}$.
  
  \vspace{1ex}
  
  \noindent
  \textbf{Type (I):} We claim  $\Trop\,\mathcal{X}$ is a trivalent caterpillar tree on six leaves with internal edge lengths $\ww_3-\ww_2$, $\ww_4-\ww_3$ and $\ww_5-\ww_4$. Indeed, since $-\val(\alpha_i-\alpha_j)=\ww_j$ for $i<j$ we have
  \[
-\val(\Phi(\underline{\alpha}))\!:=  (\ww_2,\ww_3,\ww_4,\ww_5,\ww_6, \ww_3, \ww_4,\ww_5, \ww_6, \ww_4, \ww_5, \ww_6,\ww_5, \ww_6, \ww_6)\in \Trop\,\Gr_0(2,6)/\RR^6.
  \]
By construction, the half-distance vector equals $-\val(\Phi(\underline{\alpha}))$.  The four-point condition implies that the corresponding line in $\PP^5$  is a trivalent caterpillar tree. Linear algebra recovers the expected lengths on its three bounded edges \cite[Remark 4.3.7]{MSBook}.   Note that the lengths assigned to the six legs in the second column of~\autoref{tab:CombAndLengthData} play no role here: the associated half-distance vector in $\RR^{15}$ is in the same class modulo the lineality space in $ \Trop\,\Gr_0(2,6)$. The claim follows.

    \vspace{1ex}

  \noindent
  \textbf{Type (II):} 
By construction,  $\Phi(\underline{\alpha})$ has negative valuation vector
  \[-\val(\Phi(\underline{\alpha}))\!:=  (\ww_2,\ww_3,\ww_4,\ww_5,\ww_6, \ww_3, \ww_4,\ww_5, \ww_6, d_{34}, \ww_5, \ww_6,\ww_5, \ww_6, \ww_6)\in \Trop\,\Gr_0(2,6)/ \RR^{6},\]
  where the $\ww_i$ and $d_{34}$ are as in \eqref{eq:valn4Param}.
  The four--point conditions  imply that the tropical line in $\PP^5$ is a snowflake tree with internal edges $\ww_3-\ww_2$, $\ww_5-\ww_3$ and $\ww_3-d_{34}$, as indicated on the second column of the table.

    \vspace{1ex}

    \noindent
    \textbf{Types (III) through (VII):} The tropicalization induced by the Pl\"ucker embedding  shows that the metric trees on these lower-dimensional cells of $M_{0,6}^{\trop}$ are obtained by specializing the trees for Type (I) or Type (II): both the combinatorial type and the metric are obtained by coarsening either the caterpillar or the snowflake trees.      The edge length formulas match those given in~\autoref{tab:CombAndLengthData}.
\end{proof}

  \begin{table}
    \centering
{\renewcommand{\arraystretch}{0.35}}
  \begin{tabular}{|c|c|c|c|}
    \hline
    \small{Type}\normalsize& Cover with lengths on $M_{0,6}^{\trop}$ 
    & Defining conditions & Lengths on $M_2^{\trop}$ \\ 
    \hline
    \hline 
                        \multirow{6}{*}{(I)} &  \multirow{7}{*}{\scalebox{0.75}{\includegraphics[scale=0.5]{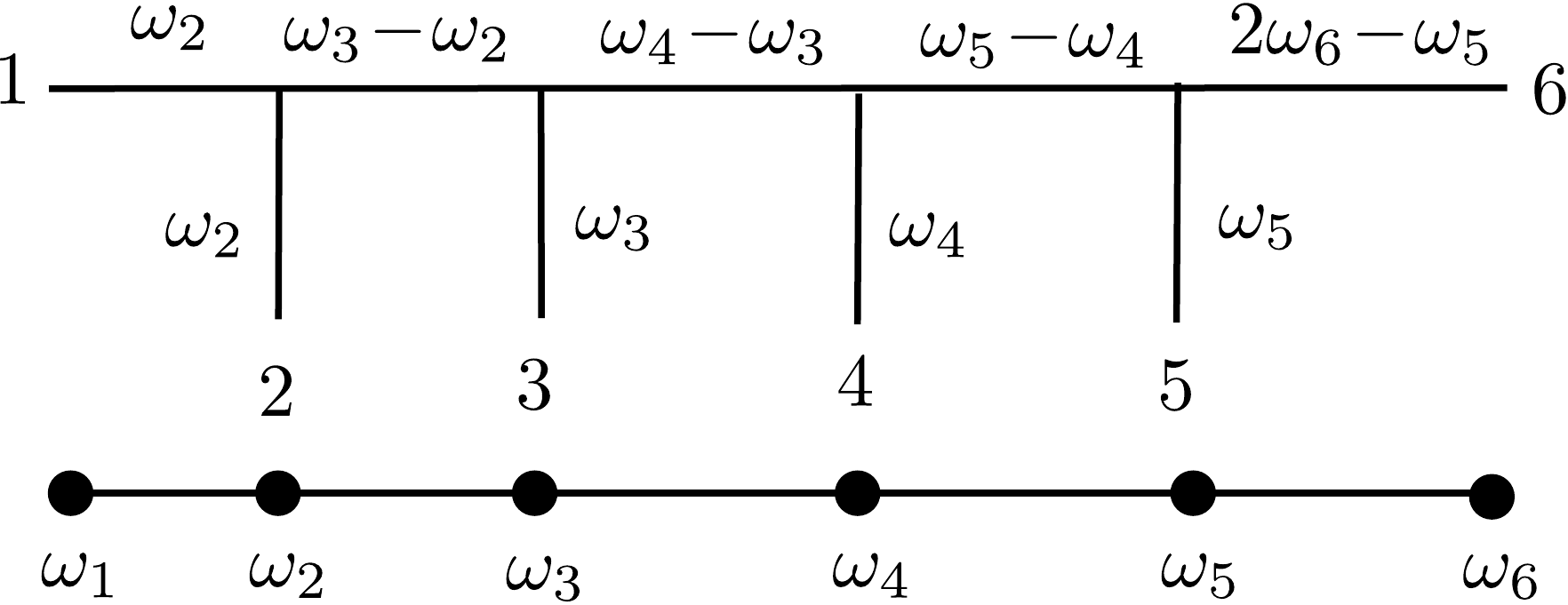}}} & \multirow{6}{*}{$\ww_1\!<\!\ww_2\!<\!\ww_3\!<\!\ww_4\!<\!\ww_5\!<\!\ww_6$}& \multirow{2}{*}{$L_0=(\ww_4-\ww_3)/2$}\\ &
      & & \\ & & & \multirow{2}{*}{$L_1=2(\ww_5-\ww_4)$\ }\\  & & & \\     & & & \multirow{2}{*}{$L_2=2(\ww_3-\ww_2)$\ }\\ & & & \\  & & & \\ 

             \hline
        \multirow{7}{*}{(II)} &  \multirow{7}{*}{\scalebox{0.75}{\includegraphics[scale=0.5]{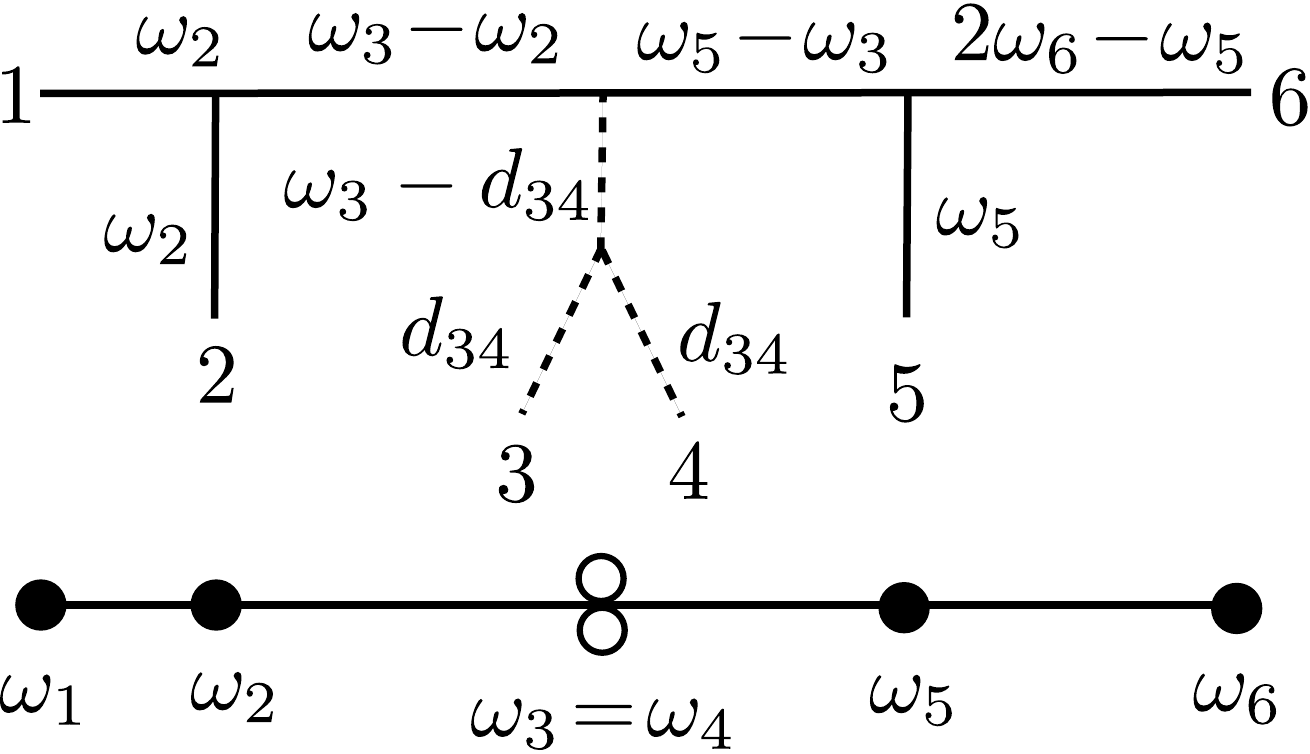}}} & \multirow{2}{*}{$\ww_1<\ww_2<\ww_3<\ww_5<\ww_6$}&\multirow{2}{*}{$L_0=2(\ww_3-d_{34})$}\\ &
      & & \\ & &  \multirow{2}{*}{$\ww_3 = \ww_4$}& \multirow{2}{*}{$L_1=2(\ww_5-\ww_3)$\ }\\  & & & \\     & & \multirow{2}{*}{$\init(\alpha_3) = \init(\alpha_4)$}& \multirow{2}{*}{$L_2=2(\ww_3-\ww_2)$\ }\\ & & & \\  & & & \\ 

        \hline
        \multirow{5}{*}{(III)} &  \multirow{6}{*}{\scalebox{0.75}{\includegraphics[scale=0.5]{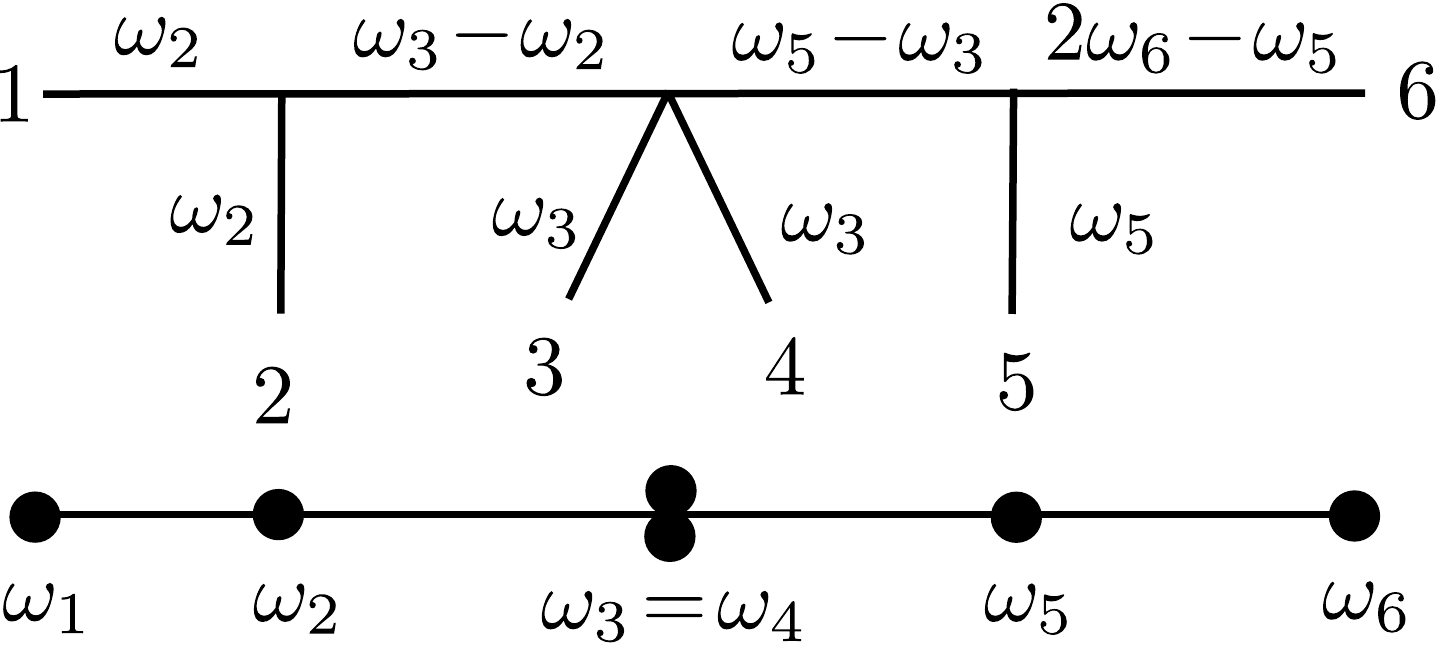}}} & \multirow{2}{*}{$\ww_1<\ww_2<\ww_4<\ww_5<\ww_6$}&\multirow{2}{*}{$L_0= 0$\ \qquad\qquad}\\ &
      & & \\ & &  \multirow{2}{*}{$\ww_3 = \ww_4$}& \multirow{2}{*}{$L_1=2(\ww_5-\ww_3)$\ }\\  & & & \\     & & \multirow{2}{*}{$\init(\alpha_3) \neq \init(\alpha_4)$}& \multirow{2}{*}{$L_2=2(\ww_3-\ww_2)$\ }\\ & & & \\ 
    \hline
        \multirow{5}{*}{(IV)} &  \multirow{6}{*}{\scalebox{0.75}{\includegraphics[scale=0.5]{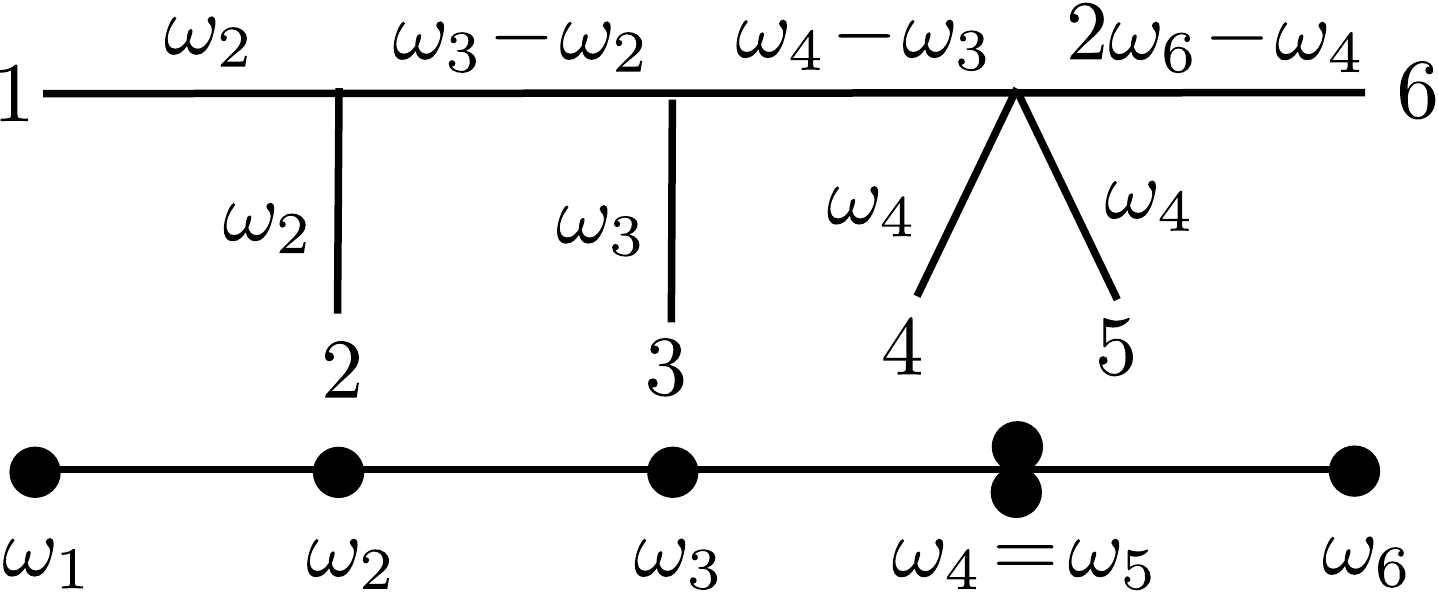}}} & \multirow{2}{*}{$\ww_1<\ww_2<\ww_3<\ww_4<\ww_6$}&\multirow{2}{*}{$L_0=(\ww_4-\ww_3)/2$}\\ &
      & & \\ & & \multirow{2}{*}{$\ww_4 = \ww_5$}& \multirow{2}{*}{$L_1= 0$\ \qquad\qquad}\\  & & & \\     & & \multirow{2}{*}{$\init(\alpha_4) \neq \init(\alpha_5)$}& \multirow{2}{*}{$L_2=2(\ww_3-\ww_2)$\ }\\ & & & \\
        \hline 
        \multirow{5}{*}{(V)} &  \multirow{6}{*}{\scalebox{0.75}{\includegraphics[scale=0.5]{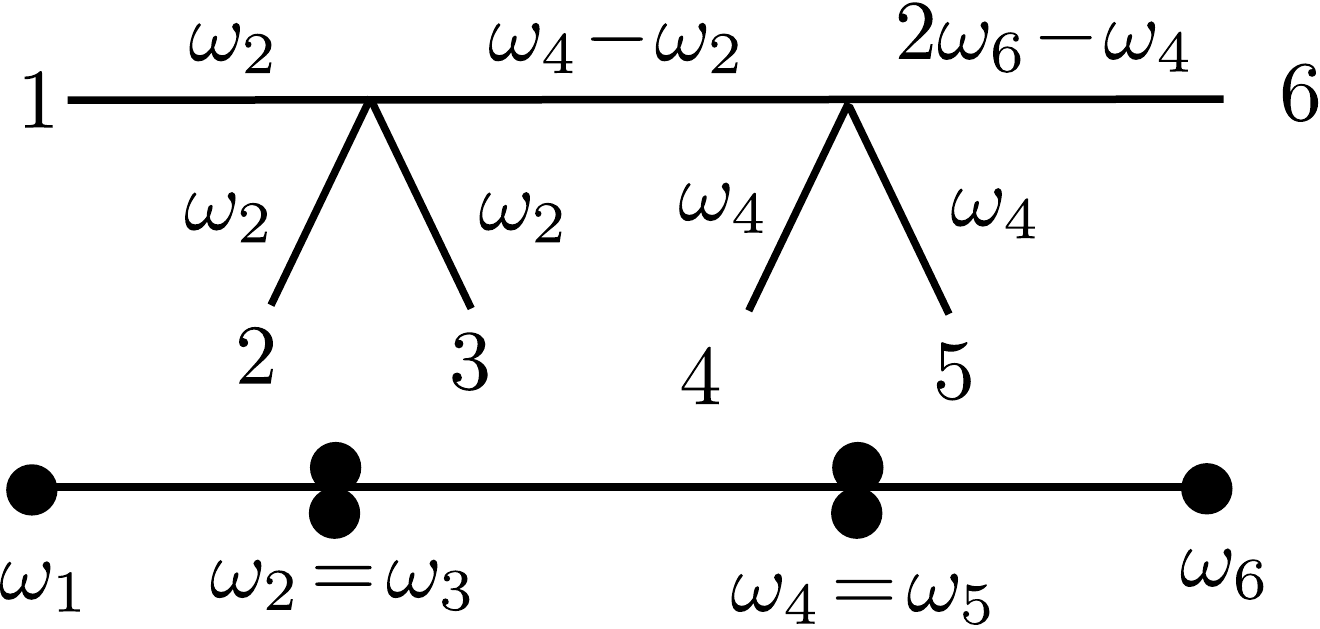}}} & \multirow{2}{*}{$\ww_1<\ww_2<\ww_4<\ww_6$}&\multirow{2}{*}{$L_0=(\ww_4-\ww_2)/2$}\\ &
      & & \\ & & \multirow{2}{*}{$\ww_2 = \ww_3$\,,\, $\ww_4 = \ww_5$}&  \multirow{2}{*}{$L_1= 0$\ \qquad\qquad}\\  & & {}& \\  & & {$\init(\alpha_2) \neq \init(\alpha_3)$}&  \multirow{2}{*}{$L_2= 0$\ \qquad\qquad} \\  & & {$\init(\alpha_4) \neq \init(\alpha_5)$}& \\ 
    \hline 
    \multirow{7}{*}{(VI)}   &  \multirow{7}{*}{\scalebox{0.75}{\includegraphics[scale=0.5]{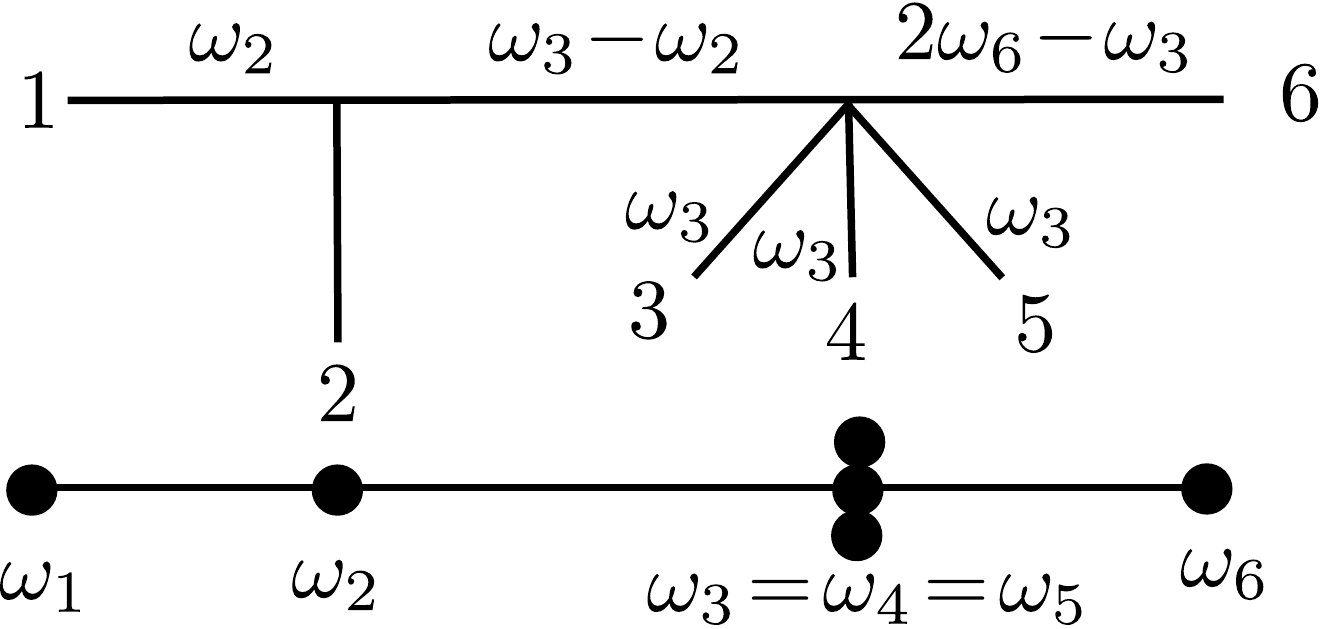}}} &  \multirow{2}{*}
             {$\ww_1<\ww_2<\ww_3<\ww_6$}& \multirow{2}{*}
             {$L_0=0$\ \qquad\qquad}\\ &
      & &\\ & & \multirow{2}{*}
                 {$\ww_3=\ww_4 = \ww_5$} & \multirow{2}{*}
                 {$L_1= 0$\ \qquad\qquad}
                 \\   & & &\\ & & {$\init(\alpha_3) \neq \init(\alpha_4)$}&  \multirow{3}{*}
                 {$L_2= 2(\ww_3-\ww_2)$} \\  & & {$\init(\alpha_3) \neq \init(\alpha_5)$}&  \\     & & {$\init(\alpha_4) \neq \init(\alpha_5)$}&  \\   
    \hline 
\multirow{6}{*}{(VII)} &  \multirow{7}{*}{ \scalebox{0.75}{\includegraphics[scale=0.5]{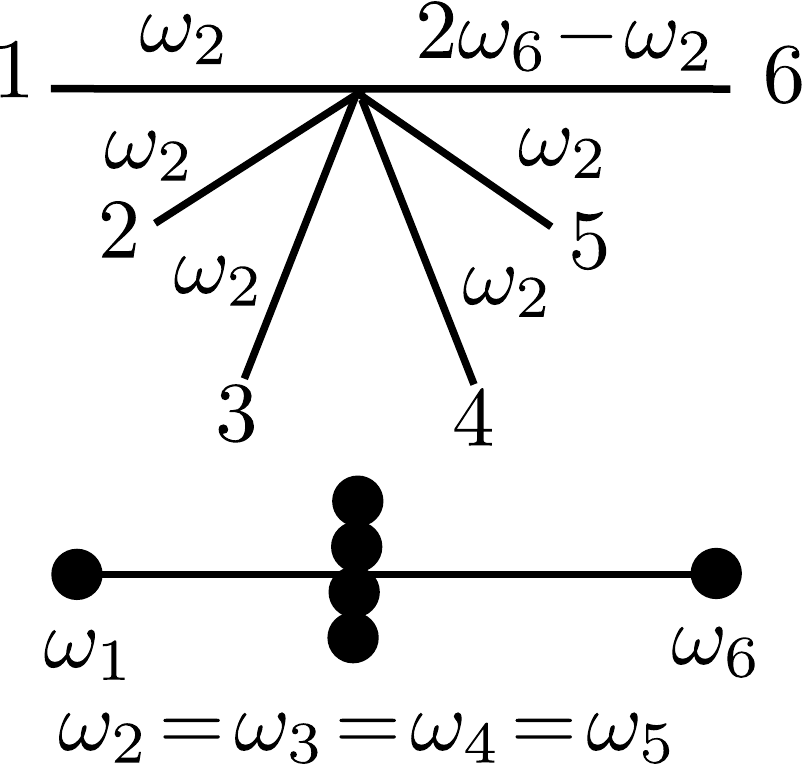}}} &  \multirow{2}{*}
             {$\ww_1<\ww_2<\ww_6$}& \multirow{2}{*}{$L_0= 0$\ \qquad\qquad}\\ & & & \\  & &  {$\ww_2 = \ww_3=\ww_4 = \ww_5$}& \multirow{2}{*}{$L_1= 0$\ \qquad\qquad}\\  & & & \\  & &  {$\init(\alpha_i) \neq \init(\alpha_j)$} & \multirow{2}{*}{$L_2= 0$\ \qquad\qquad}\\  & & {for $1<i<j<6$}& \\  & & & \\  
    \hline
  \end{tabular}
\caption{Combinatorial types  with the corresponding defining valuation conditions, and length data for  $M_{0,6}^{\trop}$ and $M_2^{\trop}$. Here, $\ww_i=-\val(\alpha_i)$, $d_{34}=-\val(\alpha_3-\alpha_4)$, and the edge lengths $L_0, L_1$ and $L_2$ refer to~\autoref{fig:allSkeletaAndMetrics}.}\label{tab:CombAndLengthData}  \end{table}

  In the remainder of this section we discuss why these seven regions $\Omega^{(i)}$ suffice to classify all smooth genus two tropical curves. Indeed, Algorithms~\ref{alg:ineq} and~\ref{alg:sep} describe an explicit combinatorial procedure that takes six distinct points $\alpha_1,\ldots, \alpha_6$ in $K^*$ and provides linear changes of coordinates in $\PP^1$ producing a tuple of points in one of the sets $\Omega^{(i)}$, after   iteratively combining two steps:
  \begin{description}
  \item [(A) \textbf{Separate points}] We take a coordinate $\ww_k$ of $\underline{\omega}$ and two points $\alpha_i$ and $\alpha_j$ of valuation $-\ww_k$ where $\val(\alpha_i-\alpha_j)$ is maximal, and make a linear change of coordinates that turns the tuple $\underline{\alpha}\in (K^*)^6$ into $\underline{\alpha}'\in (K^*)^6$, where  $-\val(\alpha_i')$ is the unique smallest element of $\ww'$. The method is described in~\autoref{lm:sepPts}. 
    
    \item [(B) \textbf{Turn around}] We change coordinates from one open affine chart of $\PP^1$ to another by replacing $x$ by $1/x$. As a result, $-\val(\alpha_i')=\val(\alpha_i)$ and the relative order of the valuations on the tuple $\underline{\alpha}$ is reversed on the new tuple $\underline{\alpha}'$.
  \end{description}

 As was mentioned earlier in this section, our assumptions on $K$ ensures the density of the value group of $K$ in $\RR$ and the existence of a splitting $\gamma\mapsto t^{\gamma}$ to the valuation. We use these to properties to separate branch points:
  \begin{lemma}\textbf{\emph{[Separating points]}}\label{lm:sepPts} Consider a repeated coordinate $\ww$ of $\underline{\omega}$, and write \[
    \beta=\max\{\val(\alpha_m-\alpha_l) \colon \ww_m=\ww_l=\ww \text{ for }m\neq l\}\geq -\ww.\]
    Fix two indices $i,j$ with $\ww_i\!=\!\ww_j\!=\!\ww$ and $\beta=\val(\alpha_i-\alpha_j)$. If $\init(\alpha_i-\alpha_j)=\overline{\zeta}\in \resK$ for some $\zeta$ with $\val(\zeta)=0$, choose $\gamma\in \val(K^*)$ with $\beta<\gamma<\val(\alpha_i-\alpha_j-\zeta t^{\beta})$. Then, the linear change of coordinates          $\psi\colon \PP^1\to\PP^1$ defined locally by 
    \begin{equation}\label{eq:sepPts}
 \psi(x) = x-\alpha_j-\zeta t^{\beta}-t^{\gamma}
    \end{equation}
      turns the tuple $\underline{\alpha}\in (K^*)^6$ into $\underline{\alpha}'\in (K^*)^6$, where their coordinatewise negative valuations $\underline{\omega}$ and $\underline{\omega}'$ satisfy the following properties:
\begin{enumerate}[(1)]
\item $\ww_s' =\ww_s > \ww_i  \;\text{ if } \ww_s > \ww_i$;
  \item  $\ww_s'=\ww_i$, and $\init(\alpha'_s) = -\init(\alpha_i) \text{ if }\ww_s<\ww_i$;
\item   $\ww_i'=-\gamma<\ww_s' = -\val(\alpha_s-\alpha_i)  \leq \ww_i\;  \text{ if }\ww_s=\ww_i  \,\text{ and }s\neq i$.
\end{enumerate}
  \end{lemma}
\begin{proof} The first claim follows immediately from the strong non-Archimedean triangle inequality since $\alpha_j+\zeta t^{\beta}+t^{\gamma}$ has valuation $-\ww_i$. A similar argument proves the second claim. In particular, $\alpha_s'\neq 0$ whenever $\ww_s\neq \ww_i$.

  We now prove the third item. Again, $\val(\alpha_i-\alpha_j-\zeta t^{\beta})>\gamma$, so $\val(\alpha'_i)=\gamma$ and $\alpha'_i\neq 0$. Pick $s\neq i$ with $\ww_s=\ww_i$. We write
  \[\alpha'_s = \underbrace{(\alpha_s-\alpha_i)}_{-\ww_i\leq \val(\cdot)\leq \beta} + \underbrace{(\alpha_i-\alpha_j-\zeta t^{\beta}) - t^{\gamma}}_{\val(\cdot) =\gamma>\beta}.
  \]
By the strong non-Archimedean inequality, $-\ww_i\leq \val(\alpha'_s)\!=\!\val(\alpha_s-\alpha_i)<\gamma$, so $\alpha_s'\!\neq\!0$.
\end{proof}

  As the next example illustrates, the effect of the coordinate change in~\autoref{lm:sepPts}  can easily be visualized by means of a tropical modification followed by a coordinate projection.
\begin{example}Consider  points in the Puiseux series field $K\!=\PS$:
    \[
    \alpha_1= t^{3},\; \alpha_2=2+t,\; \alpha_3=2+t^2,\; \alpha_4=t^{-2}, \;\alpha_5=t^{-3},\; \text{ and } \alpha_6=t^{-4} \text{ in }K^\ast,\]
    where $\ww_1=-3$, $\ww=\ww_2=\ww_3=0$, $\ww_4=2$, $\ww_5=3$, $\ww_6=4$, $\beta\!=\!\zeta\!=\!1$, $1<\gamma=3/2<\val(t^2)$.

    To separate $\alpha_2$ from $\alpha_3$, and place $-\val(\alpha_2)$ to the very left of $\RR$, we reembed the line in the plane via $y=x-(2+t^2)-t-t^{3/2}$. The tropicalization of this planar line together with its marked points and the projection to the $y$-coordinate is depicted in~\autoref{fig:modifypoints}. 
\end{example}

  \begin{figure}
\includegraphics[scale=0.22]{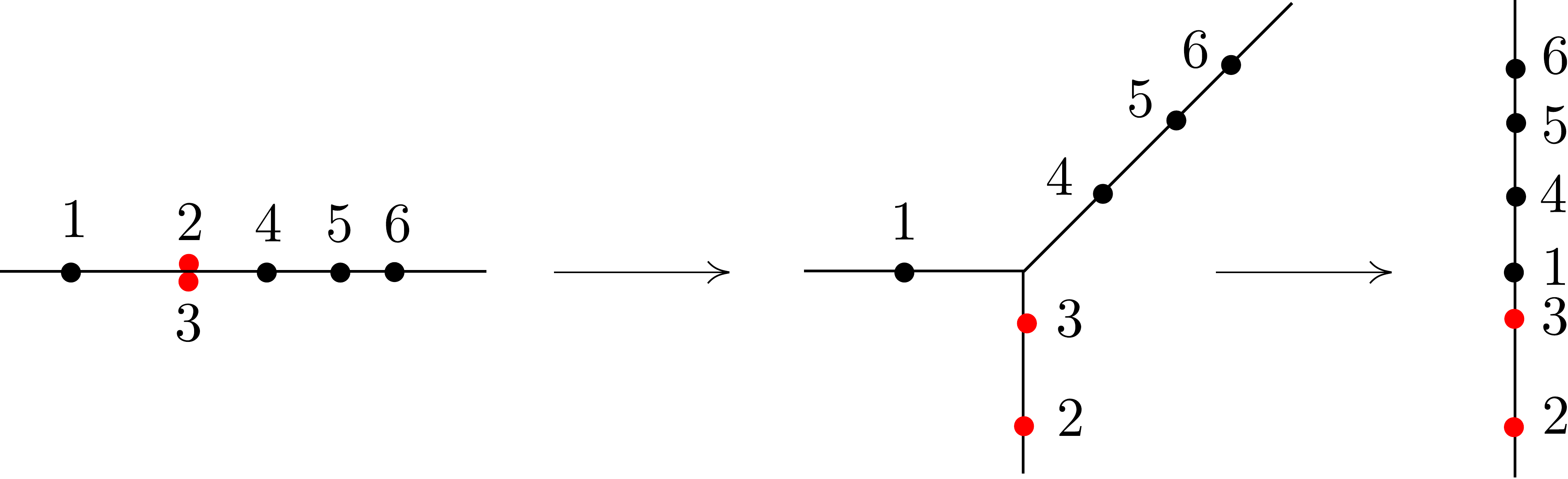}
\caption{Visualizing the coordinate change (A) via tropical modifications.}\label{fig:modifypoints}
\end{figure}

  Our first combinatorial procedure uses a change of coordinates in $\PP^1$ and a relabeling to produce a new tuple $\underline{\alpha}'$ from $\underline{\alpha}$ with the additional property that the maximum and minimum values of $\underline{\omega}'$ are attained exactly once. This is the content of~\autoref{alg:ineq}. In turn,~\autoref{alg:sep} transforms the output of~\autoref{alg:ineq} into a configuration of points in a suitable  region $\Omega^{(i)}$. 
We measure improvement by two auxiliary variables:
\begin{itemize}
\item the \emph{deficiency} $\defua(\underline{\alpha})$ of the point configuration defined as the size of the partition of $[6]=\{1,\ldots, 6\}$ identifying equal coordinates of $\underline{\omega}$,
\item a refined partition $\Lambda$ taking both the valuation and  the initial terms into account.
\end{itemize}
Our partitions will always have the singletons $\{1\}$ and $\{6\}$ since $\ww_1$ and $\ww_6$ remain isolated after each iteration of Step (A).

  \begin{algorithm}[htb] 
  \KwIn{A tuple $\underline{\alpha} = (\alpha_1,\ldots, \alpha_6)$ of six distinct labeled points in $K^*$.}
  \textbf{Assumption:} $\val(K^*)$ is  dense in $\RR$ and the valuation on $K$ splits via $\ww\mapsto t^{\ww}$.\\
  \KwOut{A tuple $\underline{\alpha}'$ obtained from $\underline{\alpha}$ by a linear change of coordinates in $\PP^1$  followed by a relabeling, where~$-\val(\alpha_1')\!<\!-\val(\alpha_2')\!\leq \!\ldots \!\leq \!-\val(\alpha_5')\!<\!-\val(\alpha_6')$.}
  \smallskip
  Relabel the points so that $-\val(\alpha_1)\leq -\val(\alpha_2)\leq \ldots \leq -\val(\alpha_5)\leq -\val(\alpha_6)$\;
  $\underline{\alpha}' \leftarrow \underline{\alpha}$\ ;\ $\underline{\omega}'\leftarrow -\val(\underline{\alpha}'):=(-\val(\alpha_1'), \ldots, -\val(\alpha_6'))$\;
  $\Lambda_{\min}  \leftarrow \{i \colon \ww_i = \min(\underline{\omega}')\}$\ ;\ $\Lambda_{\max} \leftarrow \{i \colon \ww_i = \max(\underline{\omega}')\}$\;
  \If{ $|\Lambda_{\min}|>1$}
     {Relabel $\Lambda_{\min}$ so that $\{\val(\alpha_i-\alpha_j)\colon i, j \in \Lambda_{\min}, i\neq j\}$ is maximized at $i\!=\!1, j\!=\!2$\;\ifhmode\\\fi
       $\underline{\alpha}'\leftarrow$  $\psi(\underline{\alpha})$ where $\psi$ is defined as in \eqref{eq:sepPts} with $\ww=\ww_1$, $i=1$, $j=2$\;
       $\underline{\omega}'\leftarrow -\val(\underline{\alpha}')$\;}
     \If{$|\Lambda_{\max}|>1$}
        {$\underline{\alpha}' \leftarrow $  Coordinate change \textbf{(B)} on $\PP^1$ applied to $\underline{\alpha}'$  ;  $\Lambda_{\max} \leftarrow \Lambda_{\min}$\;\ifhmode\\\fi
          Relabel $\Lambda_{\max}$ so that $\{\val(\alpha_i'-\alpha_j')\colon i, j \in \Lambda_{\max}, i\neq j\}$ is maximized for $i\!=\!6, j\!=\!5$\;
        $\underline{\alpha}'\leftarrow$  $\psi(\underline{\alpha}')$ where $\psi$ is defined as in \eqref{eq:sepPts} with  $\ww=\ww_6'$, $i=6$, $j=5$\;
        $\underline{\alpha}' \leftarrow $  Coordinate change \textbf{(B)} on $\PP^1$ applied to $\underline{\alpha}'$\;}
  \KwRet{$\underline{\alpha}'$.}
    \caption{Separate the minimum and maximum values of $\underline{\omega}$. \label{alg:ineq}}
  \end{algorithm}

\begin{algorithm}[htb] 
  \KwIn{A tuple $\underline{\alpha} = (\alpha_1,\ldots, \alpha_6)$ of six distinct labeled points in $K^*$ with  $-\val(\alpha_1)<-\val(\alpha_2)\leq \ldots \leq -\val(\alpha_5)<-\val(\alpha_6)$.}
  \textbf{Assumption:} $\val(K^*)$ is  dense in $\RR$ and the valuation on $K$ splits via $\ww\mapsto t^{\ww}$.\\ 
  \KwOut{A pair $(\underline{\alpha}', i)$ where $i\in\{\text{I}, \ldots,\text{VII}\}$ and $\underline{\alpha}'$  lies in $\Omega^{(i)}$. 
     Here,  $\underline{\alpha}'$ is  obtained from $\underline{\alpha}$ by a linear change of coordinates in $\PP^1$  followed by a relabeling of $[6]=\{1,\ldots, 6\}$ if needed.}
  \smallskip
    $\underline{\alpha}'\leftarrow \underline{\alpha}$ \ ; \  $\underline{\omega}'\leftarrow -\val(\underline{\alpha}')$\ ;\ 
  $d\leftarrow \operatorname{def}(\underline{\alpha}'):=$ deficiency of $\underline{\alpha}'$\;
  $\Lambda \leftarrow$ $\partua(\underline{\alpha}'):=$ partition of $[6]$ 
  determined by equality among $(\ww_i', \init(\alpha_i'))$'s\;
          {
    \While{$d = 3$}
       {
          \lIf{$|\Lambda|=6$}
               {\DontPrintSemicolon\KwRet{$(\underline{\alpha}',\text{VII})$.}}
               \PrintSemicolon
 {             Relabel $\{2,\ldots, 5\}$ so  $\max\{\val(\alpha'_i-\alpha'_j): \ww_i'=\ww_j'=\ww_2'\} = \val(\alpha'_2-\alpha'_3)$\;
              $\underline{\alpha}'\leftarrow$  $\psi(\underline{\alpha}')$ where $\psi$ is defined as in \eqref{eq:sepPts} with $\ww=\ww_2'$, $i=2$, $j=3$\;
               Relabel $[6]$ by incr.~$-\!\val(\underline{\alpha}')$; $\Lambda\leftarrow \partua(\underline{\alpha}')$; \
              $\underline{\omega}'\leftarrow -\!\val(\underline{\alpha}')$;  $d\leftarrow \defua(\underline{\alpha}')
               $\;}}

       \While{$d = 4$}
             {\lIf{$|\Lambda| = 6$, $\ww_2'=\ww_3'$, and
                 $\ww_4'=\ww_5'$ }
               {\DontPrintSemicolon\KwRet{$(\underline{\alpha}', \text{V})$.}}
               \PrintSemicolon
               \lElseIf{$|\Lambda| = 6$, $\ww_2'<\ww_3'$}
               {\DontPrintSemicolon\KwRet{$(\underline{\alpha}', \text{VI})$.}}
               \PrintSemicolon
                                    \uElseIf{($|\Lambda|=6$ and $\ww_4'<\ww_5'$) or ($|\Lambda|<6$, $\ww_3'<\ww_4'$ and $\init(\alpha'_2)\neq\init(\alpha'_3)$) or ($|\Lambda|<6$ and $\ww_2'<\ww_3'$)}
                       {$\underline{\alpha}' \leftarrow $ Coordinate change \textbf{(B)} on $\PP^1$ applied to $\underline{\alpha}'$ with  relabeling of $[6]$\;
$\underline{\omega}'\leftarrow -\val(\underline{\alpha}')$  \ ; \ $\Lambda\leftarrow \partua(\underline{\alpha}')$\  ;\  $d\leftarrow \defua(\underline{\alpha}')
               $\;                       }
                       \Else
                       {Relabel $\{2,\ldots, 5\}$ so that  $\max\{\val(\alpha'_i-\alpha'_j)\!: \ww_i'=\ww_j'=\ww_2'\} \!=\! \val(\alpha'_2-\alpha'_3)$\;
                         $\underline{\alpha}'\leftarrow$
$\psi(\underline{\alpha}')$ where $\psi$ is defined as in \eqref{eq:sepPts} with
 $\ww=\ww_2'$, $i=2$, $j=3$\;
              Relabel $[6]$ by incr.~$-\!\val(\underline{\alpha}')$; $\Lambda\leftarrow \partua(\underline{\alpha}')$; \
              $\underline{\omega}'\leftarrow -\!\val(\underline{\alpha}')$;  $d\leftarrow\defua(\underline{\alpha}')$\;}}

     \While{$d = 5$}
           {
             \lIf{$\ww_3'=\ww_4'$ and $\init(\alpha_3')=\init(\alpha_4')$}
                 {
                   \DontPrintSemicolon\KwRet{$(\underline{\alpha}', \text{II})$.}
                 }
               \PrintSemicolon
                       \lElseIf{$\ww_3'=\ww_4'$ and $\init(\alpha_3')\neq \init(\alpha_4')$}
                   {
                     \DontPrintSemicolon\KwRet{$(\underline{\alpha}', \text{III})$.}
                   }
                 \PrintSemicolon
        \uElseIf{$\ww_2'=\ww_3'$ and $\init(\alpha_2') = \init(\alpha_3')$}
                              {$\underline{\alpha}'\leftarrow$ $\psi(\underline{\alpha}')$ where $\psi$ is defined as in \eqref{eq:sepPts} with $\ww\!=\!\ww_2$, $i\!=\!2$, $j\!=\!3$\;
                               Relabel $[6]$ by incr.~$-\!\val(\underline{\alpha}')$ and                                \KwRet{$(\underline{\alpha}', \text{I})$.}
                             }
                 \lElseIf{$\ww_4'=\ww_5'$ and $\init(\alpha_4') \neq \init(\alpha_5')$}
                     {
                       \DontPrintSemicolon\KwRet{ $(\underline{\alpha}', \text{IV})$.}
                     }
                     \PrintSemicolon
                             \Else{$\underline{\alpha}' \leftarrow $ Coordinate change \textbf{(B)} on $\PP^1$ applied to $\underline{\alpha}'$ with relabeling of $[6]$\;
                       $\underline{\omega}'\leftarrow -\val(\underline{\alpha}')$  \ ; \ $\Lambda'\leftarrow \partua(\underline{\alpha}')$\;}
           }
          \KwRet{$(\underline{\alpha}', \text{I})$.}}

           \caption{Finding a representative of a tuple $\underline{\alpha}$ in some $\Omega^{(i)}$ for $i\in\{\text{I}, \ldots,\text{VII}\}$.\label{alg:sep}}
  \end{algorithm}

\begin{proof}[Proof of~\autoref{alg:sep}] If the input $\underline{\alpha}$ is already in one of the desired regions $\Omega^{(i)}$ for $i$ in $\{\text{I}, \ldots, \text{VII}\}$, the algorithm outputs the pair $(\underline{\alpha}, i)$.
  If not,  the deficiency of the partition $\Lambda$ of $\underline{\alpha}'$ gives us precise rules to apply  transformations (A) and (B) to improve this invariant one step at a time.
 Before each iteration, we use the turn around transformation (B) followed by a relabeling of $[6]$ (to satisfy $-\val(\alpha_j)\leq-\val(\alpha_{j+1})$ for all $j=1,\ldots, 6$) to reduce ourselves to the case when  $\ww_2=\ww_3>-\val(\alpha_2-\alpha_3)$ and  $\{\val(\alpha_i-\alpha_j):  \ww_i=\ww_j=\ww_2\}$ is maximized at $i=2, j=3$. In this situation, the change of coordinates (A) on $\PP^1$ with $\ww\!=\!\ww_2$, $i\!=\!2$, and $j\!=\!3$  turns $\underline{\alpha}$ to $\underline{\alpha}'\in (K^*)^6$ and $\defua(\underline{\alpha}')>\defua(\underline{\alpha})$. After each such transformation, a relabeling of $[6]$ is performed to ensure the $-\val(\alpha_i)$ are ordered increasingly. The process stops in at most four steps.  
\end{proof}

\section{Faithful re-embedding of planar hyperelliptic curves}
\label{sec:faithful-reemb}

Up to this point, we have only dealt with abstract tropical curves. In this section, we turn our attention to \emph{embedded} tropical plane curves, defined as the dual complex to Newton subdivisions of~\eqref{eq:hyperelliptic}~\cite{bru.ite.mik.sha:14,MSBook, FirstSteps}. Our objective
is to prove~\autoref{thm:faithfulRe-embedding}. Along the way, we analyze the combinatorics of the re-embedded tropical curves, which will vary with the type of the input planar hyperelliptic curve. We assume throughout that the valued group of $K$ is dense in $\RR$ and we fix a splitting $\ww\to t^{\ww}$ of the valuation.

Our first result allows us
to assume that the hyperelliptic cover~\eqref{eq:hyperelliptic} is branched at both $0$ and $\infty$, and that  the leading coefficient $u$ equals  1.
It  ensures that the description of witness regions from~\autoref{tab:CombAndLengthData} remains valid in this setting, for $\ww_1\!=\!-\infty$ and $\ww_6\!=\!\infty$:
\begin{lemma}\label{lm:0InftyBranchPoints}
  After an automorphism of $\PP^1$ sending $\underline{\alpha}'$ to $\underline{\alpha}$, the equation~\eqref{eq:hyperelliptic} becomes
\begin{equation}\label{eq:hypPlanar0Infty}
g(x,y):=  y^2 - x\prod_{i=2}^5(x-\alpha_i)=0 \qquad \alpha_2,\ldots, \alpha_5 \in K^*,
\end{equation}
where  $\ww_i\!:=\!-\val(\alpha_i) \in \RR$,  $\ww_i\!=\! \ww_i' -2 \ww_6' = 2\val(\alpha_6')-\val(\alpha_i')$, $\init \alpha_i = \init \alpha_i' /\init \alpha_6'^2$  for all $i=2,\ldots, 5$, and $\ww_2\leq \ww_3\leq \ww_4 \leq \ww_{5}$.\end{lemma}
\begin{proof} Equation~\eqref{eq:hyperelliptic} is obtained from~~\eqref{eq:hypPlanar0Infty} by means of the projective transformation $\varphi(x)\!:=\!   (x-\alpha_1')/\big((\alpha_1'-\alpha_6')(x-\alpha_6')\big)$ and replacing   $y$ with $\displaystyle{y/\Big (x-\alpha_6')^3\!\!\sqrt{u\!\!\!\prod_{1\leq k\leq 5}\!\!\!(\alpha_k'-\alpha_6')}\Big )}$.
  \end{proof}

As discussed in~\autoref{sec:faithf-trop-trop}, the na\"ive tropicalization $\Trop\,V(g)$ induced by~\eqref{eq:hypPlanar0Infty}  is almost never faithful. Our goal in this section is to produce faithful re-embeddings in $(K^*)^3$ for all seven witness regions, both at the level of minimal and extended Berkovich skeleta. We will make full use of the techniques developed in~\autoref{sec:faithf-trop-trop}, in particular \autoref{lm:projections}, which describe these re-embedded tropical curves by means of the three coordinate projections. 

As we will see, except for Type (II), faithfulness can be achieved in the $XZ$-plane, since the relative interior of the  cell $\sigma_4$ from~\eqref{eq:cellsModif} will contain no point from the re-embedded tropical curve $\Trop\, V(I_{g,f})$. For this reason, we postpone the treatment of Type (II) to the end of this section. Furthermore, a refined algebraic lift of the tropical polynomial $F=\max\{Y , A + X , B +2X\}$ from~\eqref{eq:F} will yield faithfulness on the extended skeleta for Types (I) and (III). 

The rest of this section is organized as follows. We start by giving a complete description of vertices, edges and tropical multiplicities of the $xy$-tropicalizations, whose Newton subdivisions are shown in the middle column of~\autoref{tab:nonFaithfulness} and in~\autoref{fig:missingTypeVI}. We do so by calculating various initial forms of the input hyperelliptic equation $g(x,y)$.  The explicit values will depend on the genericity of the branch points $\alpha_2,\ldots, \alpha_5$ and the relation between the expected valuations of all coefficients in  $g$ and their actual valuations. These computations allow us to  determine the function $f(x,y)$ from~\eqref{eq:liftF} appearing in \autoref{thm:faithfulRe-embedding}. \autoref{lm:liftF} confirms the validity of $f$ as a lift of the tropical polynomial $F$. A refined choice $\tilde{f}(x,y)$ of this function, described in~\eqref{eq:liftFtilde}, will allow us to control the combinatorics of the re-embedded tropical curves and achieve faithfulness on the extended skeleta on certain types of curves.
Propositions~\ref{pr:heightsXZ},~\ref{pr:refinedHeights} and \autoref{lm:TypeVINS} analyze the combinatorics of the $xz$-tropicalizations, visible on the right-column of~\autoref{tab:nonFaithfulness}. The description of the $yz$-tropicalizations for each type is done on separate subsections.

\smallskip

In order to find the appropriate lift $f(x,y)$ of the tropical polynomial $F$, we must first predict the Newton subdivision of~\eqref{eq:hypPlanar0Infty} for each witness region. This is done by computing the expected heights of all monomials (i.e., the negative valuation of the coefficients) in the Newton polytope of $g(x,y)$ in terms of $\underline{\omega}$:
\begin{equation*}\label{eq:expHts}
  \tht(x^5)\!=\!\tht(y^2)\!=\!0 \, ,\; \tht(x) \!=\! \sum_{i=2}^5 \ww_i \,
  , \;\tht(x^2)\!=\! \sum_{i=3}^5 \ww_i \, ,\;\tht(x^3) \!=\! \ww_5+\ww_4\;\text{ and }\; \tht(x^4) \!=\! \ww_5.
\end{equation*}
These heights determine the induced subdivision, as seen in~\autoref{tab:nonFaithfulness} and~\autoref{fig:5-7Naive}. Notice that outside Types (I) and (II), the expected heights may not be attained. For example, the coefficient of $x^3$ equals $\alpha_5(\alpha_4 + \alpha_3) + \sum_{i<j<5}\alpha_i\alpha_j$. Unless $\init(\alpha_4) = -\init(\alpha_3)$, its expected height in Type (III) will be achieved. We indicate these situations by red  points in the Newton polytopes. Nonetheless, these special situations have no effect on the tropical world: they will only unmark the given lattice point.

The expected heights determine all vertices in  $\Trop\, V(g)$ from~\autoref{tab:nonFaithfulness} and~\autoref{fig:5-7Naive}:
\[v_1 \!=\! (\ww_2, \ww_2 + \frac{\ww_3+\ww_4+\ww_5}{2})\,,\, v_2 \!=\! v_1 + (\ww_3-\ww_2)\mathbf{1}, v_3 \!=\! (\ww_4, 2\ww_4+\frac{\ww_5}{2}), v_4  \!=\! v_3 + (\ww_5-\ww_4) (1,2).
\]
Unless $v_1=v_2$, the edge $e_{12}$ joining $v_1$ and $v_2$ has tropical multiplicity 2. Similar behavior occurs for the edge $e_{34}$ joining $v_3$ and $v_4$.
Notice that the combinatorial types for $\Trop\, V(g)$ are all distinct, except for Types (II) and (III). However, these two differ as \emph{tropical cycles}, since the tropical multiplicities of the vertex $v_2$ are distinct: it is one for Type (III) but two for Type (II). This follows by computing the initial degenerations with respect to $v_2$:
\[
 \init_{v_2}(g) = y^2 + x^2\init(\alpha_5)(x-\init(\alpha_3))(x-\init(\alpha_4)) \in \resK[x^{\pm}, y^{\pm}].
 \]
 Indeed, $\init_{v_2}(g)$ is irreducible if and only if $\init(\alpha_3)\neq \init(\alpha_4)$. This holds for Type (III) but fails for Type (II) as~\autoref{tab:CombAndLengthData} indicates. In the latter case, $\init_{v_2}(g)$ has two reduced components, so $m_{\trop}(v_2) = 2$.
  \begin{table}
    \centering
{\renewcommand{\arraystretch}{0.35}}
  \begin{tabular}{|c|c|c|}
    \hline
    \small{Cells and skeleta}\normalsize& Na\"ive tropicalization
        &  xz-tropicalization     \\ 
    \hline
    \hline 
                         \multirow{2}{*}
                                 {(I)}      &   \multirow{4}{*}
                                 {\scalebox{0.8}{\includegraphics[scale=0.45]{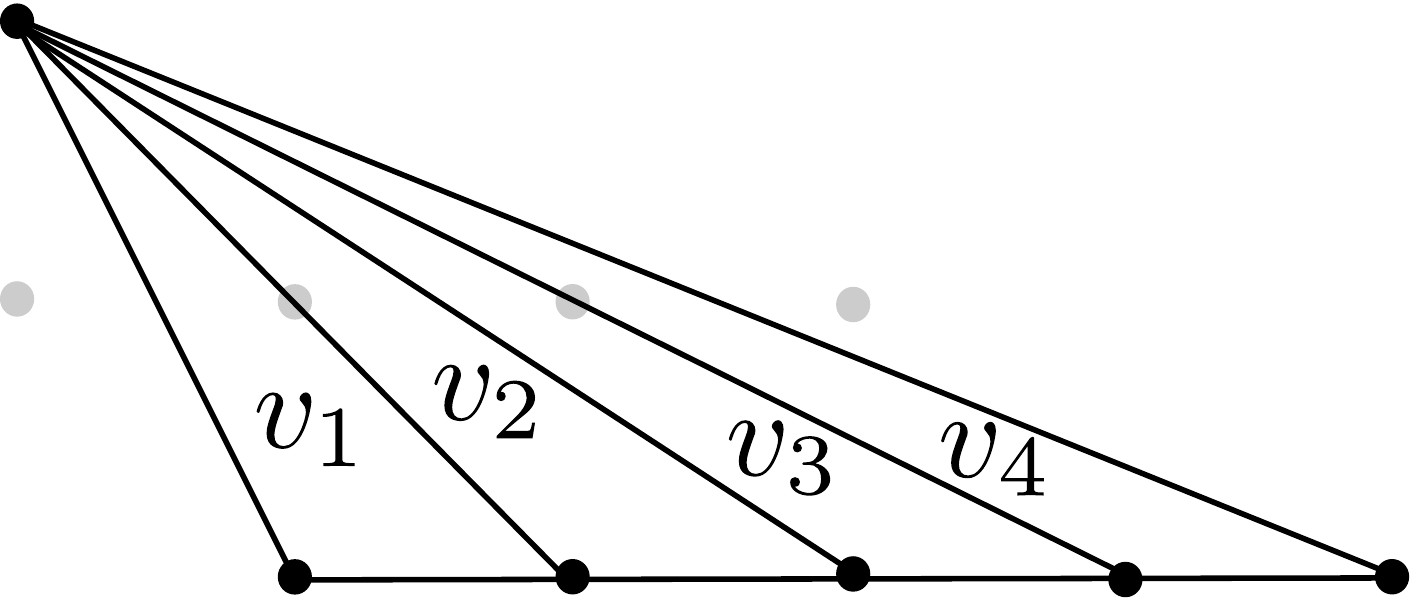}}}
                                 &  
                                 \multirow{4}{*}
                                 {\scalebox{0.8}{\includegraphics[scale=0.45]{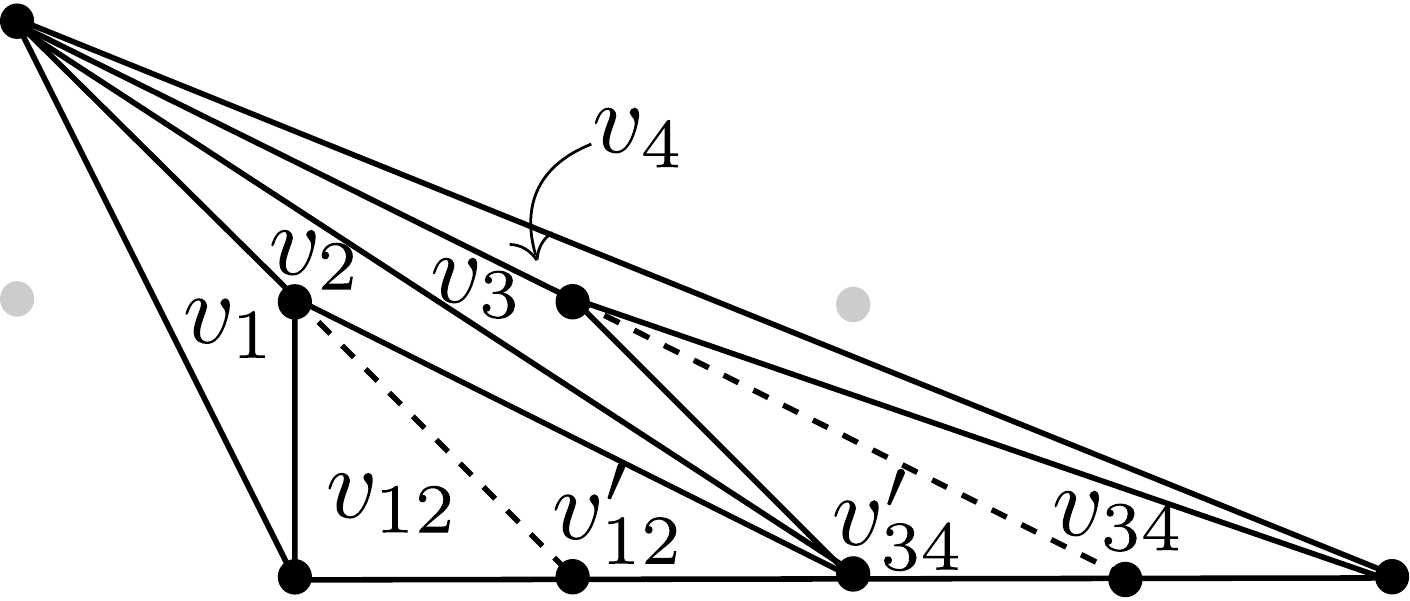}}}      \\
                                 & &   \\
                                                     \multirow{2}{*}{ \includegraphics[scale=0.3]{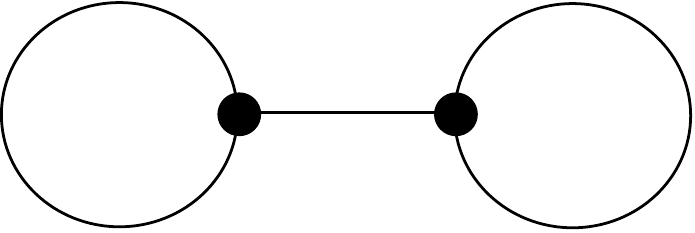}} & &  \\
    & & \\ && \\
             \hline
                         \multirow{2}{*}
                                 {(II)} &   \multirow{4}{*}
                                 {\scalebox{0.8}{\includegraphics[scale=0.45]{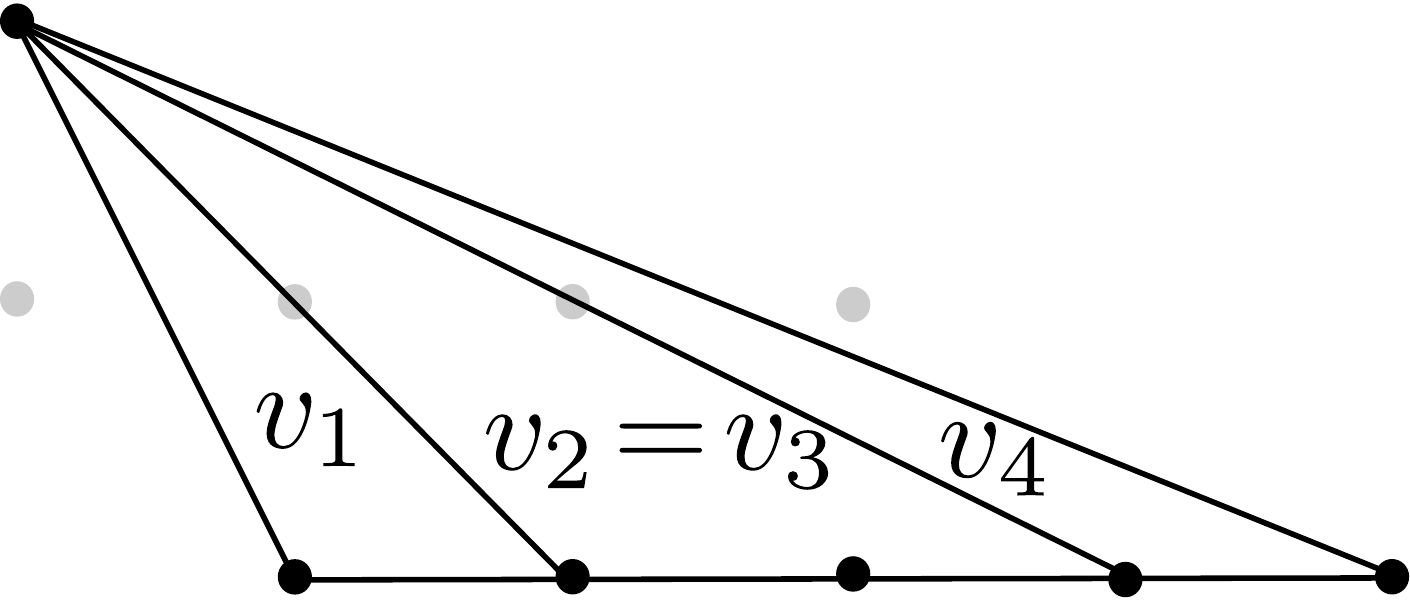}}}
                                 &  
                                 \multirow{4}{*}
                                 {\scalebox{0.8}{\includegraphics[scale=0.45]{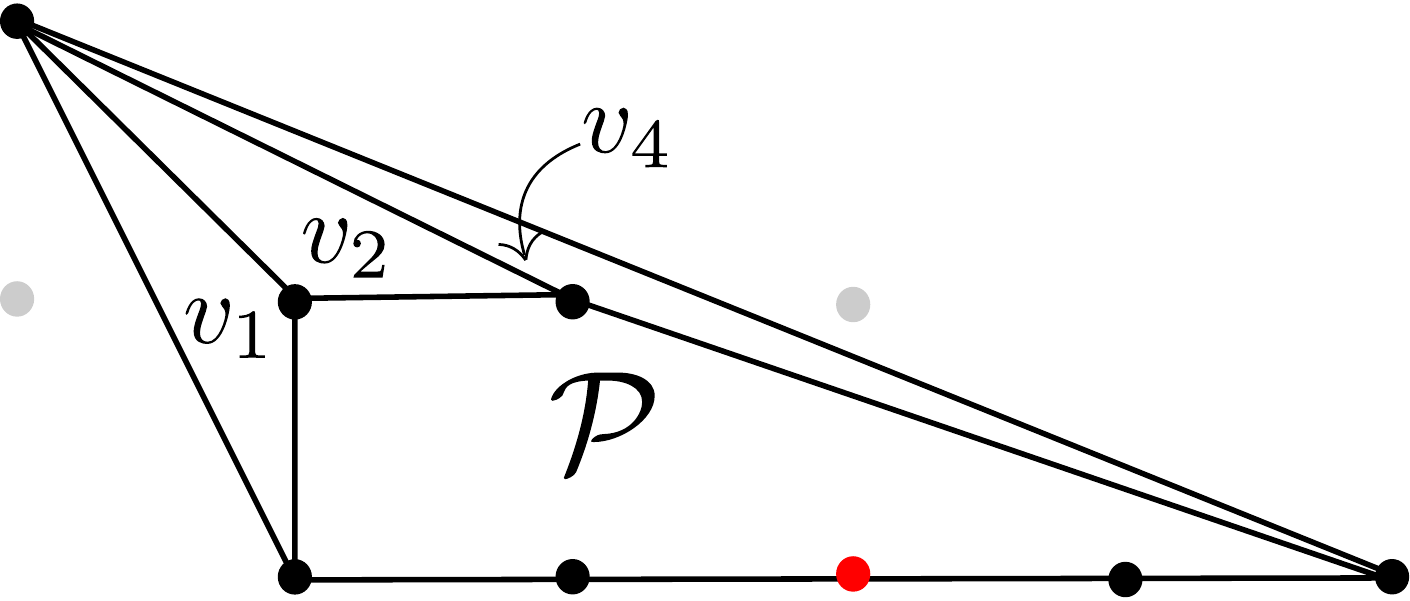}}} 
                                  \\
                                             & &   \\
 \multirow{2}{*}{ \includegraphics[scale=0.3]{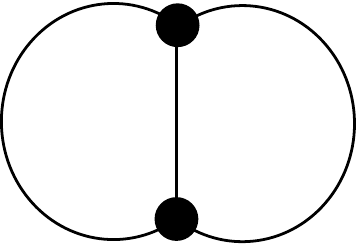}}                                 & &  \\
    & & \\ && \\
             \hline

                         \multirow{2}{*}
                                 {(III)} &   \multirow{4}{*}
                                 {\scalebox{0.8}{\includegraphics[scale=0.45]{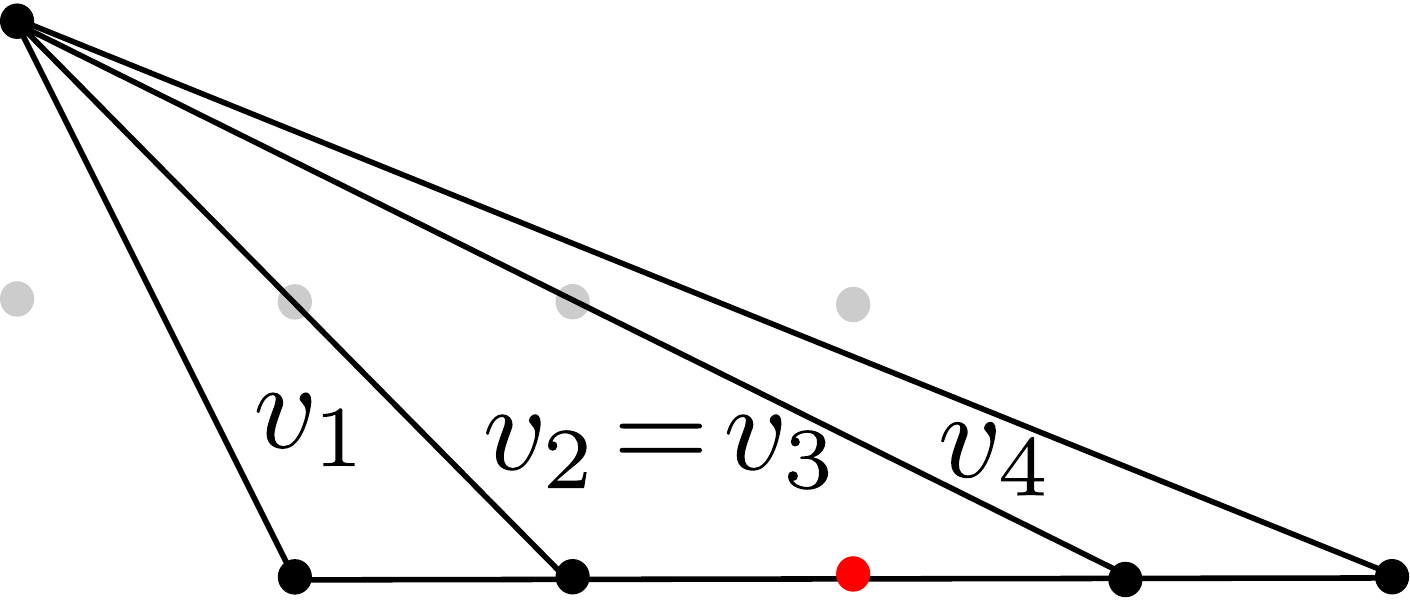}}}
                                 &  
                                 \multirow{4}{*}
                                 {\scalebox{0.8}{\includegraphics[scale=0.45]{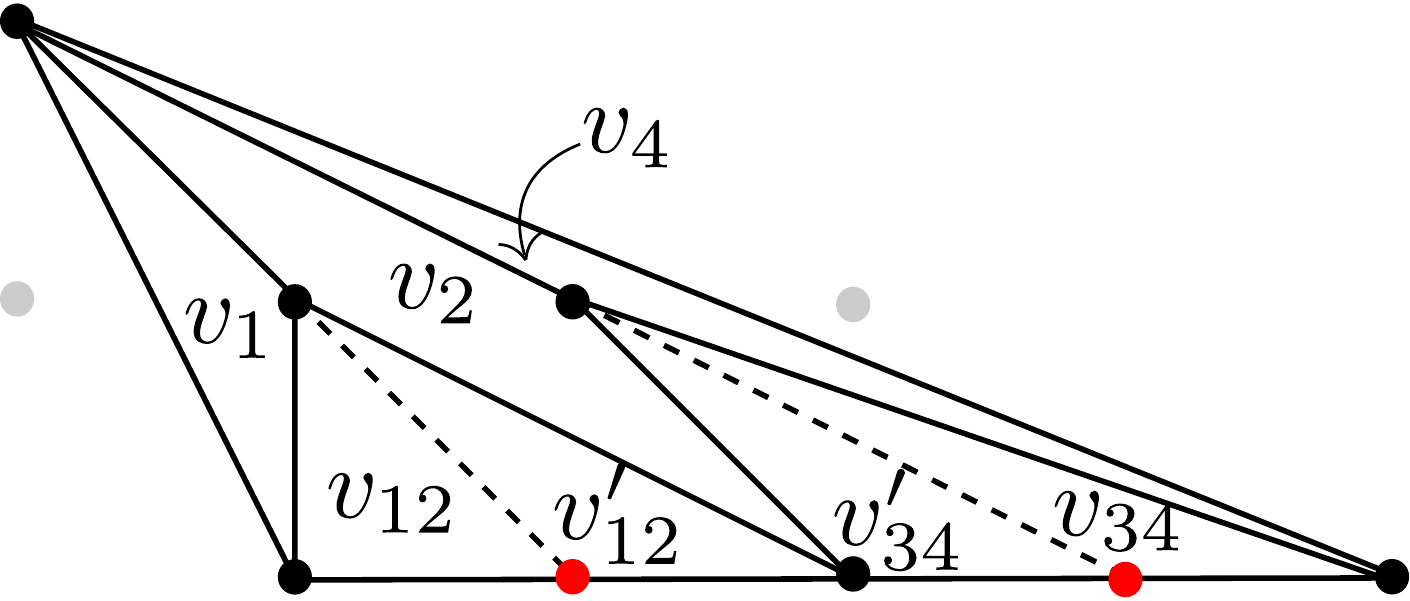}}} 
                                  \\
                                 & &   \\
                             \multirow{2}{*}{ \includegraphics[scale=0.3]{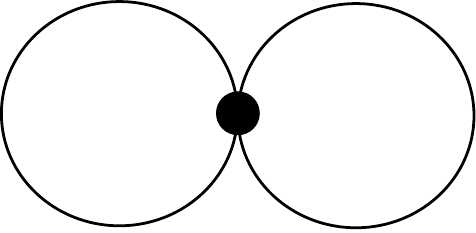}}    & &  \\
                                 
    & & \\ && \\
             \hline
                         \multirow{2}{*}
                                 {(IV)} &   \multirow{4}{*}
                                 {\scalebox{0.8}{\includegraphics[scale=0.45]{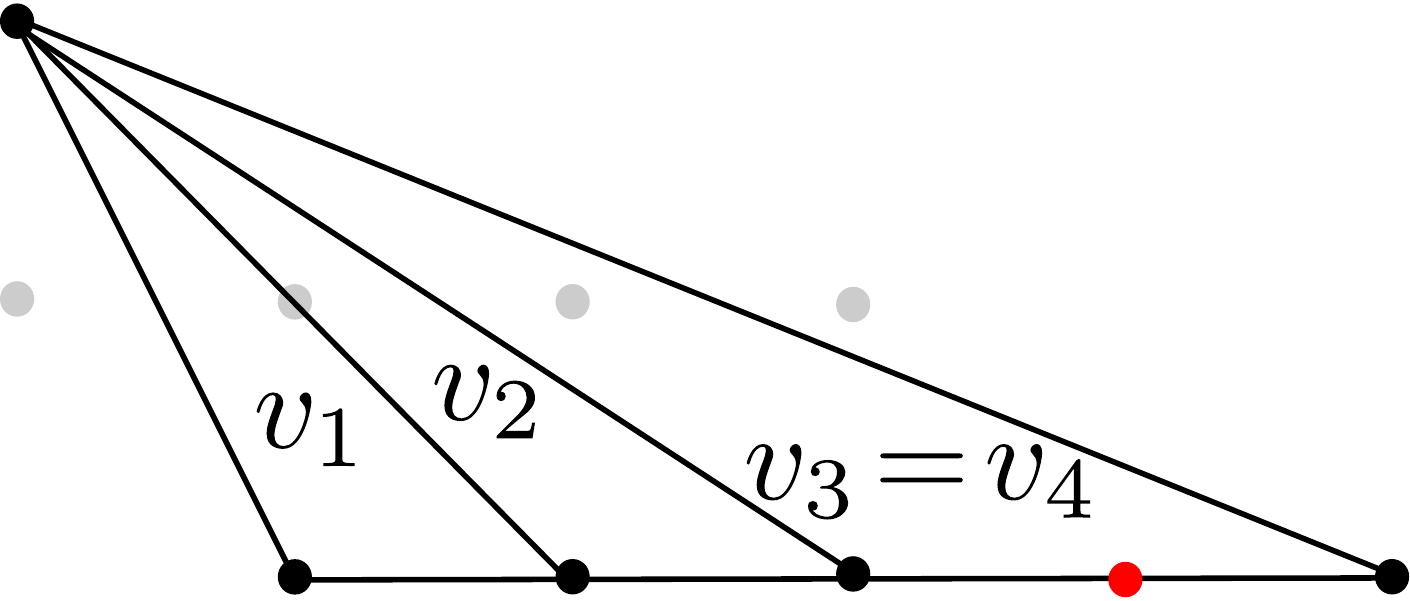}}}
                                 &  
                                 \multirow{4}{*}
                                 {\scalebox{0.8}{\includegraphics[scale=0.45]{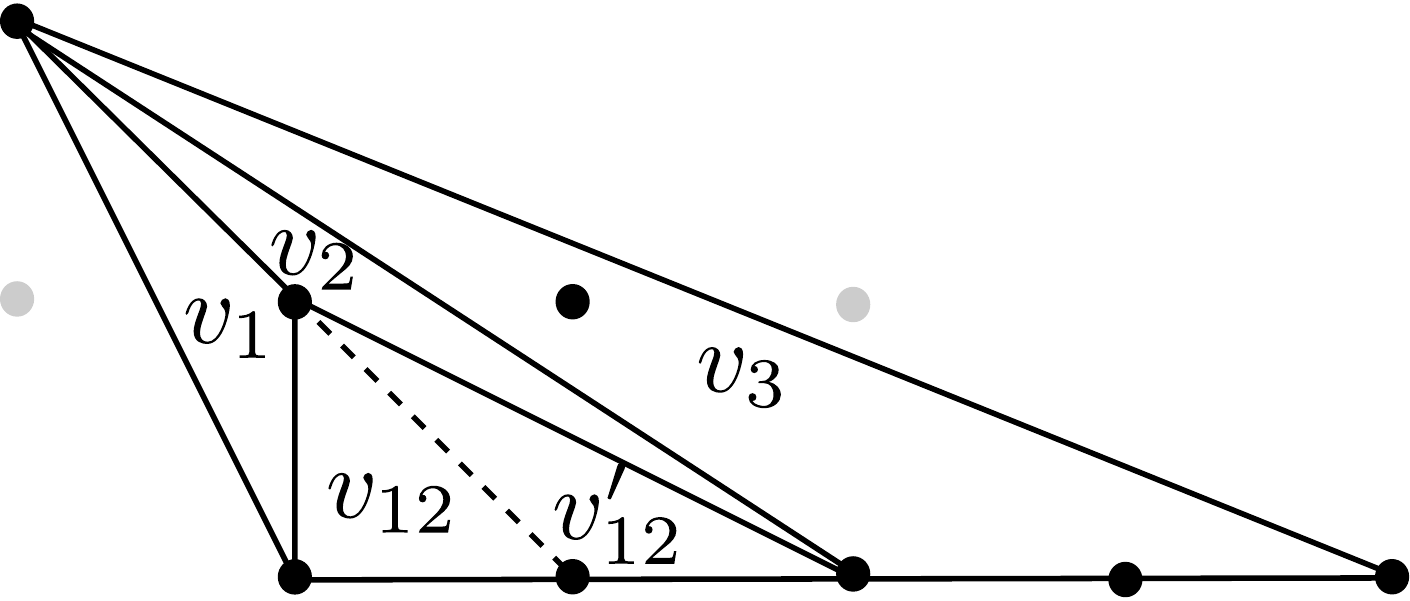}}} 
                                  \\
                                 & &   \\
                              \multirow{2}{*}{ \includegraphics[scale=0.3]{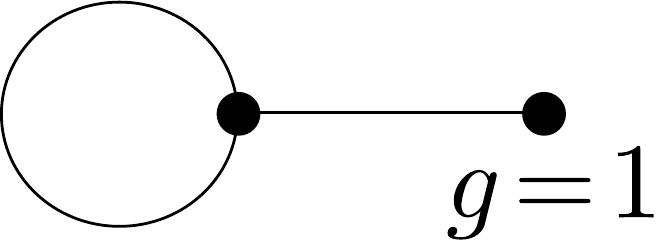}}   & &  \\
    & & \\ && \\
             \hline

                                      \multirow{2}{*}
                                 {(VI)} &   \multirow{4}{*}
                                 {\scalebox{0.8}{\includegraphics[scale=0.45]{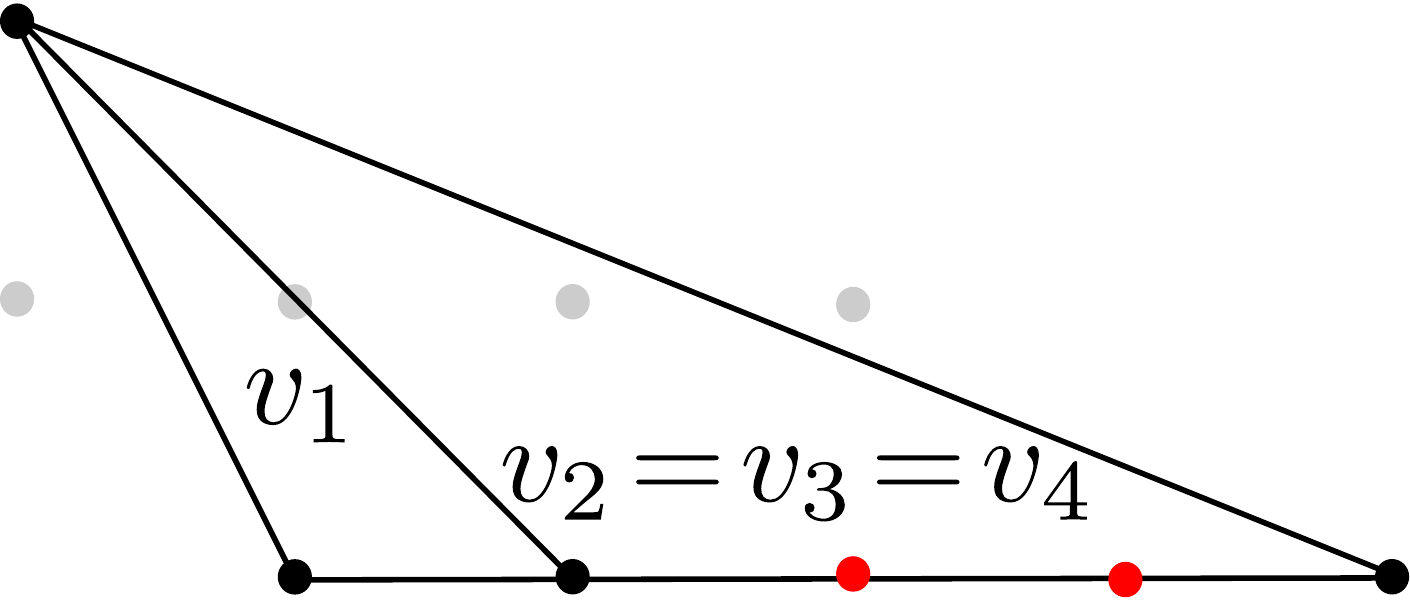}}}
                                 &  
                                 \multirow{4}{*}
                                 {\scalebox{0.8}{\includegraphics[scale=0.45]{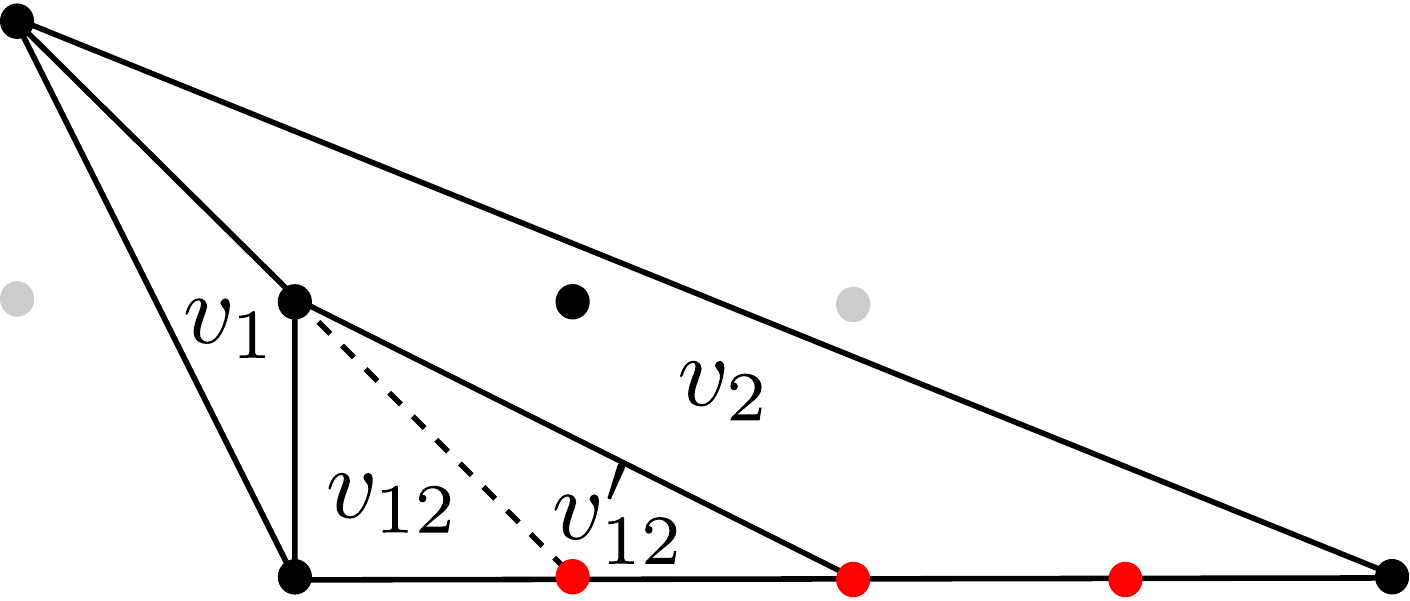}}} 
                                  \\
                                 & &   \\
                     \multirow{2}{*}{ \includegraphics[scale=0.3]{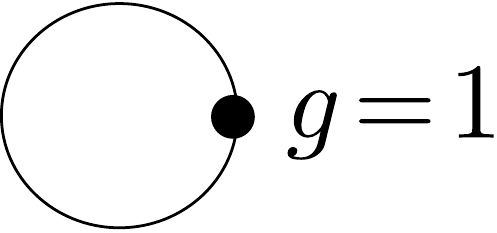}}            & &  \\
    & & \\ && \\
             \hline

  \end{tabular}
\caption{Na\"ive tropicalization for cells (I), (II), (III), (IV) and generic (VI), and planar re-embeddings described by Newton subdivisions. All planar re-embeddings are faithful except for Type (II). The polygon $\mathcal{P}$ will be further subdivided, as in~\autoref{ss:Theta}. The dashed edges correspond to the refined lift $\tilde{f}$~\eqref{eq:liftFtilde} of the tropical polynomial $F$.  The red points' heights might be lower than expected for special choices of $\alpha_2,\ldots, \alpha_5$. The grey  points have height $-\infty$. All vertices are described in~\eqref{eq:vertices} and~\eqref{eq:verticesRef}.\label{tab:nonFaithfulness}}  \end{table}

The tropical polynomial $F$ from~\eqref{eq:F} associated to  $A:=(\ww_3+\ww_4+\ww_5)/2$ and $B:=\ww_5/2$ contains all  vertices of $\Trop\, V(g)$ and the edges between them.
Our choice of lifting for $F$ is governed by the initial degenerations of $\Trop \,V(g)$ along the (possibly degenerate) multiplicity two edges $e_{12}=\overline{v_1v_2}$ and $e_{34}=\overline{v_3v_4}$. Whenever these edges have positive length, the method unfolds them and produces loops in the re-embedded tropical curve, as in~\cite[Theorem 3.4]{cue.mar:16}. We propose:
\begin{equation}\label{eq:liftF} 
  f(x,y) := y - \sqrt{-\alpha_3\,\alpha_4\,\alpha_5}\; x +\sqrt{-\alpha_5}\; x^2.
\end{equation}

Since $\init_{e_{12}}(g) = y^2 + \init(\alpha_3\alpha_4\alpha_5)\,x^2$ and $\init_{e_{34}}(g) = y^2 + \init \alpha_5 \, x^4$ we verify:
\begin{lemma}\label{lm:liftF} The polynomial $f$ from~\eqref{eq:liftF} is a lifting of $F$ and its initial degenerations $\init_{e_{12}}(f)$ and $\init_{e_{34}}(f)$ are irreducible components of $\init_{e_{12}}(g)$ and $\init_{e_{34}}(g)$, respectively.  
\end{lemma}

The next result recovers the Newton subdivision of the polynomial
\begin{equation}\label{eq:gtilde}
  \tilde{g}(x,z)\!:=\!g(x,z+\beta_3\beta_4\beta_5\,x-\beta_5\,x^2)\quad \text{ where }\; \alpha_i\!=\!\beta_i^2 \text{ for } i=2,3,4 \text{ and }\alpha_5\! =\! -\beta_5^2. 
\end{equation}
generating the  ideal $I_{g,f}\cap K[x^\pm, z^\pm]$ from~\autoref{tab:nonFaithfulness}:
\begin{proposition}\label{pr:heightsXZ} For Types (I), (III), (IV) and (VI), the  expected heights of $\tilde{g}(x,z)$ are:
      \[\begin{aligned}
    \tht(z^2)&\!=\tht(x^5)=0\;,\quad\tht(xz)\! =\! \frac{\ww_5\!+\!\ww_4\!+\!\ww_3}{2}\;,\quad\tht(x^2z)\! =\! \frac{\ww_5}{2}\;,\quad \tht(x) \!=\! \ww_5\!+\!\ww_4\!+\!\ww_3\!+\!\ww_2
    ,\\
\tht(x^3) &= \ww_5+\ww_4 \;,\qquad\tht(x^4)=\ww_4\;,\qquad\qquad\quad\tht(x^2)=\ww_5+\ww_4+\ww_3.
  \end{aligned}
  \]
  The expected heights for $z^2, x^5, xz$ and $x^2z$ are always achieved. For the remaining monomials, genericity conditions need to be imposed for  Types (III) and (VI)  (see~\autoref{tab:nonFaithfulness}.)
\end{proposition}
\begin{proof}
An explicit computation with \texttt{Singular} (see the Supplementary material):
reveals that the coefficients of  $\tilde{g}$ from~\eqref{eq:gtilde} equal:
  \[
  \begin{aligned}
    \cf(x^5) &=-\cf(z^2) =1\;,\quad
  \cf(xz) =-2\beta_3\beta_4\beta_5\;,\quad
  \cf(x^2z) = 2\beta_5\;,\\
  \cf(x^4) &= -\alpha_2-\alpha_3-\alpha_4\;,\quad\cf(x^3) = \alpha_5(\beta_3-\beta_4)^2+\alpha_3 \alpha_4 + \alpha_2(\alpha_3+\alpha_4+ \alpha_5)\;,\\
  \cf(x^2) &= - \alpha_2((\alpha_3+\alpha_4)\alpha_5 +\alpha_3\alpha_4 )
  \;,\quad \cf(x) = \alpha_2\alpha_3\alpha_4\alpha_5\;.
  \end{aligned}
  \]
  The characterization of each witness region  in~\autoref{tab:CombAndLengthData} gives both the expected heights for each relevant monomial and the genericity conditions requiered to achieve them:
\begin{description}
  \item [{$\mathbf{x^4}$}] $\init(\alpha_3)+\init(\alpha_4) \neq 0$ for (III) or (VI),
  \item [$\mathbf{x^3}$] $\init(\alpha_5)(\init(\beta_3)-\init(\beta_4))^2 + \init(\alpha_3)\init(\alpha_4) \!\neq 0$ for (VI),
  \item [$\mathbf{x^2}$]  $\init(\alpha_3)\!+\!\init(\alpha_4)\!\neq\!0$ for (III); $(\init(\alpha_3)\!+\!\init(\alpha_4))\init(\alpha_5) \!+\!\init(\alpha_3)\init(\alpha_4)\!    \neq\! 0$ for (VI).\qedhere
\end{description}
\end{proof}
\noindent
The previous result, together with the characterization of  all six maximal cells of $\Trop\, V(f)$ in~\eqref{eq:cellsModif} yield explicit formulas for all vertices of the $XZ$-projections depicted in~\autoref{tab:nonFaithfulness}:
\begin{equation}\label{eq:vertices}
   \begin{minipage}{0.55\textwidth}
  \[
  \begin{aligned}
      v_1 &= (\ww_2,\, \ww_2\!+\!\frac{\ww_3\!+\!\ww_4\!+\!\ww_5}{2},\, \ww_2\!+\!\frac{\ww_3\!+\!\ww_4\!+\!\ww_5}{2}),\\
v_{12}&=v_{12}' = v_1 +(\ww_3-\ww_2)/2 \,(1,1,0)    \;,\\
v_{34} &= v_{34}' \!=\! v_3 + (\ww_5-\ww_4)/2\, (1,2,1)\;,  \end{aligned}
  \]
     \end{minipage}
   \begin{minipage}{0.35\textwidth}
     \[\begin{aligned}
     v_3\! &= (\ww_4,\, 2\ww_4\!+\!\frac{\ww_5}{2},\, 2\ww_4\!+\!\frac{\ww_5}{2}),\\
     v_2 &= v_1 + (\ww_3-\ww_2) \mathbf{1}\;, \\
     v_4  & = v_3 + (\ww_5-\ww_4) (1,2,2).
        \end{aligned}
    \]  \end{minipage}
\end{equation}
The formulas for $v_{12}$ and $v_{34}$ are valid for Types (III) and (VI)
only generically.  Furthermore, the description of Type (VI) curves done in~\autoref{tab:nonFaithfulness} is only generic.~\autoref{fig:missingTypeVI} shows the combinatorial types of $\Trop\, V(\tilde{g})$ for special configurations of Type (VI). In particular, for this type we can only get a triangle as the dual polygon to $v_2$ in the Newton subdivision of $\tilde{g}(x,y)$ when the coefficients of $x^3$ and $x^4$ are non-generic. We conclude:
\begin{lemma}\label{lm:TypeVINS}
On Type (VI), the initial form $\init_{v_2-(0,0,\lambda)}(\tilde{g}(x,z))$ for any $\lambda>0$ is monomial   only if $\init(\alpha_3)= -\init(\alpha_4)$ and  $2\init(\alpha_5) = \init(\beta_3\beta_4)$.
\end{lemma}

In order to address this non-generic behavior and the combinatorics of $\Trop\,V(\tilde{g})$ for all types discussed in~\autoref{pr:heightsXZ}, it will be convenient to choose a refined lift $\tilde{f}$ of $F$ on Types (I), (III), (IV) and (VI). We define:
\begin{equation}\label{eq:liftFtilde} 
  \tilde{f}(x,y)\! :=\! y - \sqrt{-\alpha_3\,\alpha_4\,\alpha_5}(1 +  t^{\ep})\,x +\sqrt{-\alpha_5}\,(1+\delta t^{\ep'})\, x^2\quad \text{ for }0\!<\!\ep, \ep'\! \ll\! 1,\, 
  \delta \!=\! 0/1.
\end{equation}
By construction,~\autoref{lm:liftF} holds for $\tilde{f}$ as well, and $\Trop\,V(f)\!=\!\Trop\,V(\tilde{f})$. The parameters $\ep, \ep'$  depend on the branch points $\alpha_2,\ldots, \alpha_5$, while the choice of $\delta$ depends solely on the curve type: $\delta=1$ for Types (I) and (III), whereas  $\delta=0$ for (IV) and (VI). Following the notation from~\eqref{eq:gtilde}, the generator $\tilde{g}'$ of the ideal $I_{g,\tilde{f}} \cap K[x^{\pm}, z^{\pm}]$ becomes
\begin{equation}\label{eq:gtilderef}
  \tilde{g}'(x,z)\!:=\!g(x,z+\beta_3\beta_4\beta_5\,(1+t^\ep)x-\beta_5\,(1+\delta t^{\ep'})x^2).
\end{equation}

Our next result shows that when $\ep$ and $\ep'$ are chosen appropriately,  $\tilde{f}$  produces faithfulness on the whole extended skeleton in Types (I) and (III), as~\autoref{tab:nonFaithfulness} indicates.

\begin{proposition}\label{pr:refinedHeights} For Types (I), (III), (IV) and (VI), the coefficients of $\tilde{g}'(x,z)$ and $\tilde{g}(x,z)$ agree with the following five exceptions:
  \[
  \begin{aligned}
  \cf(x^2z) &=\!2\beta_5(1\!+\!\delta t^{\ep'})\,,\; \qquad  \cf(x^2) \!=\! - \alpha_2(\alpha_3\!+\!\alpha_4)\alpha_5 -\alpha_2\alpha_3\alpha_4 + \alpha_3\alpha_4\alpha_5t^\ep (2 + t^{\ep})\,,\\
  \cf(xz)\;\; &=\!-2\beta_3\beta_4\beta_5(1\!+\! t^{\ep}) \,,\;\;
  \cf(x^4) \!=\! -\alpha_2-\alpha_3-\alpha_4 +\alpha_5\delta t^{\ep'}(2+\delta t^{\ep'})\,,\\\cf(x^3)\;\; &= \alpha_5(\beta_3-\beta_4)^2+\alpha_3 \alpha_4 + \alpha_2(\alpha_3+\alpha_4+ \alpha_5) -2\alpha_5\beta_3\beta_4(t^{\ep} + \delta t^{\ep'} + \delta t^{\ep +\ep'})\,.
  \end{aligned}
  \]
  The heights  of $xz$ and $x^2z$ agree with those in~\autoref{pr:heightsXZ}. The expected height of $x^3$ is $\ww_5+\ww_4$ and it is achieved for Type (VI) only when $\val(\alpha_5(\beta_3-\beta_4)^2 + \alpha_3\alpha_4)= -2\ww_4$.
  
  Moreover, if $0<\ep<(\ww_3-\ww_2)/2$ and $0<\ep'<(\ww_5-\ww_4)/2$ (if $\ww_5\neq \ww_4$)), then
  \begin{itemize}
  \item $\tht(x^2) = \ww_5+\ww_4+\ww_3-\ep>\ww_5+\ww_4+\ww_2$ for all four types,
  \item$\tht(x^4)=\ww_5-\ep'$ for Types (I) and (III),
  \item $\tht(x^4) = \ww_4$ for Type (IV), and
    \item $\tht(x^4)\leq \ww_4$ for Type (VI). Equality is achieved if and only if $\init(\alpha_3)\neq -\init(\alpha_4)$.
  \end{itemize}
  \end{proposition}
\begin{proof}
  The result follows by direct computation (see the Supplementary material). The conditions on $\ep$ (and $\ep'$ for (I) and (III)) guarantee that the  heights of $x^2$ and  $x^4$ satisfy:
  \[
\tht(x) + \tht(x^3) \leq \tht(x) + \text{exp}\tht (x^3)\!<\!2\tht(x^2) \; \text{ and }   \tht(x^3) + \tht(x^5) \leq  \text{exp}\tht (x^3)   < 2\tht(x^4).
\]
Under these constraints, the point $x^2$  lies above the plane spanned by $x, x^3$ and $xz$ in the extended Newton polygon. Therefore, the triangle in the Newton subdivision with vertices $x$, $x^3$ and $xz$ will be subdivided by an edge joining $xz$ and $x^2$. For Types (I) and (III), our choice of $\ep'$ produces the same effect for  $x^4$ and the facet spanned by $x^3, x^5$ and $x^2z$.
\end{proof}

\autoref{pr:refinedHeights} implies that when the expected height of $x^3$ is attained, the refined modifications  replace $v_{12}$ and $v_{34}$ by two pairs of vertices, as seen in~\autoref{tab:nonFaithfulness}:
\begin{equation}\label{eq:verticesRef}
     v_{12}\! =\! v_1 + \ep(1,1,0)\,,\;   v_{12}' \! =\! v_2 -\ep(1,1,2)\,,\;  v_{34} \!= \! v_4 -\ep'(1,2, 3)\,\text{ and }\;v_{34}'\!  =\! v_3 + \ep'(1,2,1).
\end{equation}

\begin{figure}
  \includegraphics[scale=0.33]{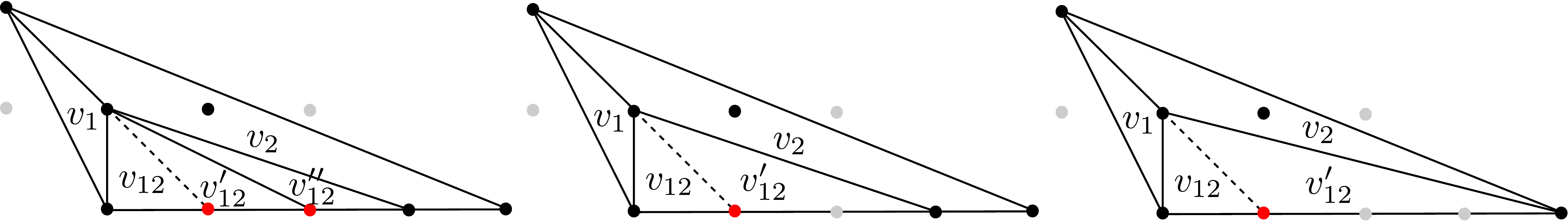}
  \caption{Non-generic $xz$-tropicalizations for Type (VI)  with respect to the height of $x^3$. The rightmost is non-generic with respect to $x^4$ as well.\label{fig:missingTypeVI}}
\end{figure}

\begin{remark}\label{rm:refinedXZ}
The combinatorial types arising from $\tilde{g}(x,z)$ and $\tilde{g}'(x,z)$ for non-generic Type (VI) curves is  more subtle. All possible Newton subdivisions  are shown in~\autoref{fig:missingTypeVI} and they depend on the behavior of $x^3$ and $x^4$. Our bound for $\ep$ given in \autoref{pr:refinedHeights} allows us to split the vertex $v_{12}$ into two or three vertices. There are three cases to analyze:
\begin{enumerate}
  \item 
    When $x^3$ is non-generic but marked and the behavior of $x^4$ is generic (as in the leftmost picture), there will be no high-multiplicity leg in the direction $(0,0,-1)$ and the $xz$-tropicalization will be faithful on the whole extended skeleton. Precise formulas for $v_{12}$, $v_{12}'$ and $v_{12}''$ will depend on the heights of $x^3$ and $x^4$.
  \item When $x^4$ is generic and $x^3$ is unmarked (as in the middle picture), the vertex $v_{12}$ splits into two vertices, with coordinates
    \[v_{12} =v_1+\ep(1,1,0)\;, \qquad v_{12}' = (\ww_4-\ep/2, (5\ww_4/2-\ep)/2, (5\ww_4-3\ep)/2).\]
    A multiplicity two leg in the direction of $(0,0,-1)$ is attached to the vertex $v_{12}'$, so faithfulness on the extended skeleton induced by $\tilde{g}'$ is not guaranteed. If we consider $\tilde{g}$ instead, then $v_{12}=v_{12}'$ and the leg has  multiplicity three. The precise coordinates of $v_{12}$ will depend on the height of $x^3$. 
  \item When $x^4$ and $x^3$ are both non-generic, we cannot predict the combinatorics of the Newton subdivision of $\tilde{g}$. We bypass this difficulty by choosing the refined lift $\tilde{f}$ from~\eqref{eq:liftFtilde} with $\delta=1$ and $\ep, \ep'$ satisfying:
\[0<\ep <\ep'<\min\{\ww_4-\ww_2/2, \ww_4 +\val(\alpha_3+\alpha_4)\}.
\]
In this case, convexity shows that the $xz$-tropicalization of $I_{g,\tilde{f}}$ has a  unique high-multiplicity leg dual to the segment with endpoints $x^2$ and $x^5$, as in the rightmost picture. The remaining legs are adjacent to $v_1$ and $v_4$ and lie in the cells $\sigma_1$ and $\sigma_2$. The heights of $x^2$, $x^3$ and $x^4$ in the right-most picture in~\autoref{fig:missingTypeVI} become $3\ww_4-\ep$, $2\ww_4-\ep'$ and $\ww_4-\ep'$, respectively. Furthermore, the vertices of $\Trop\,V(I_{g,\tilde{f}})$ in $\sigma_6$ are $v_1 = (\ww_2, \ww_2+3\ww_3/2, \ww_2 + 3\ww_3/2)$, $v_4 = (\ww_3, 5\ww_3/2, 5\ww_3/2)$ and
\begin{equation*}\label{eq:nonGenx3x4VI}
  v_{12} = v_1 + \ep(1,1,0),    \quad v_{12}'= v_4 - \ep/3(1, 1, 4).
\end{equation*}
    \end{enumerate}
\end{remark}

\smallskip

The remainder of this section is devoted to the proof of~\autoref{thm:faithfulRe-embedding}, which we do by  a detailed case-by-case analysis.  Following~\cite[Theorem 5.24]{bak.pay.rab:16} we certify faithfulness for $I_{g,f}$ and $I_{g,\tilde{f}}$ by verifying that the tropical multiplicities of all vertices and edges on the tropical (extended) skeleton under the forgetful map equal one.
The Poincar\'e-Lelong formula~\cite[Theorem 5.15]{BPRContempMath} will help us analyze the tropicalizations
\begin{equation}\label{eq:tropMaps}
\trop\colon\Sigma(\mathcal{X})\to \Trop\, V(I_{g,f})\qquad \text{ and} \qquad \trop'\colon \Sigma(\mathcal{X})\to \Trop\,V(I_{g,\tilde{f}})
\end{equation}
where $\Sigma(\mathcal{X})$ denotes the extended skeleton of $\mathcal{X}^{\an}$ with respect to the six branch points. They correspond to the source curves on the left of~\autoref{fig:M2BarAndMumford}.
For all types except (V) and (VII), 
the legs in $\Sigma(\mathcal{X})$ marked with $\alpha_1=0$ and $\alpha_6=\infty$ are mapped isometrically to the legs attached to $v_1$ and $v_4$ with directions $(-2,-1,-1)$ and $(2,5,5)$, respectively.

Whenever faithfulness on $\Sigma(\mathcal{X})$ cannot be achieved via $f$ or $\tilde{f}$, we overcome this issue by employing vertical modifications along tropical polynomials of the form $\trop(x-\alpha_i)$. \autoref{ex:ExTypeII} provides a detailed explanation of our re-embedding methods presented briefly in~\autoref{ex:thetaC}. The Supplementary material includes a complete list of examples (with scripts) for each  combinatorial type, considering generic and special branch point behaviors. The interested reader can simply change the parameters $\alpha_2$, and $\beta_i$'s on the script corresponding to a fixed curve type to produce new examples. 

\subsection{Proof for Type (I)}\label{ss:Dumb} From  the $XZ$-projections of both
$\Trop\, V(I_{g,f})$ and $\Trop\, V(I_{g,\tilde{f}})$ given in~\autoref{tab:nonFaithfulness} we know that the maximal cell $\sigma_4$ does not meet any of these two curves. Thus,  we can ignore the $YZ$-projection when reconstructing the space curves using~\autoref{lm:projections}: it suffices to attach a leg in the direction $(0,-1,0)$ to the vertices $v_1, v_2, v_3$ and $v_4$ in the charts $\sigma_1$ and $\sigma_2$.

From~\autoref{tab:nonFaithfulness}, we see that all vertices and edges in $\Trop\,V(\tilde{g})$ and $\Trop\,V(\tilde{g}')$ have tropical multiplicities one, since their initial degenerations are reduced and irreducible.   This shows that both $xz$-tropicalizations are faithful on the minimal skeleta.
Furthermore, all legs in $\Trop\, V(\tilde{g}')$ have multiplicity one, thus the refined modification induces a faithful tropicalization on the whole tropical curve.  This is not the case for $\Trop\, V(I_{g,f})$ since there are two multiplicity two legs in the direction $(0,0,-1)$.

The tropicalization maps in~\eqref{eq:tropMaps} can be read off from the combinatorics of both re-embedded curves. The  legs attached to $v_1$, $v_2$, $v_3$ and $v_4$  are the isometric images of the legs marked with $\alpha_2, \alpha_3, \alpha_4$ and $\alpha_5$ under the tropicalization maps. These legs get contracted under the $XZ$-projections.~\qed

\subsection{Proof for Type (III)}\label{ss:Fig8}
The $XY$- and $XZ$-projections reveal that $\sigma_4$ intersects both tropical curves $\Trop\, V(I_{g,f})$ and $\Trop \, V(I_{g,\tilde{f}})$ along the ray $\sigma_1\cap \sigma_2\cap \sigma_4$. Thus, we can use~\autoref{tab:nonFaithfulness} to reconstruct the space curves.

All trivalent vertices in the $XZ$-projections of both space curves have tropical multiplicities 1. By~\cite[Corollary 2.14]{cue.mar:16}, we can confirm that $v_2$ has also multiplicity one by showing that the discriminants $\Delta$ of $\init_{v_2}(\tilde{g}(x,z))$ and $\init_{v2}(\tilde{g}'(x,z))$ do not vanish. The explicit descriptions of $\tilde{g}(x,z)$ and $\tilde{g}'(x,z)$ from Propositions~\ref{pr:heightsXZ} and~\ref{pr:refinedHeights} give
\[
\Delta \!=\! \init(\cf(xz))\init(\cf(x^2z))-\init(\cf(x^3))\init(\cf(z^2)) = \init(\alpha_5) (\init(\beta_3)+\init(\beta_4))^2 \neq 0.
\]
From the Newton subdivisions, we see that all bounded edges of both $XZ$-projections have tropical multiplicity one, so both planar re-embeddings are faithful on the minimal skeleta. Since all legs on $\Trop \,V(\tilde{g}')$ have also multiplicity one, we conclude that the $XZ$-projection for the refined modification is also faithful on the extended skeleton.

As with Type (I), the tropicalization~\eqref{eq:tropMaps} maps the legs of $\Sigma(\mathcal{X})$ marked by $\alpha_2$ and $\alpha_5$  isometrically onto the leg adjacent to $v_1$ and $v_4$ in the cells $\sigma_1$ and $\sigma_2$. Since $m_{\trop}(v_2)=2$, the legs marked with $\alpha_3$ and $\alpha_4$ are mapped isometrically onto the leg adjacent to $v_2$, so  these tropicalizations in $\RR^3$ are not faithful on the extended skeleta.
This can be repaired in dimension four by a vertical modification along $X=\ww_4$, via the ideal
\begin{equation}\label{eq:fiathfulExtendedIII}
  J = I_{g,\tilde{f}} + \langle u-(x-\alpha_4)\rangle \subset K[x^{\pm}, y^{\pm}, z^{\pm}, u^{\pm}].
\end{equation}
The tropical curve $\Trop\,V(J)$ in $\RR^4$ is obtained from $\Trop\,V(I_{g,\tilde{f}})$ by four simple operations:
\begin{enumerate}[(i)]\item
  points $p = (p_1,p_2,p_3)$ in $\Trop\,V(I_{g,\tilde{f}})$ with $p_1<\ww_4$ lift to points of the form $(p,\ww_4)$; 
\item points $p=(p_1,p_2,p_3)$ in $\Trop\,V(I_{g,\tilde{f}})$ with $p_1>\ww_4$ lift to points $(p,p_1)$;
  \item the vertex $v_3$ in $\Trop\,V(J)$ has coordinates $(\ww_4,\, 2\ww_4+{\ww_5}/{2},\, 2\ww_4+{\ww_5}/{2}, \ww_4)$;
\item the multiplicity two leg with direction $(0,-1,0)$ adjacent to $v_3$ splits into two multiplicity one legs $\ell_3$ and $\ell_4$, with directions $(0,-1,0,0)$ and $(0,-1,0,-2)$: these are the images of the corresponding legs in $\Sigma(\mathcal{X})$ under the tropicalization map. Indeed,
\[  \begin{aligned}
  \init_{\ell_3}(J) &= \langle -x^2\init(\alpha_5)(x-\init(\alpha_3))u, u-x+\init(\alpha_4), z + \init(\beta_5\beta_4\beta_3)x - \init(\beta_5) x^2\rangle, \\
      \init_{\ell_4}(J) &= \langle y^2-x^2\init(\alpha_5)(x-\init(\alpha_3))u, -x+\init(\alpha_4), z + \init(\beta_5\beta_4\beta_3
      )x - \init(\beta_5) x^2\rangle.
  \end{aligned}
  \]
\end{enumerate}
These two identities follow from standard Gr\"obner bases techniques over valued fields, in particular~\cite[Proposition 2.6.1, Corollary 2.4.10]{MSBook}. 
Notice that the $UY$-projection and its Newton subdivision can be easily obtained by the change of variables $u=x+\alpha_4$. Indeed, the result is a hyperelliptic genus two curve covering $\mathbb{P}^1$, whose six branch points have negative valuations  $-\infty, \ww_4$, $\ww_4$, $\ww_4$, $\ww_5$ and $\infty$. As a consequence, we subdivide its Newton polytope  along an edge joining $y^2$ and $u^4$. Similar reasoning applies to the  $UZ$-projection.~\qed

\subsection{Proof for Type (IV)}\label{ss:TypeIV} The $XY$-and $XZ$-projections from~\autoref{tab:nonFaithfulness} confirm that the two tropical space curves contain no points in $\sigma_4$. Furthermore, both  curves can be obtained from their $XZ$-projections  by attaching  a leg in the direction $(0,-1,0)$ to the vertices $v_1, v_2$ and $v_3$. The leg attached to $v_3$ has multiplicity two, and it is the image of the legs marked with $\alpha_4$ and $\alpha_5$ in $\Sigma(\mathcal{X})$. The legs marked with $\alpha_2$ and $\alpha_3$ are mapped isometrically onto the legs adjacent to $v_1$ and $v_2$ in $\sigma_1$. Both curves have a multiplicity two leg $\ell$ with direction $(0,0,-1)$ attached to  $v_3$:
  \begin{equation}\label{eq:ellIV}
  \init_{\ell}(I_{g,f}) = \init_{\ell}(I_{g,\tilde{f}}) = \left\langle y^2-x^3\prod_{i=4}^5(x-\init(\alpha_i)), y  + \init(\beta_3\beta_4\beta_5)x - \init(\beta_5) x^2\right \rangle.
  \end{equation}

\noindent  
  The vertex $v_3$ is the image of the unique genus one vertex in the Berkovich skeleton, and it is dual to the unique genus one triangle in the Newton subdivision of $g$. Furthermore:
\begin{claim}\label{cl:initDegv3IV} The initial degeneration $\init_{v_3}(g)$ defines a smooth elliptic curve in $(\resK^*)^2$.
    \end{claim}
\noindent Indeed, a direct computation and the Type (IV) defining conditions from~\autoref{tab:CombAndLengthData}   reveal
\begin{equation}
  \label{eq:initv3TypeIV}
  \init_{v_3}(g) = y^2-x^3(x-\init(\alpha_4))(x-\init(\alpha_5)) = x^2 ((y/x)^2
  -x(x-\init(\alpha_4))(x-\init(\alpha_5)),
\end{equation}
so its projectivization is a double cover of $\PP^1_{\resK}$
branched at four distinguished points. 
\begin{remark}\label{rm:jInvTypeIV} An alternative proof for~\autoref{cl:initDegv3IV} can be given in terms of $j$-invariants, by considering the plane cubic curve $\mathcal{X}'$ defined by the truncation $g'$ of $g$ corresponding to all monomials in the triangle dual to $v_3$ in the Newton subdivision of $g$. By construction,  $\Trop\,V(g')$ is the star of $\Trop\,V(g)$ along $v_3$. A direct computation with \texttt{Singular} and \sage\, (available in the Supplementary material) confirms that for any characteristic of $\resK$ other than two, the $j$-invariant of $\mathcal{X}'$ has non-negative valuation, so $\mathcal{X}'$ has good reduction and the vertex of $\Sigma(\mathcal{X}')$ maps to $v_3$.
\end{remark}

The previous discussion confirms that faithfulness occurs at the level of the minimal skeleta but fails for the extended one, due to the presence of the multiplicity two leg $\ell$ in $\sigma_2$ adjacent to $v_3$. This can be fixed using a vertical modification and the  ideal $J$ from~\eqref{eq:fiathfulExtendedIII}. The same procedure from~\autoref{ss:Fig8} allows us to recover $\Trop\,V(J)$ from $\Trop\, V(I_{g,\tilde{f}})$ and $\Trop\,V(I_{g,f})$, where the role of $\ell_3$ is replaced by a leg $\ell_5$. The following identities hold:
\[  \begin{aligned}
  \init_{\ell_5}(J) &= \langle -x^3(x-\init(\alpha_5))u, u-x+\init(\alpha_4), z + \init(\beta_3\beta_4\beta_5)x - \init(\beta_5) x^2\rangle, \\
      \init_{\ell_4}(J) &= \langle y^2-x^3\init(\alpha_5)u, -x+\init(\alpha_4), z + \init(\beta_3\beta_4\beta_5)x - \init(\beta_5) x^2\rangle.
  \end{aligned}
\]
The legs $\ell_4$ and $\ell_5$ adjacent to $v_3$ have directions $(0,-1,0,-2)$ and $(0,-1,0,0)$ and they are isometric images of the legs in $\Sigma(\mathcal{X})$ marked with $\alpha_4$ and $\alpha_5$, respectively. By combining  \eqref{eq:ellIV} with   the identity $\init_{\ell}(J)  = \init_{\ell}(I_{g,\tilde{f}}) + \langle u-x+\init(\alpha_4) \rangle$ we see that the leg $\ell$ from $\Trop\,V(I_{g,\tilde{f}})$ survives in $\Trop\,V(J)$: it has direction $(0,0,-1,0)$ and multiplicity two. 
\qed

\subsection{Proof for Type (VI)}\label{ss:TypeVI}
From~\autoref{tab:nonFaithfulness} we see that the vertex $v_2$ is dual to the unique genus one lattice polygon in the Newton subdivision of $g$. As in Type (IV),  $v_2$ is the image of the unique genus one vertex in the Berkovich skeleton under the $xy$- and $xz$-tropicalizations.

 \begin{claim}\label{cl:initDegv3VI} The initial degeneration $\init_{v_2}(g)$ defines a smooth elliptic curve in $(\resK^*)^2$.
    \end{claim}
\noindent Indeed, the conditions from~\autoref{tab:CombAndLengthData} reveal that   $\init_{v_2}(g) =  x^2 ((y/x)^2-\prod_{i=3}^5(x-\init(\alpha_i))$,
so its projectivization is a double cover of $\PP^1_{\resK}$
branched at four distinguished points.

By construction, the na\"ive tropicalization maps the legs marked with $\alpha_3, \alpha_4, \alpha_5$ in $\Sigma(\mathcal{X})$ isometrically to the leg adjacent to $v_2$ with direction $(0,-1)$. The next initial form computation reveals that this leg is the projection of a multiplicity three leg $\ell$ with direction $(0,-1,0)$ adjacent to $v_2$ which is the image of the aforementioned marked legs in $\Sigma(\mathcal{X})$:
\begin{equation}\label{eq:ellForVI}
\init_{\ell}(I_{g,f}) = \init_{\ell}(I_{g,\tilde{f}}) = \langle x^2\prod_{i=3}^5(x-\init(\alpha_i)), z+\init(\beta_3\beta_4\beta_5)x - \init(\beta_5)x^2\rangle.
\end{equation}
As was discussed earlier, the combinatorics of the $xz$-tropicalizations depend heavily on the genericity of the coefficients
of $x^3$ and $x^4$ in both $\tilde{g}(x,z)$ and $\tilde{g}'(x,z)$. A careful case-by-case analysis confirms that all vertices have multiplicity one. Furthermore,
\[
\init_{v_2}(I_{g,f}) = \init_{v_2}(I_{g, \tilde{f}}) = \langle \init_{v_2}(g), z-y + \init(\beta_3\beta_4\beta_5)x-\init(\beta_5)x^2 \rangle.
\]
Since all bounded edges also have multiplicity one, we conclude that the $xz$-tropicalizations are faithful on the minimal skeleton.  In what follows, we describe the combinatorics of both space curves in each relevant case and analyze faithfulness on the extended skeleton. The genericity conditions for both $x^3$ and $x^4$ are described in Propositions~\ref{pr:heightsXZ} and~\ref{pr:refinedHeights}.\\

\noindent \textbf{Case 1: generic for $x^3$.}
Extended faithfulness cannot be guaranteed  since each star of $v_2$ contains a
multiplicity two leg in $\sigma_5$ with direction $(0,0,-1)$.  The vertex $v_{12}=v_{12}'$ of $\Trop\,V(I_{g,f})$ also has a
multiplicity two leg in $\sigma_6$ with the same direction.

\vspace{1ex}

\noindent \textbf{Case 2: non-generic for $x^3$, generic for $x^4$.} The two possible $xz$-tropicalizations are obtained from the Newton subdivision of $\tilde{g}$ and $\tilde{g}'$ in the left and center of~\autoref{fig:missingTypeVI}. They depend on whether $x^3$ is marked or not. Both cases were discussed in \autoref{rm:refinedXZ}. In the marked case, the  $xz$-tropicalization $\Trop\,V(\tilde{g}')$ is not faithful on the extended skeleton. Indeed, the high multiplicity leg attached to $v_{12}=v_{12}'$ in the direction $(0,0,-1)$ induces an initial degeneration with two distinct reduced components, and faithfulness fails for the extended skeleton. It can be repaired by a vertical modification along this leg and a lift induced by one of these two components.

Similarly, in the unmarked case, ~\autoref{pr:refinedHeights} shows that 
the high multiplicity leg attached to $v_{12}'$ in the direction $(0,0,-1)$ induces an initial degeneration with reduced distinct components. So extended faithfulness fails for the $xz$-tropicalization. Vertical modifications along this leg adapted to these components will repair this situation in dimension three for $I_{g,\tilde{f}}$ and four for $I_{g,f}$.

Finally,  the multiplicity of the leg $\ell$ described in~\eqref{eq:ellForVI} and~\autoref{lm:projections} ensure that the leg attached to the vertex $v_2$ in both $xz$-tropicalizations is the projection of a single multiplicity two leg in the direction $(0,0,-1)$ attached to $v_2$.  This completes the description of the combinatorics of both space curves.

\vspace{1ex}

\noindent \textbf{Case 3: non-generic for both $x^3$ and $x^4$.} As discussed in~\autoref{rm:refinedXZ}, the Newton subdivision of $\tilde{g}$ cannot be predicted, so we focused on the refined modification and the embedding $I_{g,\tilde{f}}$. The Newton subdivision of $\tilde{g}'$, depicted in the right of~\autoref{fig:missingTypeVI} shows that no point of $\Trop(I_{g,\tilde{f}})$ lies in the relative interior of $\sigma_4$. The star of $v_2$ consists of the multiplicity three leg $\ell$ with direction $(0,-1,0)$, the leg $\ell_6$ with direction $(2,5,5)$ and  two bounded edges with directions $(-1,-1,-1)$ and $(-1,-1,-3)$, respectively. The vertex $v_1$ is adjacent to a unique leg, with direction $(-2,-1,-1)$. By~\autoref{pr:refinedHeights}, the vertex $v_{12}'$ is adjacent to a multiplicity two leg with direction $(0,0,-1)$ whose initial degeneration has two distinct reduced components. The $xz$-tropicalization is not faithful on the extended skeleton. This can be repaired by a vertical modification along $\max\{Z,(5\ww_4-3\ep)/2\}$, adapted to one of these  components. 
\smallskip

As with Type (III), the extended skeleton $\Sigma(\mathcal{X})$ can only be revealed by means of vertical modifications through $v_2$ designed to separate the images of the legs  marked with $\alpha_3, \alpha_4$ and $\alpha_5$. We use the ideal
\[
J = I_{g,\tilde{f}} + \langle z_3-(x-\alpha_3), z_4-(x-\alpha_4)\rangle\subset K[x^{\pm}, y^{\pm},z^{\pm}, z_3^{\pm}, z_4^{\pm}].
\]
The leg $\ell$ in the star of $v_2$ in $\Trop\,V(I_{g,f})$ and $\Trop\,V(I_{g,\tilde{f}})$ is replaced by three multiplicity one legs ($\ell_3$, $\ell_4$ and $\ell_5$), with directions $(0,-1,0,-2,0)$, $(0,-1,0,0,-2)$, and $(0,-1,0,0,0)$, each coming from the expected marked leg in $\Sigma(\mathcal{X})$.~\qed

\subsection{Proof for Type (II)}\label{ss:Theta} Throughout this section, and to simplify the exposition, we assume $\charF \resK\neq 2, 3$. A refinement of our methods will be required in characteristic three.

The Type (II) cone manifests itself as the most combinatorially challenging cell of $M_2^{\trop}$. It is the only case for which the chart $\sigma_4$ in the tropical modification of $\RR^2$ contains points of the re-embedded tropical curve $\Trop\,V(I_{g,f})$ in its relative interior. In particular, information from all three coordinate projections is necessary to recover the space curve using~\autoref{lm:projections}. Furthermore, as was already observed in~\autoref{fig:ThetaModif}, depending on the values of the three edge lengths in the theta graph, the $YZ$-projection of  $\Trop\,V(I_{g,f})$ introduces extra crossings and higher multiplicities that need to be unraveled in the reconstruction process. Here is our main result:

\begin{theorem}\label{thm:allCombTypes}
  In Type (II) the tropical curves $\Trop\,V(I_{g,f})$ come in 13 combinatorial types, depicted in~\autoref{fig:allCombTypesTypeII}. These graphs are determined by a subdivision of the Type (II) cone along its baricenter. Precise coordinates for all vertices are given in~\eqref{eq:verticesTypeII}.
\end{theorem}

\noindent The proof of this result is computational and it  involves genericity conditions of the branch points giving each graph. As usual, examples for all cases are provided in the Supplementary material.

   \begin{figure}
     \includegraphics[scale=0.85]{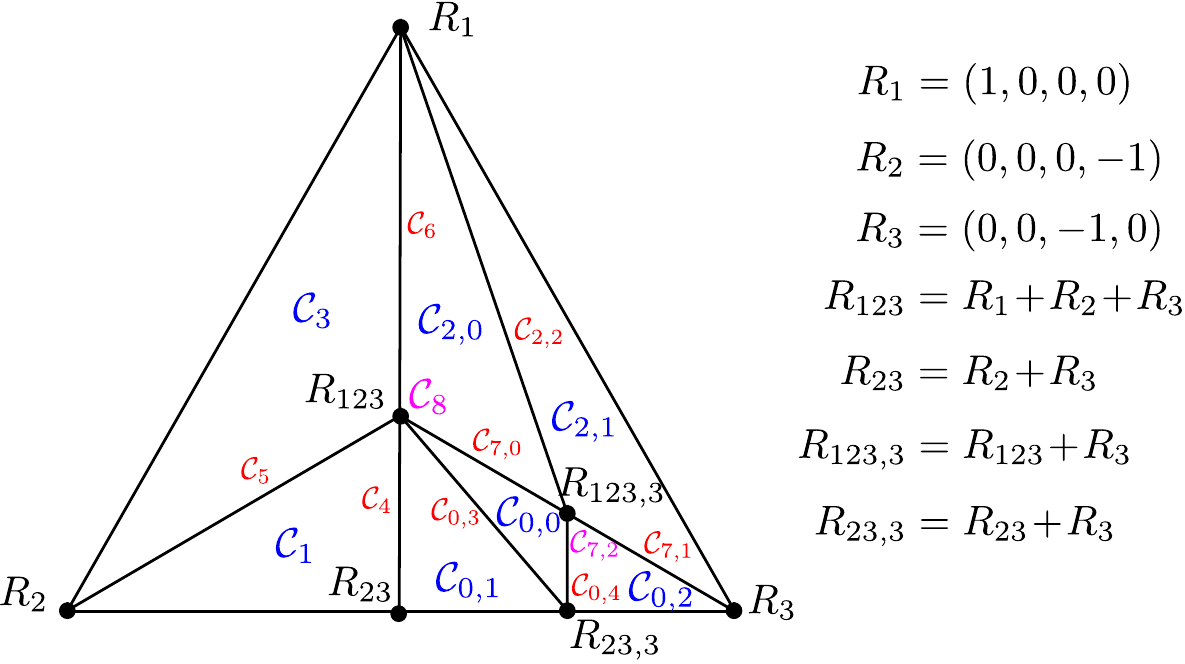}
   \caption{Refined subdivision of the Type (II) cone induced by all possible Newton subdivisions of the $yz$-projection after removing the one-dimensional lineality space. The first index of each cone reflects the label within the subdivision by leading terms of coefficients as in~\autoref{tab:LTerms}.  The blue cones have maximal dimension four, the red cones have dimension three and the purple cone has dimension two.
     \label{fig:SubdivisionThetaCone}}\end{figure}

The condition $\init(\alpha_3)=\init(\alpha_4)$ characterizing the witness Type (II)  region  in~\autoref{tab:CombAndLengthData} suggests a new strategy to determine the combinatorics of $\Trop\,V(I_{g,f})$ by controlling the value of $d_{34}$. We introduce a new variable $\beta_{34} := \beta_3 -\beta_4$ and redefine the third branch point as  $\alpha_3:=(\beta_4 + \beta_{34})^2$,  where $-\val(\beta_{34}) = d_{34}  + \val(\alpha_4)/2 =   -\val(\alpha_3 - \alpha_4)   + \val(\beta_4)$.
The hyperelliptic equation becomes $g(x,y) = y^2 - x(x-\beta_2^2)(x-(\beta_4+\beta_{34})^2)(x-\beta_4^2)(x+\beta_5^2)$, and the lifting $f$ from~\eqref{eq:liftF} of the  tropical polynomial $F$ from~\eqref{eq:F} equals
\begin{equation*}\label{eq:type2Modif}
f(x,y) = y-\beta_5(\beta_4+\beta_{34})\beta_4\, x + \beta_5\, x^2.
  \end{equation*}

The weight vector $\underline{u}\in \RR^4$ encoding the negative valuation of the four parameters equals: 
\begin{equation}\label{eq:valn4Param}
 \underline{u} :=(-\val(\beta_5), -\val(\beta_4), -\val(\beta_{34}), -\val(\beta_2)) = (\ww_5/2,\ww_4/2, d_{34}-\ww_4/2, \ww_2/2).
\end{equation}
We set $u_i = -\val(\beta_i)$ for each $i=5,4,34,2$ and write the coordinates of $\RR^4$ in that order. The Type (II) cone  is then determined by the following inequalities:
\begin{equation}\label{eq:NewTypeIICone}
  u_5 > u_4, \quad u_4> u_{34}, \quad \text{ and } \quad u_{4}>u_2.
  \end{equation}
An easy \sage\ computation reveals that the closure of this cone is spanned by three vectors ($R_1$, $R_2$ and $R_3$ in~\autoref{fig:SubdivisionThetaCone}) and has a one-dimensional lineality space generated by the all-ones vector. We are solely interested in its interior, since its various proper faces correspond to other curve types in $M_2$.

On the algebraic side, the interplay between the combinatorics of $\Trop\,V(I_{g,f})$ and the weight vector $\underline{u}$ is determined by the projection to $\RR^4$ of the Gr\"obner fan of the extended ideal $I_{g,f}K[\beta_5^{\pm},\beta_{4}^{\pm}, \beta_{34}^{\pm},\beta_2^{\pm}, x^{\pm},y^{\pm},z^{\pm}]$. Since the computation of this fan with build-in \sage~functions does not terminate, we turn to~\autoref{lm:projections} and compute $\Trop\,V(I_{g,f})$ by means of the three coordinate projections as we vary $\underline{u}$. In the remainder of this section we describe the interplay between the weight vector $\underline{u}$ and the $(x,z)$- and $(y,z)$-subdivisions.

   \begin{table}
     \centering
         \begin{tabular}{|c|Sc|c|c|}
         \hline
 \textbf{Monomials} & \textbf{Leading Terms} & \textbf{Weights} & \textbf{Cones} \\ 
   \hline 
& $2  \, b_2^2 \, b_5^3$ &  $\ww_2+3\ww_5/2$ & [0, 2, 7]\\ 
$y^4$ &  $2 \, b_{34}^2 \, b_5^3$ & $2d_{34} -\ww_4+3\ww_5/2$ & [1, 3, 5]\\
&     $ 2  \, b_5^3 \, (b_{34}^2 + b_2^2)$ & $ \ww_2 +  3\ww_5/2$ & [4, 6, 8]
\\
\hline $y^3z$ & $-2 \, b_4^2 \, b_5^3$ &  $ \ww_4 +  3\ww_5/2$
		     & all \\ \hline
		      $y^2z^2, yz^3, z^4$  &
		     $ b_5^5$\quad (coeffs 4, -4, 1, resp.) & $5\ww_5/2$
		     & all \\\hline

& $-b_2^2 \,  b_4^4 \,  b_5^4$ & $\ww_2 + 2(\ww_4+\ww_5)$ & [0] \\
& $-b_{34}^2 \,  b_4^4 \,  b_5^4$ & $2(d_{34}+\ww_5)+\ww_4$ & [1]\\
& $b_2^4 \,  b_5^6$ & $2\ww_2 + 3\ww_5$ & [2] \\
& $b_{34}^4 \,  b_5^6$ & $4d_{34}-2\ww_4 + 3\ww_5$ & [3]\\
{$y^3$}& $-b_4^4 \,  b_5^4 \,  (b_{34}^2 + b_2^2)$ & $\ww_2 + 2(\ww_4+\ww_5)$ & [4]\\
& $b_{34}^2 \,  b_5^4 \,  (-b_4^2 + b_5\, b_{34}) \,  (b_4^2 + b_5\, b_{34})$ & $2(d_{34}+\ww_5)+\ww_4$ & [5]\\
& $b_5^6 \,  {(b_{34}^2 + b_2^2)}^2$ & $2\ww_2 + 3\ww_5$ &  [6]\\
& $b_2^2 \,  b_5^4 \,  (-b_4^2 + b_5\, b_2) \,  (b_4^2 + b_5\, b_2)$ & $\ww_2 + 2(\ww_4+\ww_5)$ & [7]\\
& $b_5^4 \,  (b_{34}^2 + b_2^2) \,  (-b_4^4 + b_5^2\, b_{34}^2 + b_5^2 \, b_2^2)$ & $\ww_2 + 2(\ww_4+\ww_5)$ &  [8]
     \\\hline
& $b_2^2 \, b_4^2 \, b_5^6$\; (coeffs 2, -3, 1 resp.) & $\ww_2 + \ww_4+3\ww_5$ & [0, 2, 7] \\
$y^2z, yz^2, z^3$ & $b_{34}^2 \, b_4^2 \, b_5^6$\; (coeffs -6, 9, -3, resp.) & $2d_{34} + 3\ww_5$ &   [1, 3, 5] \\
&      $b_4^2 \, b_5^6 \, (3\,b_{34}^2 - b_2^2)$\; (coeffs -2, 3, -1, resp.) & $\ww_2 + \ww_4+3\ww_5$ &[4, 6, 8]
    \\
\hline		      $y^2$ & $-4 \, b_2^2 \, b_{34}^2 \, b_4^4 \, b_5^7$ & $2d_{34}\!+\!\ww_2\!+\!\ww_4\!+\!7\ww_5/2$
		     & all  \\
\hline	             $y z, z^2$ & $ b_{34}^2 \, b_4^6 \, b_5^7$ \quad (coeffs 2, -1, resp.) & $2d_{34}\!+\!2\ww_4\!+\!7\ww_5/2$ & all \\ 
\hline		     $y, z$  &
$b_2^2 \, b_{34}^2 \, b_4^8 \, b_5^8$\; (coeffs 1, -1, resp.) & $2d_{34}\!+\!\ww_2\!+\!3\ww_4\!+\!4\ww_5$
& all                       \\ \hline\hline 
$x^4$  &
$ -2 \, b_4^2 $ & $2\ww_4$   &       all                       \\\hline
                      & $b_4^4$ & $2\ww_4$ & [0, 1, 4]\\
                      & $-b_2^2 \,  b_5^2$ & $\ww_2+\ww_5$ & [2]\\
&  $-b_{34}^2 \,  b_5^2$ & $ 2d_{34} + \ww_5-\ww_4$ & [3]\\
$x^3$ &  $-(-b_4^2 + b_5\, b_{34}) \,  (b_4^2 + b_5\, b_{34})$ & $2\ww_4$ & [5]\\
                      &  $-b_5^2 \,  (b_{34}^2 + b_2^2)$ & $\ww_2 + \ww_5$ & [6]\\
                      &  $-(-b_4^2 + b_5\, b_2) \,  (b_4^2 + b_5\, b_2)$ & $2\ww_4$ & [7]\\
& $-(-b_4^4 + b_5^2\, b_{34}^2 + b_5^2\, b_2^2)$ & $2\ww_4$ 
                      & [8]
                      \\
                      \hline $x^2$   & $2 \, b_2^2 \, b_4^2 \, b_5^2$  & $\ww_2+\ww_4+\ww_5$& all                      \\\hline         
       \end{tabular}
       \caption{From top to bottom: Expected leading terms for all relevant coefficients of  $h(y,z)$ (14 total) and $\tilde{g}(x,z)$ (three total) on the nine cones $\mathcal{C}_i$ coarsening  the refined subdivision of the Type (II) Cone in~\autoref{fig:SubdivisionThetaCone}. Each $b_i$ is the initial form of the parameter $\beta_i$.\label{tab:LTerms}}
   \end{table}

Following earlier notation, we call $\tilde{g}(x,z)= g(x, z+(\beta_4+\beta_{34})\beta_4\beta_5\,x-\beta_5\,x^2)$ and let $h(y,z)$ be the generator of $I_{g,f}\,\cap K[y,z]$. The latter is determined by an easy elimination ideal computation using \singular, available in the Supplementary material. Its extremal monomials are $y, z, y^5$ and $z^5$. The coefficients of both $\tilde{g}$ and $h$ lie in $\ZZ[\beta_5,\beta_{4},\beta_{34},\beta_2]$. 
The first column of \autoref{tab:LTerms} shows the 17 terms of both polynomials with non-monomial coefficients.
The second column shows the factorization of the leading terms of these non-monomial coefficients  for  each of the nine cones in~\autoref{lm:Subdivision1} and justifies our characteristic assumption on $\resK$. The $\underline{u}$-weights  give the expected heights of all relevant coefficients of $\tilde{g}$ and $h$ (indicated in the third column.) The table also provides the  precise conditions on the initial forms of $\beta_5, \beta_4, \beta_{34}$ and $\beta_2$ under which these heights are lower than expected.

The $(y,z)$- and $(x,z)$-Newton subdivisions of $I_{g,f}$ will be determined by the valuations of these 17 coefficients. The answer will vary with $\underline{u}$ in a piecewise linear fashion. At first glance, the domains of linearity are determined by the common refinement of the  Type (II) cone in $\RR^4$ and the  Gr\"obner fan of the product of all these 17 non-monomial coefficients. The latter has $f$-vector $(1, 21, 54, 35)$, so the refinement is performed by intersecting the Type (II) cone with the 35 chambers in the fan.
The next statement describes this na\"ive subdivision of the Type (II) cone into four triangles determined by the baricenters $R_{123}$ and $R_{23}$ from~\autoref{fig:SubdivisionThetaCone}. Its proof is computational, and the required scripts are available in the Supplementary material.

\begin{lemma}\label{lm:Subdivision1} The Gr\"obner fans  of all 17 non-monomial coefficients of $\tilde{g}$ and $h$ induce a subdivision of the Type (II) cone into nine cones. Following~\autoref{fig:SubdivisionThetaCone} they are:\\
  \begin{minipage}{0.33\textwidth}
    \[\begin{aligned}
    \mathcal{C}_0\!:=& \RR_{> 0}\langle R_{23},\!R_{123},\!R_{3}\rangle\!\oplus\! \rspanone,\\
    \mathcal{C}_3\!:= & \RR_{> 0}\langle R_{1},\!R_{123},\!R_{2}\rangle \oplus \rspanone,\\
    \mathcal{C}_6\!:=  & \RR_{> 0}\langle R_{123}, R_{1}\rangle \oplus \rspanone, \\
    \end{aligned}
    \]
  \end{minipage}
  \begin{minipage}{0.34\textwidth}
    \[\begin{aligned}
    \mathcal{C}_1\!:=& \RR_{> 0}
    \langle R_{23},\! R_{123},\! R_{2}\rangle\! \oplus\! \rspanone,\\
    \mathcal{C}_4\!:= &    \RR_{> 0}\langle R_{23},R_{123}\rangle \oplus \rspanone,\\
    \mathcal{C}_7\!:= &    \RR_{> 0}\langle R_{123},R_{3}\rangle \oplus \rspanone,\\
  \end{aligned}
  \]
    \end{minipage}
  \begin{minipage}[l]{0.25\textwidth}
     \[\begin{aligned}
    \mathcal{C}_2\!:= &\RR_{> 0}\langle R_{1},\!R_{123},\! R_{3}\rangle\! \oplus\! \rspanone,\\
    \mathcal{C}_5\!:= &\RR_{> 0}\langle R_{123}, R_{2}\rangle \oplus \rspanone, \\
    \mathcal{C}_8\!:=   &\RR_{> 0}\langle R_{123}\rangle \oplus \rspanone.
    \end{aligned}\]
  \end{minipage}
  \end{lemma}

In what follows we discuss the combinatorics of the Newton subdivisions of $\tilde{g}$. 
 The next result summarizes our findings, depicted in~\autoref{fig:xzSubdivisions}:
   \begin{proposition}\label{pr:xzNewtonS}
     There are eight combinatorial types of unmarked Newton subdivisions of $\tilde{g}$. The monomial $x^3$ in $\tilde{g}(x,z)$ is the sole responsible for non-generic behavior, which only occurs in the cells $\mathcal{C}_i$ for $i=5,6,7, 8$. 
   \end{proposition}
   \begin{proof}
     By~\autoref{tab:nonFaithfulness}, the Newton subdivision of $\tilde{g}$ is determined by all possible subdivisions of the parallelogram $\mathcal{P}$.  To find the generic subdivision on each cell, we take as a sample weight vector $\underline{u}$ the average of its spanning  rays. We compute an example of parameters $\beta_5,\ldots, \beta_2$ with coordinatewise negative valuation $\underline{u}$ and pick initial forms $b_i =\init(\beta_i)$ ensuring the corresponding leading terms in~\autoref{tab:LTerms} do not vanish. We compute the corresponding plane tropical curve and its  dual subdivision with the \texttt{tropical.lib} package in \singular. All examples and scripts are available in the Supplementary material.
     
  To certify that each generic subdivision is valid on the entire cell, we compute explicit formulas for all the vertices dual to polygons in the subdivision, in terms of the weights of the monomials on $\mathcal{P}$ being maximized (these weights are provided in~\autoref{tab:LTerms}). Finally, the inequalities defining each of the nine cells  confirm that these vertices maximize the same monomials for every weight vector in the given cell.

    To address non-generic behavior on the cells $\mathcal{C}_5,\mathcal{C}_6,\mathcal{C}_7$ and $\mathcal{C}_8$,  we need only to focus on the monomial $x^3$. We list all possible subdivisions of $\mathcal{P}$ that can arise by lowering $x^3$ and construct numerical examples showing which ones are realized. 
     \end{proof}

   \begin{figure}
     \includegraphics[scale=0.35]{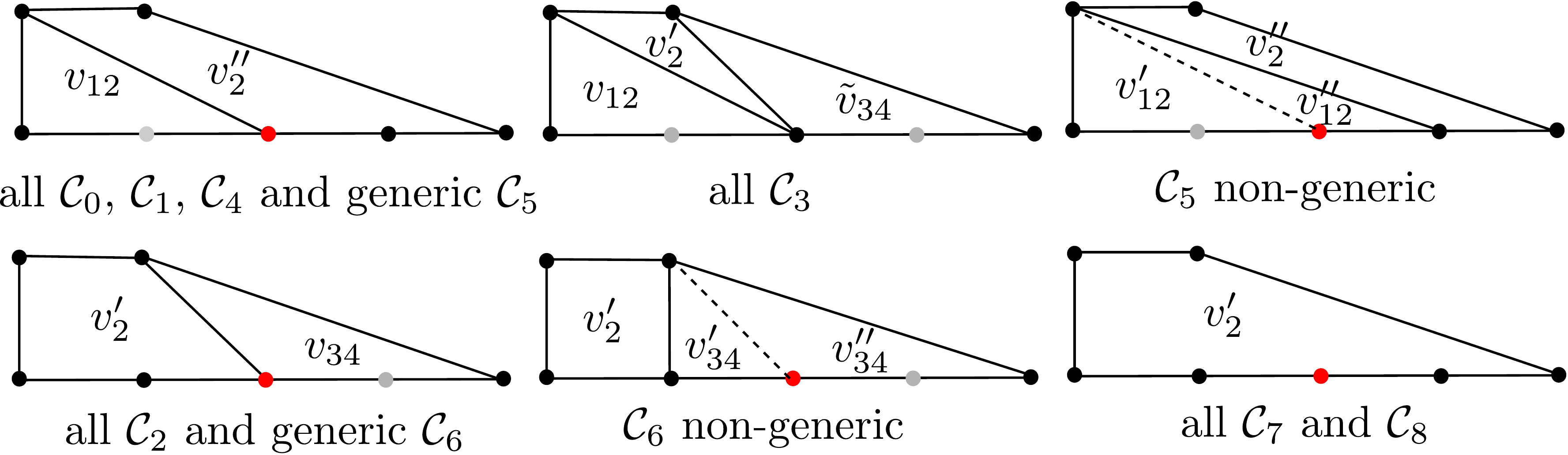}
     \caption{All eight  subdivisions of the parallelogram $\mathcal{P}$ from~\autoref{tab:nonFaithfulness}. The generic and non-generic behavior of $x^3$ on the four relevant cells are indicated by a red dot. The dashed lines correspond to two combinatorial types arising for non-generic initial forms. If absent, both vertices agree with $v_{12}$ and $v_{34}$, accordingly. For non-generic $\mathcal{C}_7$,  $x^3$ is unmarked.\label{fig:xzSubdivisions}}
   \end{figure}

Since  the linear inequalities between the expected heights of each relevant monomial in $h(y,z)$ can vary within each cell, the methods used for $\tilde{g}$ will not suffice to determine all possible  Newton subdivisions of $h$. 
A refined subdivision of the Type (II) cone induced by a subdivision of $\mathcal{C}_0$, $\mathcal{C}_2$ and the relative interior of their common facet $\mathcal{C}_7$  will be required to address this point and the effect of non-generic choices of $\beta$-parameters.

To this end, we construct nine polynomials $h_i$ for $i=0,\ldots, 8$, obtained by replacing each coefficient of $h$ by its leading term on the corresponding cone $\mathcal{C}_i$. We compute the Gr\"obner fan $\mathscr{G}_i$ of each $h_i$ in $\RR^6$,  and intersect each $\mathcal{C}_i$ with the projection of all maximal cells in $\mathscr{G}_i$ to the four $\beta$--coordinates. These calculations are easily performed  since each fan has at most 16 chambers and lineality space $\rspanone$. 
The result of this subdivision process is depicted in~\autoref{fig:SubdivisionThetaCone}.

Next, we describe all possible Newton subdivisions of $h$. As with~\autoref{pr:xzNewtonS}, the proof is computational in nature and requires a careful analysis for non-generic cases.
\begin{proposition}\label{pr:yzNewtonS}
  Each cell in~\autoref{fig:SubdivisionThetaCone} will give rise to one generic subdivision of $h(y,z)$, with further possibilities if genericity conditions are detected in~\autoref{tab:LTerms}. Figures~\ref{fig:Cell0yz} through~\ref{fig:Cell7yz} depict all possible outcomes, grouped conveniently.
  \end{proposition}
Before providing the details of the proof for each cell, we point out some common features of the various subdivisions and clarify notation. In all cases, we only indicate  vertices of $\Trop\,V(I_{g,f})$ rather than false crossings arising from  certain parallelograms (seen, for example, in the subdivision of $Q_1$ in~\autoref{fig:Cell13456yz}.) False crossings may also appear from  a polygon with at least two parallel edges when a vertex in $\sigma_6$ maps to the interior of an edge or leg in $\sigma_4$. This is seen in the polygon $Q_4$ in the same figure: the $YZ$-projection of the vertex $v_{12}$ in $\sigma_6$  lies in the projection of the leg  with direction $(0,0,-1)$ adjacent to the vertex $v_{21}$ in $\sigma_4$.

In addition to these false crossings, the $YZ$-projection has other undesirable effects: we will see vertices in $\sigma_4$ hidden in edges of $\Trop\,V(h)$, overlapping of vertices, as well as higher multiplicity edges and legs coming  in two flavors: \begin{enumerate}[(i)]
  \item 
    Multiplicity one edges and legs inherit higher multiplicities in the $yz$-tropicalization due to the push-forward formula for multiplicities. This occurs for the leg with direction $(2,5,5)$ in $\sigma_3$ adjacent to $v_4$ which inherits multiplicity 5 in $\Trop\,V(h)$.
  \item Two edges or legs (one in $\sigma_4$ and one in $\sigma_6$) overlap in the $yz$-tropicalization, and their multiplicities get added accordingly.  This will always be the case for the edges joining $z^4$ and $y^2z^2$ in all Newton subdivisions of $h$. On the tropical side, this was  observed already in~\autoref{fig:ThetaModif}.
  \item Vertices in $\sigma_4$ lie in relative interiors of edges in $\Trop\,V(h)$. This occurs for the vertex $v_2$ and the cells $\mathcal{C}_{0,i}$: $v_2$ maximizes the edge between $z^4$ and $y^2z^2$ in~\autoref{fig:Cell0yz}.
  \item A vertex in $\sigma_4$  and  one in $\sigma_6$ become the same vertex in $\Trop\,V(h)$. This will be indicated in all figures by equalities between labeling vertices dual to a given polygon.
\end{enumerate}
\begin{proof}[Proof of~\autoref{pr:yzNewtonS}] To determine the generic subdivisions we proceed by direct computation, as in the proof of~\autoref{pr:xzNewtonS}.  The results for each one of the 17 cells are shown in Figures~\ref{fig:Cell0yz} through~\ref{fig:Cell7yz}, where superscripts \emph{gen} indicate generic parameters. 

Next, we discuss the labeling of all polygons in the generic subdivisions.  By~\autoref{lm:projections}, we can place the vertices of $\Trop\,V(I_{g,f})$ we already know from~\autoref{tab:nonFaithfulness} and~\autoref{fig:xzSubdivisions} as duals to polygons or edges in the subdivision. The remaining unlabeled polygons correspond to either false crossings or vertices in $\sigma_4$. The false crossings correspond to parallelograms, and we leave them blank. The others get labeled with blue vertices of the  form $v_{2i}$ with $i=0,1$ to emphasize that they come from $\sigma_4$.

  In order to determine all non-generic subdivisions, we look for vanishing of expected leading terms in~\autoref{tab:LTerms} that will lower the corresponding monomials. In most cases, the resulting special subdivisions (marked with the superscript \emph{sp} on the figures) will differ from the generic ones in only a few polygons. We treat each cell separately to predict these special behaviors and construct numerical examples  to confirm these potential subdivisions do occur.  
  
  We start with the cell $\mathcal{C}_4$. The monomials affected are $y^4$ (if $b_{34}^2 =-b_2^2$), and $y^2, yz^2$ and $z^3$ (if $3b_{34}^2=b_2^2$). From~\autoref{fig:Cell13456yz} we see that lowering any of these four monomials will have no effect on the generic subdivision since these points were already unmarked (the unmarking of $y^4$ was indicated in pink). Therefore, there will be a single Newton subdivision for $\mathcal{C}_4$, namely the generic one.

  Special subdivisions on the cell $\mathcal{C}_5$ are determined by the behavior of $y^3$ whenever $b_4^2 = \pm b_5b_{34}$. This monomial is marked in $Q_5^{gen}$, as seen in~\autoref{fig:Cell13456yz}. When the height of this monomial is reduced, an edge between $yz$ and $y^4$ arises. Furthermore, with the exception of $y^3$, the heights of all points in the triangle $T$ with vertices $y$, $yz$ and $y^4$ are known from~\autoref{tab:LTerms}. Depending on the height of $y^3$, there will be two possible subdivisions: either $T$ is a polygon in the subdivision, or it gets divided along an edge between $y^3$ and $yz$. Numerical examples  confirm that both cases do occur. 

  The cell $\mathcal{C}_6$ has the same defining genericity conditions as $\mathcal{C}_4$, with the addition that $y^4$ drops height whenever $y^3$ does. Since $y^2$ is marked, the lowering of the monomials $z^3, y^2z$ and $yz$ will not change the subdivision, so we can disregard this genericity condition, and only require $b_{34}^2 = -b_2^2$ for special behavior.

  Furthermore, since $y^3$ and $y^4$ are both marked in $Q_6^{gen}$ as we see in~\autoref{fig:Cell13456yz}, for special parameters, an edge joining $y^2$ and $y^2z^2$ will appear and give rise to a triangle $T$ with vertices $y^2, y^2z^2$ and $y^5$. We claim that $T$ can only be further subdivided by an edge between $y^2z^2$ and $y^3$ leading to the two possibilities for $Q_6^{sp}$ shown in the figure. The reason for this lies in \autoref{lm:projections} and~\autoref{pr:xzNewtonS}. Since $v_{12}=(2\ww_4+\ww_5/2, \ww_2+\ww_4+\ww_5/2)$, this vertex lies in $\sigma_4\cap \sigma_5$. Therefore, all cells in a subdivision of $T$ will come from vertices in $\sigma_5$, namely the vertices  $v_{34}'$ and $v_{34}''$ in~\autoref{fig:xzSubdivisions}. Unless these two agree, the edge between them in $\Trop\,V(\tilde{g})$ is dual to an edge with slope $-2$ in a subdivision of $T$. By convexity, there is only one option for such an edge.

  The analysis of non-genericity for the cells $\mathcal{C}_{7,i}$ with $i=0,1,2$ is simpler that earlier cases since only the monomial $y^3$ imposes restrictions on the parameters. Only if $b_4^2 = \pm b_2b_5$ this monomial will be lower than expected. If so, due to the marking of $y^4$ in the polygon $\mathcal{B}^{gen}$ from~\autoref{fig:Cell7yz},  an edge between $y^2z$ and $y^4$ will appear for special parameters. Depending on the height of $y^3$, we will have one extra edge joining $y^2z$ and $y^3$. This yields the two possible configurations $\mathcal{B}^{sp}$ in the figure.

  Finally, we discuss the subdivisions for non-generic parameters coming from $\mathcal{C}_8$. The same six monomials from $\mathcal{C}_6$  are responsible for special choices of parameters. Since these six monomials were not vertices in the generic subdivision in~\autoref{fig:Cell28yz}, lowering them will not alter the subdivision, except for unmarking $y^3$ and $y^4$ accordingly. Thus, the generic and the special Newton subdivisions agree for $\mathcal{C}_8$. This concludes our proof.
\end{proof}

\begin{figure}
  \centering
\includegraphics[scale=0.2]{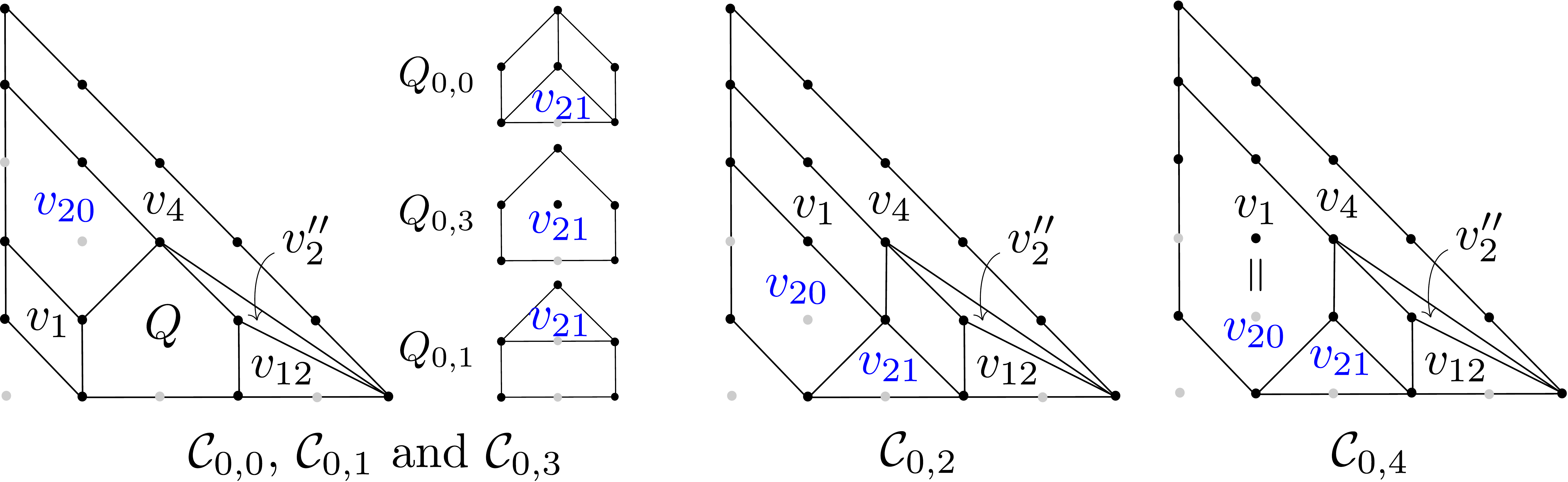}
\caption{All possible subdivisions corresponding to weight vectors in the cells $\mathcal{C}_{0,i}$ for $i=0,\ldots, 4$. The polygon $Q_{0,i}$ indicates the subdivision of the polygon $Q$ on $\mathcal{C}_{0,i}$.  Unlabeled polygons correspond to false crossings. Blue vertices come from $\sigma_4$. The notation on the remaining vertices is compatible with that of~\autoref{tab:nonFaithfulness} and~\autoref{fig:xzSubdivisions}. \label{fig:Cell0yz}}
\end{figure}

\begin{figure}
  \centering
\includegraphics[scale=0.2]{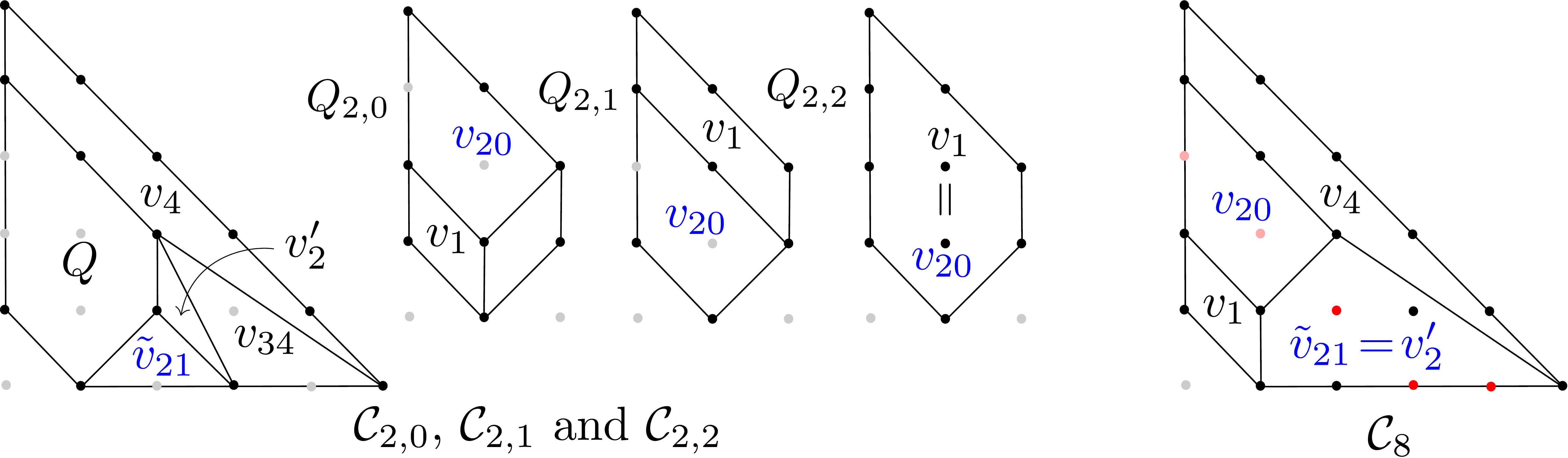}
\caption{All possible subdivisions for the cells $\mathcal{C}_{2,i}$ for $i=0,1, 2$ and $\mathcal{C}_8$.  Red (respectively pink) dots  indicate marked (resp. unmarked) monomials whose behavior varies with the genericity conditions.\label{fig:Cell28yz}}
\end{figure}

\begin{figure}
  \centering
\includegraphics[scale=0.2]{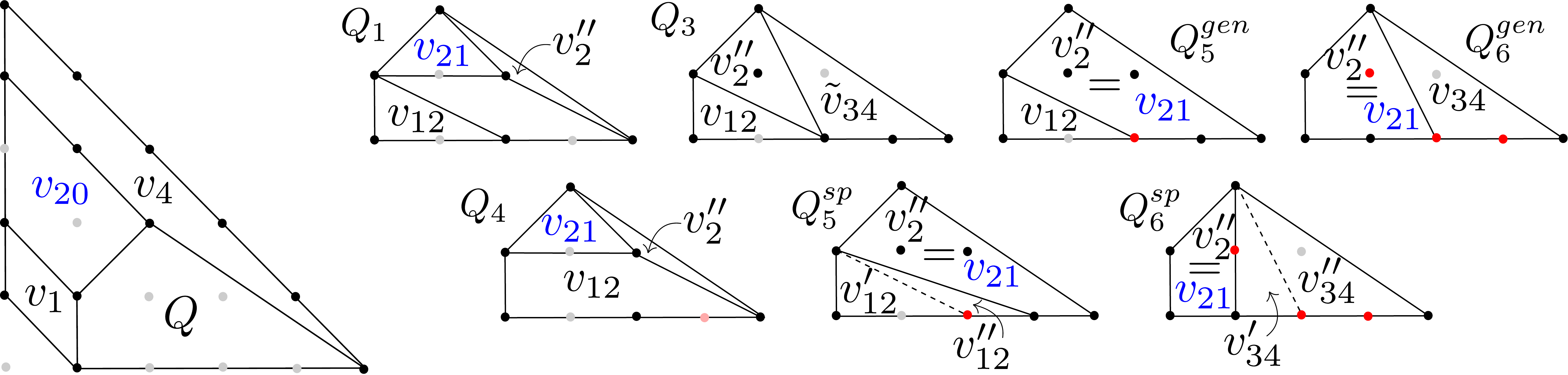}
\caption{All possible subdivisions for the cells $\mathcal{C}_{i}$ for $i=1,3,4,5,6$. The polygon $Q$ gets subdivided differently on each cell. The subscript $gen$ correspond to generic parameters $\beta_5,\beta_4, \beta_{34}$ and $\beta_2$, whereas the superscript $sp$ indicate special ones. As with~\autoref{fig:xzSubdivisions}, dotted lines correspond to extra possible subdivisions. When absent, the corresponding vertices agree. 
  \label{fig:Cell13456yz}}
\end{figure}

\begin{figure}
  \centering
\includegraphics[scale=0.2]{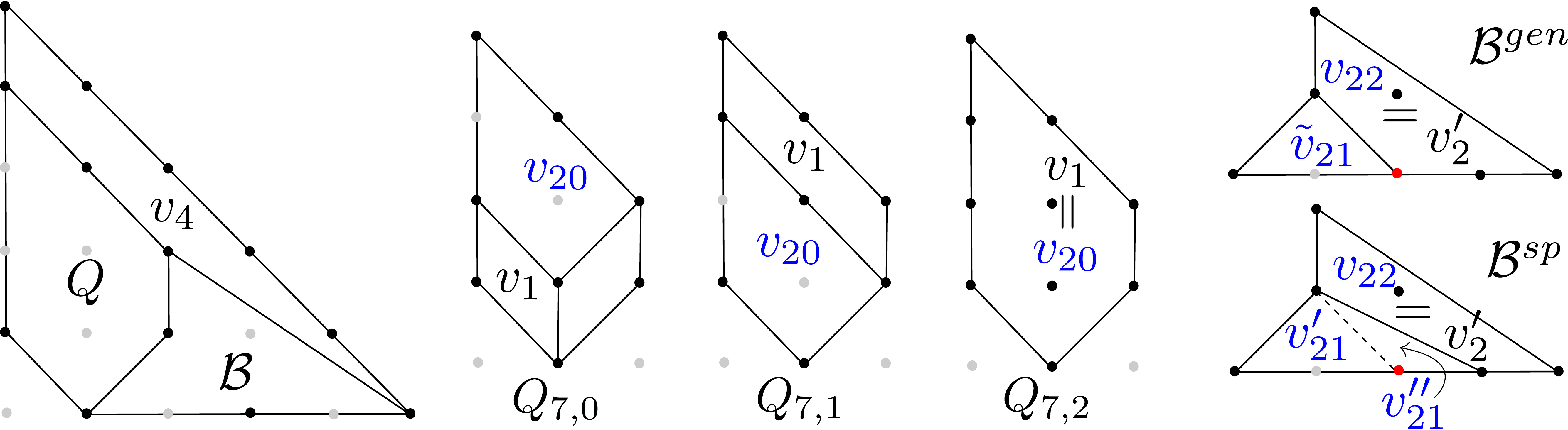}
\caption{All possible subdivisions for  the cells $\mathcal{C}_{7,i}$ for $i=0,1, 2$.\label{fig:Cell7yz}}
\end{figure}
Formulas for all vertices in $\Trop\,V(I_{g,f})$ can be given in terms of the vertices $v_1,v_2,v_4$ from~\eqref{eq:vertices} (where $\ww_3=\ww_4$) and the weight vector $(\ww_5,\ww_4,d_{34}, \ww_2)$ from~\autoref{tab:CombAndLengthData}:
\begin{equation}\label{eq:verticesTypeII}
  \begin{minipage}{0.4\textwidth}
  \[\begin{aligned}
    v_{12}& := v_1 + (\ww_5-\ww_2)/2\,(1,1,0),\\
    v_{34} & :=v_4 - (\ww_5-\ww_2)/2\,(1,2,3),\\
    \tilde{v}_{34} & := v_4-(\ww_5+\ww_4-2d_{34})/2\,(1,2,3),\\
  v_{20} & := v_2-(\ww_4-d_{34})(0,1,1),\\
  v_{21} & := v_2'' - (3\ww_4+2\ww_5-2d_{34})/2\,(0,1,1),\\
  \tilde{v}_{21} & := v_2'-(\ww_4+\ww_2-2d_{34})/2\,(0,1,1),\\
  v_{2}'& :=v_2 -(\ww_4-\ww_2) (0,0,1),
  \end{aligned}\]
\end{minipage}
  \begin{minipage}{0.45\textwidth}
    \[\begin{aligned}
    v_{2}''& :=v_2 -(\ww_5-\ww_4) (0,0,1),\\
  v_{12}' & := v_{12} - \varepsilon'(1,1,0),\\
  v_{12}'' & := v_2''-2\varepsilon'(1,1,3),\\
   v_{34}' & := v_{2}'+2\varepsilon(1,2,0) , \\
  v_{34}''& := v_{34}-\varepsilon (1,2,3),\\
  v_{21}' & := v_{21} + \varepsilon''(0,1,-1),\\
  v_{21}'' & := v_{2}' -2 \varepsilon''(0,1,2),
    \end{aligned}\]
\end{minipage}
\end{equation}

\noindent
where $0\leq \varepsilon\leq (\ww_5+\ww_2-2\ww_4)/6$, $0\leq \varepsilon' \leq  (2\ww_4-\ww_5-\ww_2)/6$ and $0\leq \varepsilon''\leq (\ww_5+2d_{34}-\ww_4)/6$. Whenever the value of $\varepsilon$ is maximal, we get $v_{34}'=v_{34}''$. Similarly, when $\varepsilon'$ and $\varepsilon''$ are maximal, it follows that  $v_{12}'=v_{12}''$ and $v_{21}'=v_{21}''$, respectively.

\begin{proof}[Proof of~\autoref{thm:allCombTypes}.]
  The result follows by combining~\autoref{lm:projections} with Propositions~\ref{pr:xzNewtonS} and~\ref{pr:yzNewtonS}. It is worth noticing that  $\mathcal{C}_{0,i}$, $\mathcal{C}_1$, and $\mathcal{C}_4$ give tropical curves in $\RR^3$ with the same combinatorial type (indicated in~\autoref{fig:allCombTypesTypeII} by the cell $\mathcal{C}_{014}$). ~\autoref{fig:ThetaModif}  corresponds to a graph in $\mathcal{C}_{014}$.
Each special configuration leads to two cells $\mathcal{C}_i^{sp}$ and $\mathcal{C}_i^{sp_2}$ for $i=5,6,7$. The latter is obtained when $v_{12}'=v_{12}''$, $v_{34}'=v_{34}''$ and $v_{21}'=v_{21}''$, respectively.
\end{proof}

\begin{figure}
\centering  \includegraphics[scale=0.25]{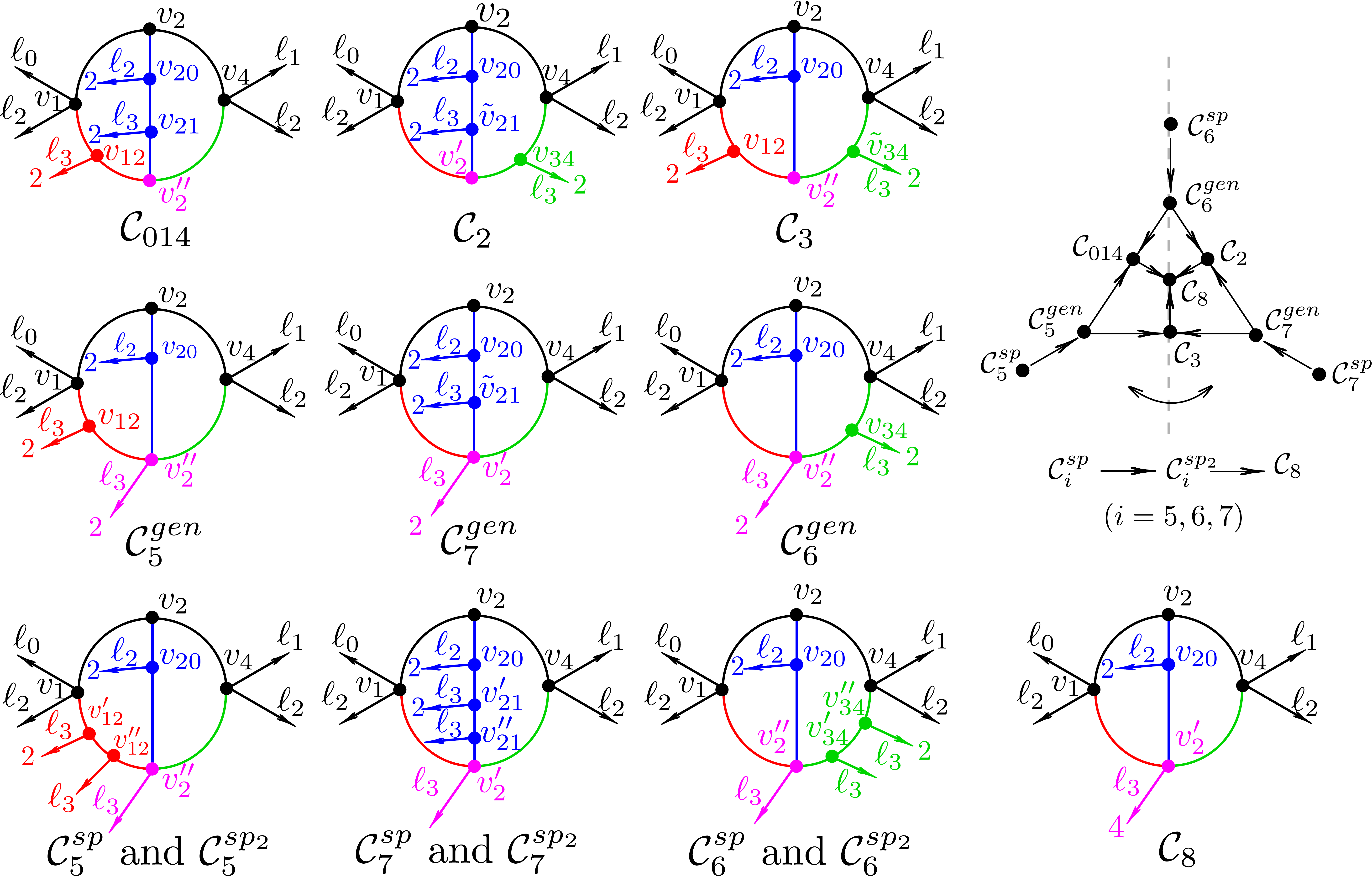}
  \caption{All combinatorial types of $\Trop\,V(I_{g,f})$ and their (symmetric) poset of specializations, where $\mathcal{C}_i^{sp_2}$ is obtained from $\mathcal{C}_i^{sp}$ by making the two  vertices $v_{kl}'$ and $v_{kl}''$ agree (the adjacent leg $\ell_3$ has  multiplicity 3).  The four leg directions are $\ell_0=-(2,1,1)$, $\ell_1= (2,2,5)$, $\ell_2=-e_2$ and $\ell_3=-e_3$.\label{fig:allCombTypesTypeII}}
  \end{figure}

\smallskip

A simple computation shows that $\init_{v}(I_{g,f})$ is reduced and irreducible for all vertices and edges in $\Trop\,V(I_{g,f})$. We conclude that the tropical skeleton is isometric to the minimal Berkovich skeleton, as predicted by~\autoref{thm:faithfulRe-embedding}. Faithfulness at the level of the extended skeleta can be achieved via the vertical modification~\eqref{eq:fiathfulExtendedIII} as in Type (III).

\begin{example}[\autoref{ex:thetaC} revisited]\label{ex:ExTypeII} As was shown in~\autoref{fig:ThetaModif}, the curve from~\autoref{ex:thetaC} is of Type (II). It lies in the witness region $\Omega^{(II)}$  with branch points
  \[\alpha_1=\infty,\; \alpha_2 = (3\,t^5)^2,\; \alpha_3 = (11\,t^2 + 5\,t^7)^2,\; \alpha_4 = (11\,t^2)^2, \; \alpha_5=(1+t^2)^2 \text{ and } \alpha_6=0.
  \]
  By construction, we have natural choices for square-roots of the relevant branch points, namely $\beta_2 = 3\, t^5$, $\beta_3 =11\,t^2 + 5\,t^7$, $\beta_4 = 11\,t^2$ and $\beta_5= 1+t^2$.  We re-embed our na\"ive tropicalization via the following algebraic lift from~\eqref{eq:liftF} of $F =\max\{Y, -4 + X, 2X\}$:
  \[f(x,y) = y - 11\,t^2(1+t^2)(11\,t^2+5\,y^7) \,x + (1+t^2)\,x^2.\]
  Since $\beta_{34} = \beta_3-\beta_4 = -5\,t^7$, the weight vector $\underline{u}$ from~\eqref{eq:valn4Param}  becomes
  \[
\underline{u} = (0, -2, -7, -5) = -2\, \mathbf{1} + 5/3 \, R_{123} + 1/3 \, R_{123,3} + 4/3\, R_{23,3} 
\]
so it lies in the cell $\mathcal{C}_{0,0}$ from~\autoref{fig:SubdivisionThetaCone}. The top left graph in  \autoref{fig:allCombTypesTypeII} shows the tropical curve  $\Trop\, V(I_{g,f})$, in agreement with~\autoref{fig:ThetaModif}.
Since $\underline{\omega}=(\ww_5,\ww_4,d_{34}, \ww_2) = (0, -4, -9, -10)$, expression \eqref{eq:vertices} gives us the vertices  $v_1 = ( -10,-14, -14)$, $v_2=v_3 = (-4, -8, -8)$ and $v_4= (0, 0, 0)$. We use~\eqref{eq:verticesTypeII} to determine all remaining vertices: $v_{12}= (-5, -9, -14)$, $v_2''= (-4,-8,-12)$,     $v_{20}= (-4,-13,-13)$, and $v_{21}= (-4,-11,-15)$.
\end{example}
\subsection{Types (V) and (VII)}\label{ss:TypeVandVII} As discussed earlier in this section, these are the only two types of curves whose na\"ive tropicalization is faithful on the minimal skeleton. As~\autoref{fig:5-7Naive} shows, these tropical curves have high-multiplicity legs with direction $(0,-1)$. They are the images of four legs on the source curves in~\autoref{fig:M2BarAndMumford}. For Type (VII), the unique multiplicity four leg adjacent to $v_1$ is the isometric image of the legs of $\Sigma(\mathcal{X})$ marked with $\alpha_2, \ldots, \alpha_5$. For Type (V), the legs marked with $\alpha_2$ and $\alpha_3$ are mapped isometrically onto the multiplicity two leg adjacent to $v_1$, while the legs marked with $\alpha_4$ and $\alpha_5$ are mapped to the corresponding vertical leg adjacent to $v_3$. In both cases, the legs marked with $\alpha_1$ and $\alpha_6$ are mapped isometrically to the legs with directions $(-2,-1)$ and $(2,5)$, respectively.
The next result discusses the behavior of the vertices of the Berkovich skeleta under tropicalization, where $v_1\!=\!(\ww_2, 3\ww_2/2 + \ww_4)$ and $v_4\!=\! (\ww_4, 5\ww_4/2)$:
\begin{lemma}\label{lm:vertexTypeVII} The initial degeneration of the vertex $v_1$ of $\Trop \, V(f)$ for Type (VII) is a smooth genus two curve over $\resK$. The vertex is the image of the unique genus two vertex of the extended skeleton of $\mathcal{X}^{\an}$ under the na\"ive tropicalization map.
  \end{lemma}
\begin{proof}
  A simple computation gives $\init_{v_1}(g) = y^2-x\prod_{i=2}^5(x-\init(\alpha_i))$. \autoref{tab:CombAndLengthData} ensures that this initial degeneration is a genus two hyperelliptic curve branched at six distinct points: 0, $\init(\alpha_2), \ldots, \init(\alpha_5)$ and  $\infty$ in $\PP^1_{\resK}$. Therefore, it is smooth. The second claim follows directly by continuity and the earlier description of the images of all legs.
  \end{proof}
  \begin{figure}
    \begin{minipage}{0.4\textwidth}
    \includegraphics[scale=0.3]{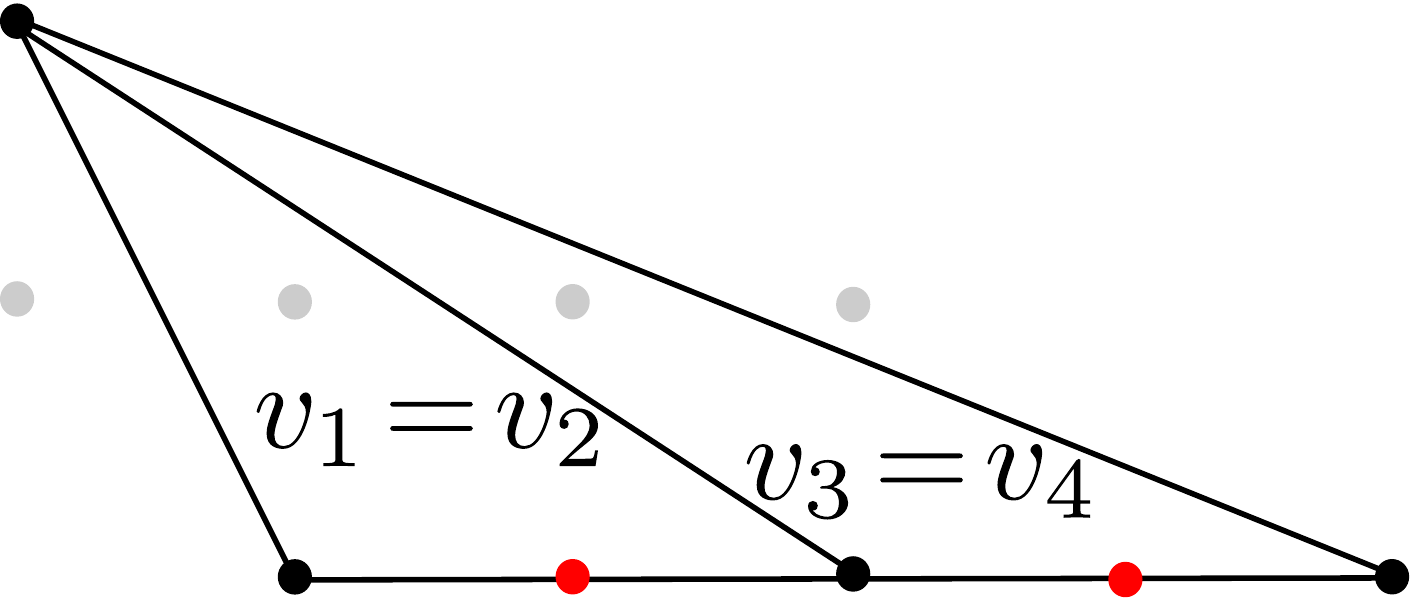}
    \end{minipage}
    \begin{minipage}{0.4\textwidth}
    \includegraphics[scale=0.3]{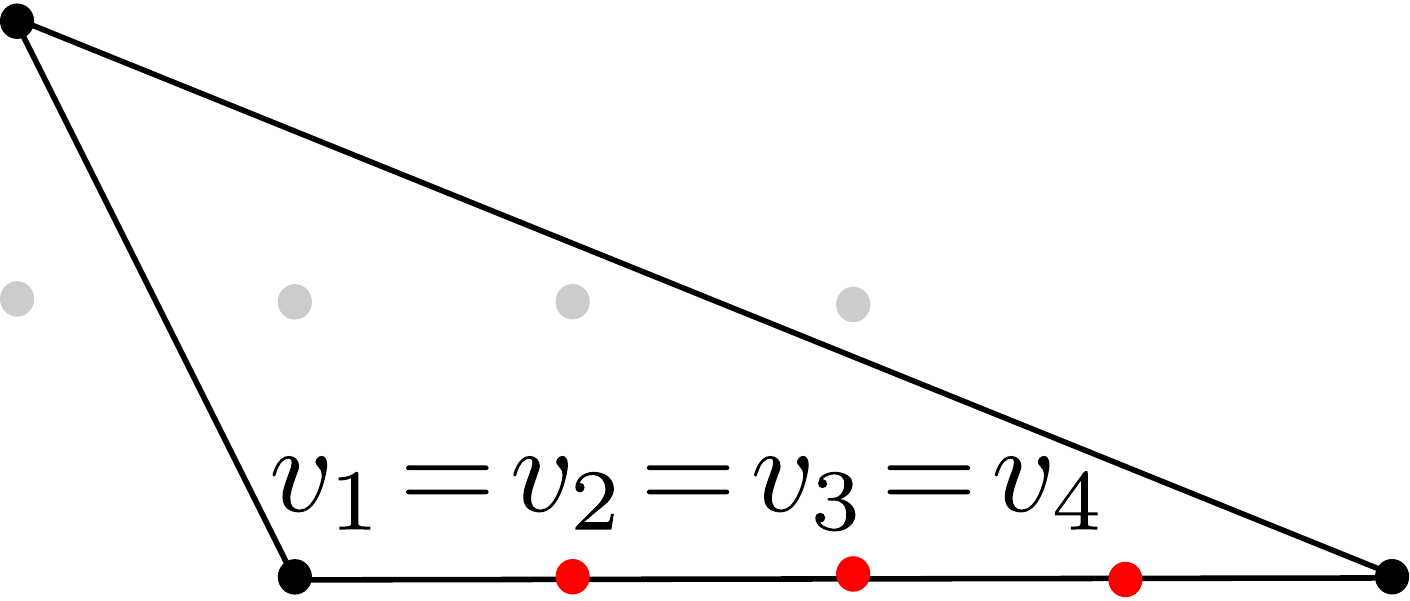}
    \end{minipage}
    \caption{From left to right: Na\"ive tropicalizations for Types (V) and (VII) via Newton subdivisions, following the notation from~\autoref{tab:nonFaithfulness}.\label{fig:5-7Naive}} \end{figure}

  \begin{lemma}\label{lm:verticesTypeV} The initial degenerations of both vertices of $\Trop \, V(g)$ for Type (V) are smooth genus one curves over $\resK$. These vertices are the images of the genus one vertices of the extended skeleton of  $\mathcal{X}^{\an}$ under the na\"ive tropicalization map.
  \end{lemma}
\begin{proof} A direct computation gives $\init_{v_1}(g) = y^2-\init(\alpha_4)\init(\alpha_5)x(x-\init(\alpha_2)(x-\init(\alpha_3))$. 
  By~\autoref{tab:CombAndLengthData} we conclude that $\init_{v_1}(g)$ defines an elliptic curve over $\resK$, since it is a double cover of
  $(\resK^*)^2$ branched at  four distinct points: $0$, $\init(\alpha_2)$, $\init(\alpha_3)$ and $\infty$ in $\PP^1_{\resK}$.
Expression~\eqref{eq:initv3TypeIV} computed  for Type (IV)  is also valid  for Type (V), so  $\init_{v_3}(g)$   is a smooth genus one curve in $(\resK^*)^2$.

  Since the images of the legs marked with $\alpha_2$ and $\alpha_3$ meet at $v_1$, we see that $v_1$ is the image of the corresponding genus one vertex. Similar arguments prove the claim for  $v_3$.
\end{proof}
\begin{remark} \label{rm:jInvTypeV} Techniques from~\autoref{rm:jInvTypeIV} can be used here to show that the vertices of $\Trop \,V(g)$ have genus one. Computations available in the Supplementary material confirm that  the valuations of the $j$-invariants of the restriction of $g$ to the  triangles dual to $v_1$ and $v_3$ are non-negative for any characteristic of $\resK$ other than two.
  \end{remark}

As discussed earlier, the na\"ive tropicalization is not faithful on the extended skeleta. We overcome this  via  vertical modifications along the tropical polynomials $\trop(x-\alpha_2)$ and $\trop(x-\alpha_4)$. Our next result show that these methods yield faithfulness for these tropical curves  in dimensions four and five. The Supplementary material provides examples illustrating this technique for both types.

\begin{proposition}\label{pr:ExtFaithVII} Let $\mathcal{X}$ be of Type (VII). Then, the embedding $\mathcal{X}\hookrightarrow (K^*)^5$ given by 
  \begin{equation}\label{eq:IdealVII}
    J=\langle g, z_i-(x-\alpha_i) \colon \; i=2,3, 4\rangle \subset K[x^{\pm}, y^{\pm}, z_2^{\pm},z_3^{\pm}, z_4^{\pm}],
  \end{equation}
  induces a faithful tropicalization for the extended skeleton with respect to $\alpha_1,\ldots, \alpha_6$. The tropical curve $\Trop\,V(J)$ has one vertex and six legs, and  all tropical multiplicities are one.
\end{proposition}
\begin{proof} The result follows from the Fundamental Theorem of Tropical Geometry~\cite[Theorem 3.2.5]{MSBook} after parameterizing  $\mathcal{X}(K)$ by the maps:
  \begin{equation}\label{eq:paramVII}
    K\ni x\mapsto (x, \pm\Big(x\prod_{i=2}^5(x-\alpha_i)\Big)^{1/2}, x-\alpha_2, x-\alpha_3, x-\alpha_4).
  \end{equation}
  We claim that $\Trop\, V(J)$ has a single  vertex $v=\ww_2(1,5/2, 1, 1, 1)$ and six legs $\ell_i$ ($i=1,\ldots, 6$) with directions $(-2,-1,0,0,0)$, $(0,-1,-2,0,0)$, $(0,-1,0,-2,0)$, $(0,-1,0,0,-2)$, $(0,-1,0,0,0)$ and $(2,5,2,2,2)$. By construction, all tropical multiplicities equal one. Indeed, standard Gr\"obner bases arguments from~\cite[Proposition 2.6.1, Corollary 2.4.10]{MSBook} ensure that the initial degeneration of the first and last legs equal
  \[\init_{\ell_1}(J)\! =\!\langle y^2+x\init(\alpha_5)\prod_{i=2}^4z_i,\, z_j-\init(\alpha_j)\colon j\!=\!2,3,4\rangle,\;\init_{\ell_6}(J)\!=\!\langle y^2-x^5,\, z_j-x\colon  j\!=\!2,3,4\rangle.\]
  \noindent
  Similarly, the initial degenerations with respect to the  legs $\ell_2, \ell_3$ and $\ell_4$ are \[
  \init_{\ell_i}(J)=\langle y^2-xz_2z_3z_4(x-\init(\alpha_5)), x-\init(\alpha_i), z_j-(x-\init(\alpha_j))\colon  j=2,3, 4, j\neq i\rangle\; (i=2,3,4),
  \]
while  $\init_{\ell_5}(J) = \langle xz_2z_3z_4(x-\init(\alpha_5)), z_j-(x-\init(\alpha_j))\colon j=2,3,4\rangle$. We conclude that all six initial degenerations  are reduced and irreducible, so $m_{\trop}(\ell_i)=1$ for all $i=1,\ldots, 6$.
  
  Finally, $\init_v(J) = \langle y^2-xz_2z_3z_4(x-\init(\alpha_5)), z_j-(x-\init(\alpha_j))\colon j=2,3,4\rangle$, so $m_{\trop}(v)=1$ as well.  A direct computation from~\eqref{eq:paramVII} shows that each leg in $\Trop\,V(J)$ is the isometric image of the corresponding marked leg of $\Sigma(\mathcal{X})$ under the new tropicalization. 
  \end{proof}
\begin{proposition}\label{pr:ExtFaithV}  Let $\mathcal{X}$ be of Type (V). Then, the embedding $\mathcal{X}\hookrightarrow (K^*)^4$ given by 
  \begin{equation}\label{eq:IdealV}
    J=\langle g,  z_2-(x-\alpha_2),  z_4-(x-\alpha_4)\rangle \subset K[x^{\pm}, y^{\pm}, z_2^{\pm},z_4^{\pm}],
  \end{equation}
induces a faithful tropicalization of the extended Berkovich skeleton of $\mathcal{X}$ with respect to the six branch points $\alpha_1,\ldots, \alpha_6$.
\end{proposition}
\begin{proof} We use the same techniques from~\autoref{pr:ExtFaithVII} and  apply two successive vertical modifications, starting from  $\trop(x-\alpha_4)$ followed by $\trop(x-\alpha_2)$. In particular,
  \[\init_{v_1}(J) = \langle \init_{v_2}(g(x,y)),  z_2-(x-\alpha_2), z_4+\alpha_4\rangle,\;\;
     \init_{v_3}(J) = \langle \init_{v_3}(g(x,y)),  z_2-x, z_4-(x-\alpha_4)\rangle.\]
  Both initial degenerations are smooth by~\autoref{lm:verticesTypeV}.
  
  The vertical modification techniques described in~\cite[Lemma 2.2, Proposition 2.3]{cue.mar:16} allow us to determine $\Trop\,V(J)$ by means of the planar  $XY$-, $Z_2Y$- and $Z_4Y$-projections.  The ambient tropical surface $\Trop\,V(\langle z_2-(x-\alpha_2),  z_4-(x-\alpha_4)\rangle)$ consists of five two-dimensional cells and it is depicted in~\autoref{fig:TypeVVerticalModif} together with $\Trop\, V(J)$.  As expected,  $\Trop\, V(J)$ consists of two four-valent vertices $v_1=(\ww_2,3\ww_2/2 +\ww_4,\ww_2,\ww_4)$, and $v_3=(\ww_4, 5\ww_4/2, \ww_4, \ww_4)$, joined by an edge with direction $(2,3,2,0)$, with six legs $\ell_1,\ldots, \ell_6$. They are the isometric image of the six marked legs of the extended skeleta and their directions are:
$\ell_1 \!=\! (-2,-1,0,0)$, $\ell_2 \!=\! (0,-1,-2,0)$, $\ell_3 \!=\!  \ell_5 \!= \! (0,-1,0,0)$, $\ell_4 \!=\! (0,-1,0,-2)$ and $\ell_6 \!=\! (2,5,2,2)$. Similar computations to the ones done in the proof of~\autoref{pr:ExtFaithVII} reveal that all tropical multiplicities  equal one.
\end{proof}
\begin{figure}
  \includegraphics[scale=0.35]{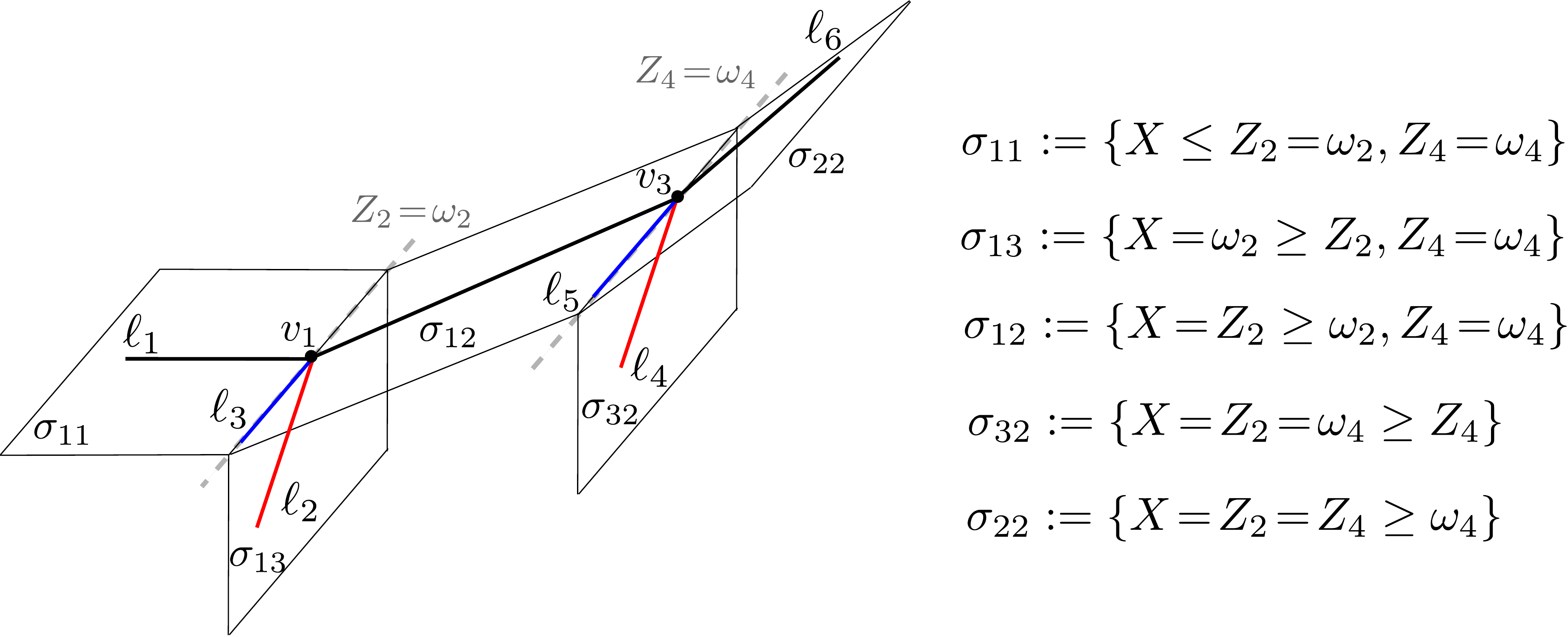}
  \caption{Extended faithfulness for Type (V) via two vertical modifications.\label{fig:TypeVVerticalModif}}
\end{figure}

  \section{Igusa invariants and their tropicalizations}\label{sec:Igusa}
  In 1960, Igusa introduced three invariants $j_1,j_2,j_3$ (called \emph{absolute Igusa invariants}) characterizing isomorphism classes of smooth genus two curves when  $\charF \resK\neq 2$~\cite{Igusa01}. These invariants can be expressed as rational functions (with integer coefficients) in the pairwise differences of the six branch points defining the hyperelliptic equation~\eqref{eq:hyperelliptic}.

  It is worth noticing that a curve $\mathcal{X}$ with Igusa invariants in $K$ need not be defined over  $K$ but rather over a field extension. A concrete algorithm for constructing the curve from these invariants was developed by Mestre~\cite{Mestre90}. For alternative methods involving  Hilbert and Siegel  moduli spaces that are better suited for computations over finite fields, as well as applications to Cryptography we refer to~\cite{LY11}.

  In order to give precise formulas for $j_1,j_2$ and $j_3$ we first construct four homogeneous polynomials $A, B, C$ and $D$ in $u$ and the six branch points $\alpha_1,\ldots, \alpha_6$. Up to an automorphism of $\PP^1$ we may assume none of the branch points lies at infinity. We write
  \begin{equation}\label{eq:diffBrPts}
    \Delta_{ij} := (\alpha_i-\alpha_j)^2 \quad \text{ for }i<j,
  \end{equation}
  and set the homogeneous degree eight polynomial $A$ in $\ZZ[u,\alpha_1, \ldots, \alpha_6]$ to be
  \begin{equation}\label{eq:A}
    A:=u^2     \sum_{\{\{i,j\},\{k,l\},\{m,n\}\}}    \Delta_{ij}\,\Delta_{kl}\,\Delta_{mn}    \;,
  \end{equation}
  where we sum over the 15 tripartitions  of $[6]=\{1,\ldots, 6\}$. Similarly, set
  \begin{equation}\label{eq:B}
B:= u^4  \sum_{\{\{i,j,k\},\{l,m,n\}\}} (\Delta_{ij}\,\Delta_{jk}\,\Delta_{ki})\,(\Delta_{lm}\,\Delta_{mn}\, \Delta_{nl})\;,
  \end{equation}
  where we sum over the ten partitions of $[6]$ into two sets of size three. We define $C$ as 
  \begin{equation}\label{eq:C}
    C:= u^6 
      \sum_{\substack{ \{\{i,j,k\},\{l,m,n\}\}
          \\\{\{i,l\},\{j,m\},\{k,n\}\}}} (\Delta_{ij}\,\Delta_{jk}\,\Delta_{ki})\,(\Delta_{lm}\,\Delta_{mn}\, \Delta_{nl})\,     (\Delta_{il}\,\Delta_{jm}\,\Delta_{kn}),
  \end{equation}
  where we sum over the 60 ways of defining a pair consisting of a partition $[6] =U_1\sqcup U_2$ into two sets  of size 3 and an ordered tripartition where each pair contains exactly one element from $U_1$. We interpret this indexing set as  a labeling of a Type (V) tree $T$ with six leaves, as in~\autoref{tab:CombAndLengthData}. Each set $U_i$ correspond to the leaves attached to each of the two vertices of the tree. The planar embedding of $T$ is relevant since the differences $\Delta_{il}, \Delta_{jm}$ and $\Delta_{kn}$ in each summand of $C$ corresponding to mirror leaves on each side of the tree. 
   This description will be used frequently to compute $-\val(C)$.

  Finally, we let $D$ be the square of the discriminant of the right hand side of~\eqref{eq:hyperelliptic}, i.e.
  \begin{equation}\label{eq:D}
    D:= u^{10}\prod_{1\leq i < j \leq 6} \Delta_{ij}.
  \end{equation}
  The polynomials $A$, $B$ and $C$ have $141$, $1\,531$ and $8\,531$ terms, respectively.
  \begin{definition} The three \emph{Igusa invariants} of the smooth hyperelliptic curve $\mathcal{X}$ equal
\begin{equation}\label{eq:IgusaInv}
j_1(\mathcal{X}):= \frac{A^5}{D}\;, \quad  j_2(\mathcal{X}):= \frac{A^3\,B}{D}\;, \; \text{ and }\quad j_3(\mathcal{X}):= \frac{A^2\,C}{D}.
\end{equation}
Notice that $j_1, j_2, j_3 \in \QQ(\alpha_1,\ldots, \alpha_6)$.
Furthermore, an easy calculation shows that applying an automorphism on the target $\PP^1$ of the hyperelliptic cover, and changing the  equation~\eqref{eq:hyperelliptic} defining $\mathcal{X}$ accordingly, yields the same three invariants.
  \end{definition}

  \smallskip

  Our objective in this section is to study the behavior of the three Igusa invariants under tropicalization and prove the first half of~\autoref{thm:IgusaTropical}. The second half is discussed in~\autoref{sec:new-Igusa}. We start by defining the tropical Igusa functions $j_i^{\trop}:M_2^{\trop}\rightarrow \mathbb{R}_{\geq 0}$. As it occurs with genus one curves and their tropical $j$-invariant~\cite{KMM08}, the construction involves a  genericity assumption. The precise hypersurfaces to avoid for each combinatorial type are discussed in the proof of~\autoref{thm:tropIgusaExpanded} and are listed in the Supplementary material.  
  \begin{definition}\label{def:tropIgusa}
    Given  a genus two abstract tropical curve $ \Gamma$ we define its three \emph{tropical Igusa invariants}  as $j_i^{\trop}(\Gamma ) := -\val(j_i(\mathcal{X}))$ for $i=1,2,3$ for a generic smooth genus two algebraic lift $\mathcal{X}$ of $\Gamma$.
  \end{definition}\noindent
Note that a generic algebraic lift $\mathcal{X}$ of $\Gamma$ is given by generic values $\alpha_i$ which are chosen using \autoref{tab:CombAndLengthData}.  
  By construction, in the non-generic case, the negative valuations will be lower than expected.
  Although it is not evident from the definition, our first result shows that these three tropical Igusa functions indeed depend solely on $\Gamma$, rather than on the isomorphism class of $\mathcal{X}$. Furthermore, they define piecewise linear functions on $M_2^{\trop}$ with domains of linearity given by the seven cones describing all combinatorial types. Here is the precise statement addressing the first half of~\autoref{thm:IgusaTropical}:
  \begin{theorem}\label{thm:tropIgusaExpanded} Let $\mathcal{X}$ be a genus two hyperelliptic curve defined over $K$ with $\charF \resK \neq 2, 3$.
    The tropical Igusa functions equal
    \begin{enumerate}[(i)]
    \item $j_1^{\trop}(\Gamma)\! =\! L_1 + 12L_0 + L_2$, $j_2^{\trop}(\Gamma) \!=\! j_3^{\trop}(\Gamma)\!=\! L_1 + 8 L_0 + L_2$ if $\Gamma$ is  a dumbbell curve, 
      \item $j_1^{\trop}(\Gamma)\! =\! j_2^{\trop}(\Gamma) \!=\! j_3^{\trop}(\Gamma)\! =\! L_1 \!+\! L_0 \!+\!L_2$ if $\Gamma$ is a theta curve,
    \end{enumerate}
    where $L_0,L_1, L_2$ denote the lengths on each curve, as in Figure~\ref{fig:allSkeletaAndMetrics}.    All three formulas remain valid under specialization and yield well-defined piecewise linear maps on the moduli space $M_2^{\trop}$ with domains of linearity corresponding to the seven combinatorial types.
    \end{theorem}

  \begin{remark}\label{rmk:noTropInvariants}  The previous result shows that the three tropical Igusa functions do not characterize tropical curves of genus two, not even within a fixed combinatorial type and should not be considered tropical analogs of Igusa invariants. Indeed, the second and third tropical Igusa functions agree on each cone in $M_2^{\trop}$, and all three agree on the cone of theta curves. In particular, we cannot recover the length data for each tropical curve from these three functions.
  \end{remark}

  We would like to comment on the relation of this result with \cite{hel:16}. In~\cite{Igusa01}, Igusa introduced ten projective invariants of smooth genus two curves, and the three specific quotients $j_1,j_2$ and $j_3$. Whenever these quotients do not vanish, they determine a unique point in $M_2$ (a smooth curve of genus two) over a field of characteristic different than two. In particular, these invariants become coordinates on $M_2$. Our work describes to which extent their tropicalizations fail to be coordinates on $M_2^{\trop}$. Building on earlier work of Liu~\cite{Liu93},~\cite{hel:16} shows that the set of ten invariants  suffices to characterize the type and lengths of the tropicalization $\Gamma$. One should take into account that in \cite{hel:16}, the ten invariants are expressed in terms of the coefficients of the hyperelliptic equation~\eqref{eq:hyperelliptic}, while our three quotients are written in terms of the branch points.

  \begin{proof}[Proof of~\autoref{thm:tropIgusaExpanded}] Consider a generic lift $\mathcal{X}$ of our tropical curve $\Gamma$. This means, a tuple of generic branch points $\alpha_1,\ldots, \alpha_6$ in $K^*$  whose valuations satisfy the conditions described in ~\autoref{tab:CombAndLengthData} and yield the metric graph $\Gamma$.    Furthermore, we assume $\val(\alpha_1)=0$, and $u\!=\!1$ since $u$ plays no role when defining each $j_i(\mathcal{X})$.   
By expression~\eqref{eq:IgusaInv}, the tropical Igusa functions of $\Gamma$  equal:
    \begin{equation}\label{eq:negValIgusa}
     \left\{ \begin{aligned}
        -\val(j_1(\mathcal{X})) & = -5\val(A)+\val(D)\, ,\\
        -\val(j_2(\mathcal{X})) & = -3\val(A)-\val(B) +\val(D)\,,\\
         -\val(j_3(\mathcal{X})) & = -2\val(A)-\val(C) +\val(D).
      \end{aligned}\right.
    \end{equation}

    We treat each invariant separately,     analyzing the contributions of each summand in the definition of the four polynomials $A, B, C, $ and $D$, and checking for potential cancellations of the expected initial terms.
    The proof is completed by discussing the behavior of each maximal cells of $M_2^{\trop}$  separately in two lemmas below.

    The genericity conditions on $\mathcal{X}$ are imposed so that the initial forms of each polynomial have the expected valuation after specializing them at the initial forms of each branch point. The two maximal cells in $M_2^{\trop}$ require no genericity assumptions, since the leading terms of all polynomials involved are monomials.
\smallskip

 \noindent
 \textbf{Type (I) cells:}
 For a Type (I) curve, the negative valuation of each $A, B, C, D$ is obtained by computing the initial term on each of these four polynomials with respect to the weight vector $\underline{\ww}:=(\ww_1,\ldots, \ww_6)\in \RR^6$ with $\ww_1<\ldots<\ww_6$.~\autoref{lm:computeInitialsTypeI} ensures that
 \begin{equation}\label{eq:valsTypeI}
   \begin{minipage}{0.42\textwidth}
     \[\begin{aligned}
   - \val(A) &= 2(\ww_4+\ww_5+\ww_6), \\
   - \val(B) & = 4(\ww_5\!+\!\ww_6)+2(\ww_3\!+\!\ww_4) ,
     \end{aligned}\]
   \end{minipage}
   \begin{minipage}{0.45\textwidth}
     \[\begin{aligned}
-\val(C) & = 6(\ww_5 +\ww_6) +4\ww_4 + 2\ww_3,\\
   -\val(D) & = 2\ww_2 + 4\ww_3 + 6\ww_4 + 8\ww_5 + 10\, \ww_6.
     \end{aligned}\]
   \end{minipage}
 \end{equation}
 Combining these values with~\eqref{eq:negValIgusa}  and the formulas for $L_0, L_1$ and $L_2$ from~\autoref{tab:CombAndLengthData} gives:
\begin{equation*}\label{eq:j1TrTypeI}
  \begin{aligned}
    j_1^{\trop}(\Gamma) & =  10(\ww_4+\ww_5+\ww_6) -( 2\ww_2 + 4\ww_3 + 6\ww_4 + 8\ww_5 + 10 \ww_6) 
    \\
&=4\ww_4+2\ww_5-
    2\ww_2-4\ww_3= 2(\ww_5-\ww_4)+6(\ww_4-\ww_3)+2(\ww_3-\ww_2)\\
    & = L_1 + 12 L_0 + L_2\,,\\
    j_2^{\trop}(\Gamma) & =
    6(\ww_4\!+\!\ww_5\!+\!\ww_6) + 4(\ww_5\!+\!\ww_6)+2(\ww_3\!+\!\ww_4)-( 2\ww_2 \!+\! 4\ww_3 \!+\! 6\ww_4 \!+\! 8\ww_5 \!+\! 10 \ww_6) 
    \\
&=2\ww_5+ 2\ww_4-2\ww_3-2\ww_2= 2(\ww_5-\ww_4)+4(\ww_4-\ww_3)+2(\ww_3-\ww_2)\\
    & = L_1 + 8 L_0 + L_2\,,\\
    j_3^{\trop}(\Gamma) & =
4(\ww_4\!+\!\ww_5\!+\!\ww_6)+6(\ww_6\!+\!\ww_5)\!+4\,\ww_4 +2\,\ww_3- (2\ww_2\!+\!4\ww_3\!+\!6\ww_4\!+\!8\ww_5\!+\!10\ww_6) \\&=   L_1 + 8 L_0 + L_2 = j_2^{\trop}(\Gamma) \,.
  \end{aligned}
\end{equation*}

 \noindent
 \textbf{Type (II) cells:}
Following~\autoref{tab:CombAndLengthData},  the weights for Type (II) curves satisfy
\begin{equation}\label{eq:wwTypeII}
  \ww_1<\ww_2<\ww_3=\ww_4<\ww_5<\ww_6\; \text{ and } -d_{34}:=\val(\alpha_3-\alpha_4)> -\ww_3.
\end{equation}
For this reason, in order to determine the valuations of $A, B, C$ and $D$ we consider the factor $\Delta_{34}$ as a new variable $\alpha_{34}^2$,  and replace each variable $\alpha_4$ by $\alpha_{34}+\alpha_3$ in all four polynomials. A similar strategy was used in~\autoref{ss:Theta}. We denote the new polynomials by
\begin{equation}\label{eq:newPolys}
  A', B', C', D'\in \ZZ[\alpha_1, \alpha_2, \alpha_3, \alpha_{34}, \alpha_5, \alpha_6].
\end{equation}

\noindent
A computation with~\texttt{Sage} (available in the Supplementary material) shows that $A',B'$ and $
C'$ have $177$, $1\,911$ and $11\,745$ terms, respectively.

The weight of the new variable $\alpha_{34}$ equals $d_{34}$.  We replace our weight vector in $\RR^6$ by $\underline{\omega} =(\ww_1,\ww_2, \ww_3, d_{34}, \ww_5, \ww_6)$.  By construction, the negative valuation of each $A, B, C, D$ agrees with that of $A', B', C'$ and $D'$. The later equals the $\underline{\ww}$-weight of the initial form of $A', B', C'$ and $D'$, respectively. \autoref{lm:computeInitialsTypeII} ensures that
 \begin{equation}\label{eq:valsTypeII}
   \begin{minipage}{0.42\textwidth}
     \[\begin{aligned}
   -\val(A)   &= 2(\ww_3+\ww_5+\ww_6)\,, \\
   -\val(B)  & = 4(\ww_3+\ww_5+\ww_6)\,,
     \end{aligned}\]
   \end{minipage}
   \begin{minipage}{0.48\textwidth}
 \[    \begin{aligned}
-\val(C) & = 6(\ww_3+\ww_5 +\ww_6)\,,\\
   -\val(D) & = 2\ww_2 + 8\ww_3 + 2d_{34} + 8\ww_5 + 10\, \ww_6.
     \end{aligned}\]
   \end{minipage}
 \end{equation}
 We conclude that $-5\val(A) = -3\val(A) -2\val(B) = -2\val(A)-3\val(C)$.
The formulas for the lengths $L_0, L_1$, and $L_2$ from~\autoref{tab:CombAndLengthData} yield
\begin{equation*}\label{eq:j1TrTypeII}
  \begin{aligned}
    j_1^{\trop}(\Gamma) & = j_2^{\trop}(\Gamma) = j_3^{\trop}(\Gamma)\\
    & = 10(\ww_3+\ww_5+\ww_6) -( 2\ww_2 + 8\ww_3  +2d_{34}+ 8\ww_5 + 10 \ww_6)   \\
&=2\ww_5+2\ww_3-
     2\ww_2-2d_{34}= 2(\ww_5-\ww_3)+2(\ww_3-\ww_2)+2(\ww_3-d_{34})\\
    & = L_1 + L_0 + L_2\,.
  \end{aligned}
\end{equation*}

\smallskip

 \noindent
 \textbf{Type (III) through (VII)  cells:}
 In order to prove the statement for lower dimensional cells, we first note that the substitution $A, B, C, D$ for $A', B', C', D'$  has no impact when computing their valuations on Type (I). Indeed, the weight $\underline{\ww} \!=\! (\ww_1,\ww_2, \ww_3, d_{34}, \ww_5, \ww_6)\in \RR^6$ satisfies $d_{34} = \ww_4$ and $\init(\alpha_{34}) = \init(\alpha_4)$.

  We fix a lower dimensional cell in $M_2^{\trop}$  and pick the weight vector $\underline{\ww}$ in $\RR^6$, where the fourth entry equals $\ww_4$, as described in~\autoref{tab:CombAndLengthData}. 
Consider a sequence of weight vectors $(\underline{\ww}^{(n)})_{n \in\NN}$ corresponding to a Type (I) curve specializing to $\underline{\ww}$. By continuity and the characterization of Gr\"obner fans of homogeneous polynomials~\cite{ComputingGroebnerFans}, we conclude that the $\underline{\ww}^{(n)}$-initial terms for $A, B, C$ and $D$ are present in the corresponding $\underline{\ww}$-initial terms. Therefore, as long as the initial forms of each polynomial do not vanish after evaluating them at $\init(\underline{\alpha})$,  the formulas for the Tropical Igusa functions on Type (I) remain valid for the lower dimensional types: the valuation of each polynomial is the expected one. The proof involving Type (II) sequences is similar since $\ww_4^{(n)} = d_{34}^{(n)} \to \ww_3$ if we approach a curve of Type (III), (VI) or (VII).
  \end{proof}

  \smallskip
  The following two lemmas are used in the proof of~\autoref{thm:tropIgusaExpanded} as well as in~\autoref{sec:new-Igusa}. They can be verified via~\MacT~computations by choosing appropriate weight vectors in $\RR^6$. The required scripts are available in the Supplementary material.  For completeness, we provide alternative non-computational proofs that help understand the behavior of these polynomials under tropicalization. They justify the need to exclude $\charF \resK=3$.
  \begin{lemma}\label{lm:computeInitialsTypeI} Assume $\charF \resK\neq 2,3$. Given a weight vector $\underline{\ww}=(\ww_1,\ldots, \ww_6)\in \RR^6$ in the relative interior of a Type (I) cell in $M_2^{\trop}$, we have
  \[
  \init_{\underline{\ww}} (A) \!=\! 6 \alpha_4^2\alpha_5^2\alpha_6^2, \,   \init_{\underline{\ww}} (B) \!=\! 4 \alpha_3^2\alpha_4^2\alpha_5^4\alpha_6^4,\,  \init_{\underline{\ww}}(C) \!= \!  8 \alpha_3^2\alpha_4^4\alpha_5^6\alpha_6^6  \text{ and } \init_{\underline{\ww}}(D) \!=\!  \alpha_2^2\alpha_3^4\alpha_4^6\alpha_5^8\alpha_6^{10}.
  \]
  \end{lemma}

\begin{proof}
 By~\autoref{tab:CombAndLengthData},  the weight vector $\underline{\ww}=(\ww_1,\ldots, \ww_6)$ corresponding to Type (I) curves satisfies $\ww_1<\ww_2<\ldots<\ww_6$. Since $D$ is given as a product of all expressions $\Delta_{ij}$ from~\eqref{eq:diffBrPts}, and $\val(\alpha_i) >\val(\alpha_j)$ for $i<j$, we get $\val(\Delta_{ij})= -2\ww_j $ for $i<j$, thus
\begin{equation*}\label{eq:valDTypeI}
  -\!\val(D) = 2\ww_2 + 4\ww_3 + 6\ww_4 + 8\ww_5 + 10 \ww_6, \quad \text{ and } \quad\init_{\underline{\ww}}(D) = \alpha_2^2\alpha_3^4\alpha_4^6\alpha_5^8\alpha_6^{10}.
\end{equation*}

The computation for $\init_{\underline{\ww}}(A)$ is more involved, since  it requires determining the initial term of each summand in $A$ and checking for potential cancellations. Each summand $\Delta_{ij}\Delta_{kl}\Delta_{mn}$ of $A$ in expression~\eqref{eq:A} has valuation $-2(\ww_j+\ww_l+\ww_n)$ for $i<j, k<l, m<n$. The conditions on the parameters $\ww_i$ ensure that the minimal valuation is attained for the six tripartitions of the form $\{\{i,4\}, \{k,5\}, \{m,6\}\}$. The coefficient associated to $\alpha_4^{2}\alpha_5^{2}\alpha_6^2$ on each of these summands equals one, so no cancellations occur and this monomial appears in $A$ with coefficient six. Thus, $\init_{\underline{\ww}}(A) = 6 \,\alpha_4^{2}\alpha_5^{2}\alpha_6^2$ if $\val(6)=0$.

Next, we analyze the summands in $B$ given in~\eqref{eq:B} to determine the initial term of $B$ with respect to the weight vector $\underline{\ww}$. The conditions on $\underline{\ww}$ ensure that the summand indexed by the partition $\{\{i,j,k\}, \{l,m,n\}\}$ has valuation $-2(\ww_j + 2\ww_k + \ww_m+ 2\ww_n)$ if $i\!<\!j\!<\!k$ and $l\!<\!m\!<\!n$. The minimum valuation is achieved when $k,n \in\! \{5,6\}$ and $j,m\!\in\! \{3,4\}$. The corresponding summands are indexed by the four splits
\[
\{1,3,5\} \sqcup \{2,4,6\}\, , \{2,3,5\} \sqcup \{1,4,6\}\, , \{1,4,5\}\sqcup \{2,3,6\}\, \text{ and } \{2,4,5\} \sqcup \{1,3,6\}.
\]
On each summand, the monomial $\alpha_3^2\alpha_4^2\alpha_5^{4}\alpha_6^2$ has coefficient one, so  $\init_{\underline{\ww}}(B) = 4\,\alpha_3^2\alpha_4^2\alpha_5^{4}\alpha_6^2$.

In order to compute $\init_{\underline{\ww}}(C)$ and $-\val(C)$ we use expression~\eqref{eq:C} and analyze the valuation of all its 60 summands. The minimum valuation equals $-(6(\ww_6+\ww_5) +4\ww_4 + 2\ww_3)$. This value is obtained for those indices where each element in the pairs $\{5,6\}$ and $\{3,4\}$ belongs to a different set of the split $\{i,j,k\}\sqcup \{l,m,n\}$. Moreover, the elements $4,5$ and $6$ must lie in different pairs in the tripartition $\{\{i,l\}, \{j,m\}, \{k,n\}\}$. A combinatorial analysis allows us to  assume $1=i<j<k$ and conclude that the summands with minimum valuation correspond to the eight ordered tuples:
\begin{equation}\label{eq:8tuplesForC}
\begin{aligned}
(i,j,k,l,m,n)  = & \;  (1,3,6,4,5,2)\,,\, (1, 4, 5, 6, 2, 3)\,,\,   (1,4,6,5,2,3)\,,\, (1, 3, 5, 4, 6, 2)\,,  \\
  & \; (1,3,6,5,4,2)\,,\,(1, 4, 5, 6, 3, 2)\,,\, (1,4,6,5,3,2)\,,\, (1, 3, 5, 6, 4, 2)\,.
\end{aligned}
\end{equation}
On these summands,  the monomial $\alpha_3^2\alpha_4^4\alpha_5^6\alpha_6^6$ is monic, so  $\init_{\underline{\ww}}(C) = 8 \,\alpha_3^2\alpha_4^4\alpha_5^6\alpha_6^6$. 
  \end{proof}
\smallskip

    \begin{lemma}\label{lm:computeInitialsTypeII} Let $A', B', C', D'\in \ZZ[\alpha_1, \alpha_2, \alpha_3, \alpha_{34},\alpha_5, \alpha_6]$ be the polynomials in~\eqref{eq:newPolys}. Given a weight vector $\underline{\ww}=(\ww_1,\ww_2,\ww_3, d_{34}, \ww_5, \ww_6)\in \RR^6$ inducing a point in the relative interior of a  Type (II) cell in $M_2^{\trop}$, we have
  \begin{equation*} 
  \begin{minipage}{0.45\textwidth}
    \[\begin{aligned}
      \init_{\underline{\ww}} (A') &= 8\,\alpha_3^2\alpha_5^2\alpha_6^2 \,, \\
      \init_{\underline{\ww}} (B') & = 4\,\alpha_3^4 \alpha_5^4\alpha_6^4\,,
    \end{aligned}\]
  \end{minipage}
  \begin{minipage}{0.4\textwidth}
    \[\begin{aligned}
       \init_{\underline{\ww}}(C') &= 8\,\alpha_3^6 \alpha_{5}^6\alpha_6^6\,,\\
      \init_{\underline{\ww}}(D') & =  \alpha_2^2\alpha_3^8 \alpha_{34}^2\alpha_5^8\alpha_6^{10}.
    \end{aligned}\]
  \end{minipage}
  \end{equation*}
\end{lemma}

    \begin{proof}
       We start with $D'$. Since the weight vector $\underline{\ww}$ satisfies~\eqref{eq:wwTypeII}, formula~\eqref{eq:D} implies
  \begin{equation*}\label{eq:valDTypeII}
    -\!\val(\Delta_{ij})= \begin{cases}
      2 d_{34} & \text{ if } (i,j) = (3,4),\\
      2\,\ww_j  & \text{ if } j\neq 4,      \\
       2\, \ww_{3}  & \text{ if } j=4, i <3,
    \end{cases} \qquad \text{ for }i<j,
  \end{equation*}

    \noindent
   because $-\val(\alpha_{34}+\alpha_{3}-\alpha_j) = \ww_{\max\{j,3\}}$ for $j\neq 3, 4$.
    We conclude that $-\val(D') = 2\,\ww_2 + 8\,\ww_3+ 2\,d_{34}+ 8\,\ww_5 + 10\, \ww_6$. Furthermore, the term realizing this valuation is $\alpha_2^2\alpha_3^8\alpha_{34}^2\alpha_5^8\alpha_6^{10}$, hence it equals $\init_{\underline{\ww}}(D')$.

  To compute $\init_{\underline{\ww}}(A')$ we proceed analogously. For each tripartition not involving $\{3,4\}$, the valuation of the corresponding summand equals $-2(\ww_j+\ww_l+\ww_n)$, assuming $i<j$, $k<l$, and $m<n$. As in Type (I), the minimum  is achieved at $-2(\ww_3+\ww_5+\ww_6)$, when $j=3$ or $4$, $l=5$, and $n=6$, namely for the 8 tripartitions
  \[
  \{\{1,*\}, \{2,5\}, \{*,6\}\},\, \{\{1, *\}, \{2,6\}, \{*,5\}\},\, \{\{1, 5\}, \{2,*\}, \{*, 6\}\},\, \{\{1, 6\}, \{2, *\}, \{*, 5\}\}.
  \]
Notice that since $\init_{\underline{\ww}}(\alpha_{34}+\alpha_3) = \init_{\underline{\ww}}(\alpha_3)$, it is easy to verify that the coefficient of $\alpha_3^2\alpha_5^2\alpha_6$ on these eight summands   equals 1. 
  
  On the contrary, if $\{3,4\}$ is a pair in the tripartition (say the middle one), the  valuation of each such summand equals $-2(\ww_j + d_{34}+\ww_n)$, which is strictly larger than $-2(\ww_3+\ww_5+\ww_6)$. We conclude that $\init_{\underline{\ww}}(A') = 8 \,\alpha_3^2\alpha_5^2\alpha_6$.
  
We proceed similarly for the polynomial $B'$, distinguishing between splits where $3$ and $4$ belong to different subsets or not. In the first case, there are four summands realizing the minimum valuation $-2(2\ww_3+2\ww_5+2\ww_6)$, corresponding to the splits where $5$ and $6$ also are in different subsets. They all contribute one monomial $\alpha_3^4\alpha_5^4\alpha_6^4$ to $B'$, each with coefficient 1. 

On the contrary, if $3$ and $4$ lie in the same subset, the minimum valuation for these summands is $-2(\ww_2 + d_{34}+ 2\ww_5+ 2\ww_6)$ and it is obtained when $5$ and $6$ are in different subsets  (as in Type (I)). The conditions on $\underline{\ww}$ ensure that $\ww_2+d_{34} <2\ww_3$, so $\init_{\underline{\ww}}(B') = 4\alpha_3^4\alpha_5^4\alpha_6^4$.

Finally, we compute $\init_{\underline{\ww}}(C')$. For each summand of $C'$ corresponding to a split with $3$ and $4$ in different subsets, the valuation is the same as the one computed for Type (I). The expected valuation is $-6(\ww_6 +\ww_5 +\ww_3)$ and it is attained at the eight tuples below:
\begin{equation*}\label{eq:8tuplesForC'}
\begin{aligned}
(i,j,k,l,m,n)  = & \;  (1,3,6,4,5,2)\,,\, (1, 4, 5, 6, 2, 3)\,,\,   (1,4,6,5,2,3)\,,\, (1, 3, 5, 4, 6, 2)\,,  \\
  & \; (1,3,6,5,2,4)\,,\,(1, 4, 5, 3, 6, 2)\,,\, (1,4,6,3,5,2)\,,\, (1, 3, 5, 6, 2, 4)\,.
\end{aligned}
\end{equation*}
The first group corresponds to the four tuples on the top row of~\eqref{eq:8tuplesForC} since the variable $\alpha_{34}$ does not appear on those summands. The second group correspond to tuples where $4$  is opposed to $5$ or $6$ but this loss is compensated by $3$  winning over $1$ and $2$. Notice that these terms did not contribute for the Type (I) cell. 
Collectively, these eight tuples contribute the monomial $8\,\alpha_3^6\alpha_{5}^6\alpha_6^6$.

For the remaining 24 summands in $C'$, where 3 and 4 lie in the same set, the possible $\underline{\ww}$-initial forms are $\alpha_6^6\alpha_5^4\alpha_{34}^2\alpha_3^4\alpha_2^2$, $\alpha_6^6\alpha_5^6\alpha_{34}^2\alpha_3^2\alpha_2^2$, and $\alpha_6^6\alpha_5^4\alpha_{34}^2\alpha_3^6$. Their valuation is strictly bigger that $-6(\ww_6+\ww_5+\ww_3)$, therefore $\init_{\underline{\ww}}(C') = 8\alpha_3^6\alpha_{5}^6\alpha_6^6$.
\end{proof}

    In the rest of this section, we discuss the behavior of the tropical Igusa invariants when  $\charF\resK = 3$. Notice that in this case, we cannot predict the valuation of the polynomial $A$ on the relative interior of the Type (I) cell, since the initial form of $A$ in~\autoref{lm:computeInitialsTypeI} has a coefficient with non-zero valuation. 
    \begin{theorem}\label{thm:tropIgusaExpandedChar3} Let $\charF\resK\!=\!3$ and  $\Gamma$ be a curve of Type (I), (IV) or (V). Then:
      \begin{enumerate} \item If $1-\ww_4 \leq \ww_3$, then the formulas for all $j_i^{\trop}(\Gamma)$ from~\autoref{thm:tropIgusaExpanded} hold.
        \item If $1-\ww_4 > \ww_3$, then $j_1^{\trop}(\Gamma)  = j_2^{\trop}(\Gamma)=L_1 + 2L_0 + L_2$, whereas  $j_3^{\trop}(\Gamma) = L_1 + 4 L_0 + L_2$.
      \end{enumerate}
      If $\Gamma$ is a (specialization of a) Type (II) curve, the formulas from~\autoref{thm:tropIgusaExpanded} hold.
      \end{theorem}

    \begin{proof} If $\Gamma$ is a Type (II) curve, or a specialization thereof, the formulas for all initial forms in \autoref{lm:computeInitialsTypeII} remain valid in characteristic 3. Therefore, the same genericity assumptions imposed in~\autoref{thm:tropIgusaExpanded} yield the formulas for the tropical Igusa invariants for these curves.

In what remains, we treat the remaining three types: (I), (IV) and (V). As discussed above, the initial form of $A$ will not have a uniform value on each of these cones when $\charF \resK = 3$. We bypass this difficulty by writing $A$ as an integer combination of four polynomials with coefficients $\pm 1$ and disjoint supports, comparing the valuation of their initial forms, and considering possible ties and cancellations. A  calculation  available on the Supplementary material yields  $A = 4A_4 + 6A_6 + 12A_{12} + 120A_{120}$. Any $\underline{\omega}$ in the relative interior of the Type (I), (IV) or (V) cells gives
    \begin{equation}\label{eq:char3TypeI}
  \begin{minipage}{0.45\textwidth}
    \[\begin{aligned}
\init_{\underline{\omega}}(A_4) & = -\alpha_6^2\alpha_5^2\alpha_4\alpha_3\,,\\  \init_{\underline{\omega}}(A_6) &= \;\;\;\alpha_6^2\alpha_5^2\alpha_4^2\,,
    \end{aligned}\]
  \end{minipage}
  \begin{minipage}{0.45\textwidth}
    \[\begin{aligned}
    \init_{\underline{\omega}}(A_{12})\;\, & = \;\;\;\alpha_6^2\alpha_5\alpha_4\alpha_3\alpha_2\,, \\ \init_{\underline{\omega}}(A_{120}) & = -\alpha_6\alpha_5\alpha_4\alpha_3\alpha_2\alpha_1.
    \end{aligned}\]
  \end{minipage}
    \end{equation}
    Since $\ww_1<\ww_2<\ldots<\ww_6$ in Type (I), we conclude that $\init_{\underline{\omega}}(A_6)>\init_{\underline{\omega}}(A_{12}), \init_{\underline{\omega}}(A_{120})$ so we need only compare the weights of $A_4$ and $6A_6$. There are two cases to analyze: 
    \begin{description}
    \item[\textbf{Case 1}] If $1-\ww_4\leq -\ww_3$, our genericity assumptions ensure that $\val(6\alpha_4-4\alpha_3) = 1-\ww_4$. Thus $\val(A)$ is  the one predicted in~\autoref{lm:computeInitialsTypeI}. The formulas for the tropical Igusa invariants described in~\autoref{thm:tropIgusaExpanded} remain valid in this setting.

    \item[\textbf{Case 2}] If $1-\ww_4>-\ww_3$, we conclude that $\init_{\underline{\omega}}(A) = 4\init_{\underline{\omega}}(A_4)$, and so $-\val(A)=2\ww_6+2\ww_5+\ww_4+\ww_3$. The expressions for $\val(B)$, $\val(C)$, and $\val(D)$ obtained from~\autoref{lm:computeInitialsTypeI} and arithmetic manipulations as in the proof of~\autoref{thm:tropIgusaExpanded} yield the desired expressions for the Igusa invariants on the Type (I) cone.\qedhere
    \end{description}
    \end{proof}
    \begin{remark} If $1-\ww_4=-\ww_3$ and $\init(6\alpha_4)=\init(-4\alpha_3)$ the above methods do not allow us to compute $-\val(A)$. We bypass this difficulty by using three Laurent monomials in the 3 Igusa invariants where the polynomial $A$ is canceled out.~\autoref{lm:computeInitialsTypeI} yields:
      \begin{equation*}\label{eq:case3char3}
        \begin{aligned}
          5j_2^{\trop}(\Gamma) - 3 j_1^{\trop}(\Gamma)  & = 2\val(D)-5\val(B) = 2L_1 + 4L_0 + 2L_2\,,\\
          3j_3^{\trop}(\Gamma) - 2 j_1^{\trop}(\Gamma)  & = 3\val(D) - 5 \val(C) = 3L_1 + 16L_0 + 3L_2\,,\\
          3j_3^{\trop}(\Gamma) - 2 j_2^{\trop}(\Gamma)  & = \val(D) + 2\val(B) -3\val(C)  = L_1 + 8L_0 + L_2.          
        \end{aligned}
      \end{equation*}
    The matrix describing the three integer linear combination of $j_1^{\trop}$, $j_2^{\trop}$ and $j_3^{\trop}$ has rank two, so we can only express the last two invariants in terms of  $j_1^{\trop}$.     
    \end{remark}

\section{A new Igusa invariant}\label{sec:new-Igusa}

As was shown in~\autoref{rmk:noTropInvariants}, the fact that the tropical Igusa functions do not yield coordinates on $M_2^{\trop}$ raises a natural question: can we replace $j_1, j_2, j_3$ by an alternative set of three algebraic invariants better suited for tropicalization? Given the expressions  in~\autoref{thm:tropIgusaExpanded} we propose to replace $j_3$ with a linear expression  in $j_1,j_2, j_3$ whose initial form corresponding to a weight vector of Type (I) or (II) appears as a result of a cancellation in the initial forms of the $j_i$'s. In other words, we aim to compute a {Khovanskii basis} of the ring of invariants of $M_2$.

The computation of $\val(A), \val(B)$ and $\val(C)$ on Types (I) and (II) in expressions~\eqref{eq:valsTypeI} and~\eqref{eq:valsTypeII} gives the  linear relation
\begin{equation*}\label{eq:findLambda}
  -\val(A)-\val(B) = -\val(C).
\end{equation*}
Therefore, a cancellation might be produced among leading terms via the expressions
\begin{equation}\label{eq:Qmu}
  Q_{\lambda} := A\,B-\lambda C \quad \text{(for Type (I)) }\quad \text{ and }\quad   Q_{\lambda}' := A'\,B'-\lambda C' \quad \text{(for Type (II)) },  
\end{equation}
for suitable  $\lambda \in K^*$.
For generic choices of $\lambda$, a \sage\, calculation shows that $Q_\lambda\in\QQ[\alpha_1, \alpha_2, \alpha_3, \alpha_4, \alpha_5, \alpha_6]$ has $12\,567$ terms, whereas $Q_{\lambda}'\in \QQ[\alpha_1, \alpha_2, \alpha_3, \alpha_{34}, \alpha_5, \alpha_6]$ has $11\,891
$.

By Lemmas~\ref{lm:computeInitialsTypeI} and~\ref{lm:computeInitialsTypeII} we know that for $\charF \resK \neq 2,3$:
\begin{equation}\label{eq:lambdas}
\begin{aligned}
  \text{ On Type (I): }\,\,&\quad \init_{\underline{\ww}}(A)\;\init_{\underline{\ww}}(B)\;
   = 24\, \alpha_3^2\alpha_4^4\alpha_5^6\alpha_6^6 = 3\,  \init_{\underline{\ww}}(C) \qquad\; \text{ so } \quad \lambda = 3\,. \\
  \text{On Type (II): }  &\quad \init_{\underline{\ww}'}(A')\, \init_{\underline{\ww}'}(B')  = 32\,\alpha_3^6\alpha_5^6\alpha_6^6 \; = 4 \,  \init_{\underline{\ww}'}(C') \qquad \text{ so } \quad \lambda = 4.
\end{aligned}
\end{equation} 
These relations shows which values of $\lambda$ will produce cancellations between the $\underline{\omega}$-leading terms of $AB$ and $C$ in $Q_{\lambda}$ and $Q_{\lambda}'$. This choice yields a new Igusa invariant in Type (I):
\[
j_3':=\frac{Q_3A^2}{D} = \frac{A^3B}{D} - 3\frac{A^2C}{D}=j_2 - 3 j_3.
\]
The tropicalization of $j_3'$ equals $-\val(j_3)$ and it is determined by the $\underline{\omega}$-initial form of $Q_3$. A \MacT~computation finds the initial terms:
\[
\init_{\underline{\ww}}(Q_3)= 8\,\alpha_6^6\alpha_5^6\alpha_4^3\alpha_3^3,\quad \text{so } -\val(Q_3) = 6(\ww_6 +\ww_5) + 3(\ww_4 +\ww_3).
\]
Combining this expression with~\eqref{eq:valsTypeI} and the length formula from~\autoref{tab:CombAndLengthData} yields
\begin{equation*}\label{eq:j3prime}
    \begin{aligned}
    j_3'^{\;\trop} & =  6(\ww_6 +\ww_5) + 3(\ww_4 +\ww_3) + 4(\ww_4+\ww_5+\ww_6) - (2\ww_2 + 4\ww_3 + 6\ww_4 + 8\ww_5 + 10\, \ww_6)\\
    & = 2\ww_5 + \ww_4 - \ww_3 -2\ww_2 = 2(\ww_5-\ww_4) + 3(\ww_4-\ww_3) + 2(\ww_3-\ww_2) = L_1 + 6 L_0 + L_2\,.
  \end{aligned}
\end{equation*}

\noindent
The new function     $j_3'^{\;\trop}$  fails to provide new length data for  Type (I) curves. For this reason, we turn to the Type (II) cell and work with $Q_4$, as predicted by~\eqref{eq:Qmu}. We set:
\begin{equation}\label{eq:j4}
j_4:=\frac{Q_4A^2}{D} = \frac{A^3B}{D} - 4\frac{A^2C}{D}=j_2 - 4 j_3.
  \end{equation}
 By construction $j_4^{\trop} = j_3^{\trop}$ on Type (I) curves if $\charF \resK \neq 2,3$.  

 Since we are interested in the behavior of $j_4^{\trop}$ on the Type (II) cell, we work with $Q_4'$ instead of $Q_4$.   A \sage~calculation reveals that $Q_4'\in \QQ[\alpha_1, \alpha_2, \alpha_3, \alpha_{34}, \alpha_5, \alpha_6]$ has $11\,379$ terms. The possible weight vectors $\underline{\ww}=(\ww_1,\ww_2,\ww_3,d_{34},\ww_5,\ww_6)\in \RR^6$ giving  Type (II) curves form  a six-dimensional open cone in $\RR^6$, whose closure we denote by $\Theta$.

  The possible valuations of $Q_4$ are determined by the Gr\"obner fan of $Q_4'$.  A \sage~computation shows that its $f$-vector equals $(1, 32, 174, 396, 420, 168)$. We are interested in the intersection of the Gr\"obner fan of $Q_4'$ with the relative interior of the cone $\Theta$. The following lemma shows that $\Theta$ gets subdivided into  three maximal pieces
  \begin{equation}\label{eq:ThetaSubdivision}
       \begin{aligned}
    \Theta_0  \!& := \Theta\cap\{d_{34}\geq \ww_2,\, \ww_5\!+\!d_{34}\geq 2 \ww_3\} \,,\;
    \Theta_1 \!:= \Theta\cap\{ 2 \ww_3 \geq \ww_5\!+\!\ww_2, \, 2 \ww_3 \geq \ww_5+d_{34}\},\\
          \Theta_2 \!& := \Theta\,\cap\,\{\ww_2\geq d_{34},\, \ww_5\!+\!\ww_2\geq 2 \ww_3\} \,.
    \end{aligned}
     \end{equation}
  \begin{lemma}\label{lm:SubdivisionOfThetaByQ4}
   The pieces $\Theta_0, \Theta_1$ and $\Theta_2$ in~\eqref{eq:ThetaSubdivision} determine the $\underline{\ww}$-initial form of $Q_4'$:
    \[
    \init_{\underline{\ww}}(Q_4')=\begin{cases}
     -8\alpha_6^6\alpha_5^6\alpha_{34}^2\alpha_3^4 & \quad \text{ if }\quad \underline{\ww} \in \relint(\Theta_0)\;,\\
     -8\alpha_6^6\alpha_5^4\alpha_3^8 & \quad \text{ if }\quad \underline{\ww} \in \relint(\Theta_1)\;,\\
     -8\alpha_6^6\alpha_5^6\alpha_3^4\alpha_2^2 & \quad \text{ if }\quad \underline{\ww} \in \relint(\Theta_2)\;.
    \end{cases}\]
On the intersection of $\Theta_i$ and $\Theta_j$, the initial form is obtained by adding the forms for each piece. On the triple intersection, the initial form equals the sum of the three forms.
  \end{lemma}
  \begin{proof} The proof is  computational, and all the required \sage~scripts  are included in the Supplementary material.
    Since the computation of the Gr\"obner fan of $Q_4$ using \sage~halts,  we replace $Q_4'$ by the sum of its extremal monomials and calculate its Gr\"obner fan. We then compute the intersection of this fan with $\Theta$ and check that only three of its maximal cones intersect $\Theta$ in dimension six. We consider a sample interior point in the relative interior of each piece (e.g.\ the sum of its extremal rays) and determine the initial forms of $Q_4'$ on each $\Theta_i$ using \MacT. The equalities defining $\Theta_i$ are determined by \sage. The last claim in the statement follows from the defining properties of  Gr\"obner fans.
  \end{proof}

  The pieces $\Theta_0$, $\Theta_1$ and $\Theta_2$ have a natural interpretation   in terms of length data:
  \begin{lemma}\label{lm:SubdivMinimalLength}
    Given $i=0,1,2$, the inequalities defining $\Theta_i$ single out the minimal edge length $L_i$ of the corresponding theta graph. 
  \end{lemma}
\noindent  In particular, the subdivision~\eqref{eq:ThetaSubdivision} of $\Theta$ is compatible with the automorphisms of this cone induced by permutations of the underlying theta graph. The proof of this result follows from the length formulas in~\autoref{tab:CombAndLengthData}.   Below is the main result in this section:
  
  \begin{theorem}\label{thm:newIgusa}
    Let $\mathcal{X}$ be a curve in $M_2$,  defined over $K$ with $\charF \resK \neq 2$, and generic with respect to its (abstract) tropicalization $\Gamma$.  
    The tropical Igusa function $j_4^{\trop}$ equals
    \begin{enumerate}[(i)]
    \item $j_4^{\trop}(\Gamma)\! = j_3^{\trop}(\Gamma)$ if $\Gamma$ is  a dumbbell curve, and
      \item $j_4^{\trop}(\Gamma)\! =\! L_0 + L_1 + L_2 -\min\{L_0 , L_1 , L_2\}$ if $\Gamma$ is a theta curve,
    \end{enumerate}
    where $L_0,L_1, L_2$ denote the lengths on each curve as in Figure~\ref{fig:allSkeletaAndMetrics}.    The  formulas remain valid under specialization and yield well-defined piecewise linear maps on $M_2^{\trop}$.
    \end{theorem}
  \begin{proof}
The formula for Type (I) will depend on the characteristic of $\resK$ and will be obtained from Theorems~\ref{thm:tropIgusaExpanded} and~\ref{thm:tropIgusaExpandedChar3}. A simple inspection shows that in all cases $j_4^{\trop}\leq j_3^{\trop}$. The genericity of $\mathcal{X}$ ensures that no cancellations occur and thus, ${j_4}_{\dumb}^{\trop} = {j_3}_{\dumb}^{\trop}$.

To prove the statement on the Type (II) cell, we notice that $j_4$ differs from $j_3$ by replacing $C'$ with $Q_4'$, so $j_4^{\trop}=j_3^{\trop}+\val(C')-\val(Q_4')$. ~\autoref{lm:SubdivisionOfThetaByQ4} and \eqref{eq:valsTypeII} gives
\[
\val(C')-\val(Q_4') =\begin{cases}
-6(\ww_6 + \ww_5 +\ww_3) + 6(\ww_6+\ww_5) + 2 d_{34} + 4\ww_3 = - L_0& \quad \text{ if } \underline{\ww} \in \Theta_0,\\
-6(\ww_6 + \ww_5 +\ww_3) + 6 \ww_6 + 4 \ww_5 + 8 \ww_3 = - L_1 & \quad \text{ if } \underline{\ww} \in \Theta_1,\\
-6(\ww_6 + \ww_5 +\ww_3) + 6(\ww_6+\ww_5)  + 4\ww_3 + 2 \ww_2 =  - L_2 & \quad \text{ if } \underline{\ww} \in \Theta_2.
\end{cases}
\]
By~\autoref{lm:SubdivMinimalLength} and~\autoref{thm:tropIgusaExpanded} we conclude that on Type (II) curves
\[
{j_4}_{\figTheta}^{\trop} = {j_3}_{\figTheta}^{\trop} - \min\{L_0, L_1, L_3\} = L_0 + L_1 + L_2 - \min\{L_0, L_1, L_3\}.
    \]
   Analogous arguments as the ones provided in the proof of~\autoref{thm:tropIgusaExpanded} and the genericity of $\mathcal{X}$ ensure that the given formulas are valid under specialization.
  \end{proof}
  The Igusa functions $j_1, j_2, j_4$ characterize isomorphism types in $M_2$. The tropical Igusa functions $j_1^{\trop}, j_2^{\trop}$ and $j_4^{\trop}$ allow us to recover partial length data for each point in $M_2^{\trop}$, once we determine the combinatorial type of the curve using~\autoref{thm:admisscovers} and~\autoref{tab:CombAndLengthData}. 
  The methods presented in this section will not produce a complete set of tropical invariants on  $M_2^{\trop}$. Indeed, we have exploited the unique relation among the valuations of $A, B$, and $C$ to build $j_4$ and no further combination of $A, B, C$ would produce a cancellation of initial terms. It remains an interesting challenge to develop an alternative approach to generate a new algebraic invariant on $M_2$ inducing the missing tropical invariant on each cell of $M_2^{\trop}$.

\section*{Acknowledgements}
We wish to thank Claus Fieker, Michael Joswig, Lars Kastner, Ralph Morrison, Sam Payne, Padmavathi
Srinivasan, Bernd Sturmfels and Annette Werner for very fruitful
conversations. We thank the two anonymous referees for useful suggestions on an earlier version of this paper. The first author was supported by an Alexander von
Humboldt Postdoctoral Research Fellowship (Germany), an NSF
postdoctoral fellowship DMS-1103857 and an NSF Standard Grant DMS-1700194 (USA). The second author was
supported by the DFG transregional collaborative research center grant SFB-TRR 195 ``Symbolic Tools in Mathematics and their Application'', INST 248/235-1 (Germany).

Part of this project was carried out during the 2013 program on \emph{Tropical Geometry and Topology} at the
Max-Planck Institut f\"ur Mathematik in Bonn (Germany), where the second author was in
residence, and during the 2016 major program on \emph{Combinatorial Algebraic Geometry} at the Fields Institute in Toronto (Canada), where both authors were in residence.  The authors would like to thank both institutes for their hospitality, and for providing excellent working conditions.

\bigskip
\noindent
\textbf{\small{Authors' addresses:}}
\smallskip
\

\noindent
\small{M.A.\ Cueto,  Mathematics Department, The Ohio State University, 231 W 18th Ave, Columbus, OH 43210, USA.
\\
\noindent \emph{Email address:} \texttt{cueto.5@osu.edu}}
\vspace{2ex}

\noindent
\small{H.\ Markwig, Eberhard Karls Universit\"at T\"ubingen, Fachbereich Mathematik, Auf der Morgenstelle 10, 72108 T\"ubingen, Germany.
  \\
  \noindent \emph{Email address:} \texttt{hannah@math.uni-tuebingen.de}}


\normalsize
 

\begin{thebibliography}{10} \label{sec:biblio}

  \bibitem{ACPModuli}
D.~Abramovich, L.~Caporaso, and S.~Payne.
\newblock The tropicalization of the moduli space of curves.
\newblock {\em Ann. Sci. \'Ec. Norm. Sup\'er. (4)}, 48(4):765--809, 2015.

\bibitem{AllermanRau}
L.~Allermann and J.~Rau.
\newblock First steps in tropical intersection theory.
\newblock {\em Math. Z.}, 264(3):633--670, 2010.

\bibitem{bak.nor:09}
M.~Baker and S.~Norine.
\newblock Harmonic morphisms and hyperelliptic graphs.
\newblock {\em Int. Math. Res. Not. IMRN}, (15):2914--2955, 2009.

\bibitem{BPRContempMath}
M.~Baker, S.~Payne, and J.~Rabinoff.
\newblock On the structure of non-{A}rchimedean analytic curves.
\newblock In {\em Tropical and non-{A}rchimedean geometry}, volume 605 of {\em
  Contemp. Math.}, pages 93--121. Amer. Math. Soc., Providence, RI, 2013.

\bibitem{bak.pay.rab:16}
M.~Baker, S.~Payne, and J.~Rabinoff.
\newblock Nonarchimedean geometry, tropicalization, and metrics on curves.
\newblock {\em Algebr. Geom.}, 3(1):63--105, 2016.

\bibitem{berkovichbook}
V.~G. Berkovich.
\newblock {\em Spectral theory and analytic geometry over non-{A}rchimedean
  fields}, volume~33 of {\em Mathematical Surveys and Monographs}.
\newblock Amer. Math. Soc., Providence, RI, 1990.

\bibitem{BHV2001}
L.~J. Billera, S.~P. Holmes, and K.~Vogtmann.
\newblock Geometry of the space of phylogenetic trees.
\newblock {\em Adv. in Appl. Math.}, 27(4):733--767, 2001.

\bibitem{bobrch:17}
B.~Bolognese, M.~Brandt, and L.~Chua.
\newblock From curves to tropical {J}acobians and back.
\newblock In {\em Combinatorial algebraic geometry}, volume~80 of {\em Fields
  Inst. Commun.}, pages 21--45. Fields Inst. Res. Math. Sci., Toronto, ON,
  2017.

\bibitem{brhe:17}
M.~Brandt and P.~A. Helminck.
\newblock Tropical superelliptic curves.
\newblock \href{https://arxiv.org/abs/1709.05761}{\texttt{arXiv:1709.05761}},
  2017.

\bibitem{PlaneModuli2015}
S.~Brodsky, M.~Joswig, R.~Morrison, and B.~Sturmfels.
\newblock Moduli of tropical plane curves.
\newblock {\em Res. Math. Sci.}, 2:Art. 4, 31, 2015.

\bibitem{bru.ite.mik.sha:14}
E.~Brugall\'e, I.~Itenberg, G.~Mikhalkin, and K.~Shaw.
\newblock Brief introduction to tropical geometry.
\newblock In {\em Proceedings of the {G}\"okova {G}eometry-{T}opology
  {C}onference 2014}, pages 1--75. G\"okova Geometry/Topology Conference (GGT),
  G\"okova, 2015.

\bibitem{BLdM11}
E.~A. Brugall{\'e} and L.~M. L{\'o}pez~de Medrano.
\newblock Inflection points of real and tropical plane curves.
\newblock {\em J. Singul.}, 4:74--103, 2012.

\bibitem{buc.mar:15}
A.~Buchholz and H.~Markwig.
\newblock Tropical covers of curves and their moduli spaces.
\newblock {\em Commun. Contemp. Math.}, 17(1):1350045, 27, 2015.

\bibitem{cap:13}
L.~Caporaso.
\newblock Algebraic and tropical curves: comparing their moduli spaces.
\newblock In {\em Handbook of moduli. {V}ol. {I}}, volume~24 of {\em Adv. Lect.
  Math. (ALM)}, pages 119--160. Int. Press, Somerville, MA, 2013.

\bibitem{cap:14}
L.~Caporaso.
\newblock Gonality of algebraic curves and graphs.
\newblock In {\em Algebraic and complex geometry}, volume~71 of {\em Springer
  Proc. Math. Stat.}, pages 77--108. Springer, Cham, 2014.

\bibitem{cav.et.at:17}
R.~Cavalieri, M.~Chan, M.~Ulirsch, and J.~Wise.
\newblock A moduli stack of tropical curves.
\newblock \href{https://arxiv.org/abs/1704.03806}{\texttt{arXiv:1704.03806}},
  2017.

\bibitem{cav.mar.ran:16}
R.~Cavalieri, H.~Markwig, and D.~Ranganathan.
\newblock Tropicalizing the space of admissible covers.
\newblock {\em Math. Ann.}, 364(3-4):1275--1313, 2016.

\bibitem{cha:12}
M.~Chan.
\newblock Combinatorics of the tropical {T}orelli map.
\newblock {\em Algebra Number Theory}, 6(6):1133--1169, 2012.

\bibitem{cha:13}
M.~Chan.
\newblock Tropical hyperelliptic curves.
\newblock {\em J. Algebr. Comb.}, 37(2):331--359, 2013.

\bibitem{cue.mar:16}
M.~A. Cueto and H.~Markwig.
\newblock How to repair tropicalizations of plane curves using modifications.
\newblock {\em Exp. Math.}, 25(2):130--164, 2016.

\bibitem{DGPS}
W.~Decker, G.-M. Greuel, G.~Pfister, and H.~Sch\"onemann.
\newblock {\sc Singular} {4-1-0} --- {A} computer algebra system for polynomial
  computations.
\newblock \url{http://www.singular.uni-kl.de}, 2016.

\bibitem{ComputingGroebnerFans}
K.~Fukuda, A.~N. Jensen, and R.~R. Thomas.
\newblock Computing {G}r\"obner fans.
\newblock {\em Math. Comp.}, 76(260):2189--2212 (electronic), 2007.

\bibitem{GKM}
A.~Gathmann, M.~Kerber, and H.~Markwig.
\newblock Tropical fans and the moduli spaces of tropical curves.
\newblock {\em Compos. Math.}, 145(1):173--195, 2009.

\bibitem{polymake}
E.~Gawrilow and M.~Joswig.
\newblock polymake: a framework for analyzing convex polytopes.
\newblock In G.~Kalai and G.~M. Ziegler, editors, {\em Polytopes ---
  Combinatorics and Computation}, pages 43--74. Birkh\"auser, 2000.

\bibitem{GP}
L.~Gerritzen and M.~van~der Put.
\newblock {\em Schottky groups and {M}umford curves}, volume 817 of {\em
  Lecture Notes in Mathematics}.
\newblock Springer, Berlin, 1980.

\bibitem{gib.mac:10}
A.~Gibney and D.~Maclagan.
\newblock Equations for {C}how and {H}ilbert quotients.
\newblock {\em Algebra Number Theory}, 4(7):855--885, 2010.

\bibitem{gor.lau:12}
E.~Z. Goren and K.~E. Lauter.
\newblock Genus 2 curves with complex multiplication.
\newblock {\em Int. Math. Res. Not. IMRN}, 5:1068--1142, 2012.

\bibitem{M2}
D.~R. Grayson and M.~E. Stillman.
\newblock Macaulay2, a software system for research in algebraic geometry.
\newblock Available at \url{http://www.math.uiuc.edu/Macaulay2/}, 2009.

\bibitem{gro:16}
A.~Gross.
\newblock Correspondence theorems via tropicalizations of moduli spaces.
\newblock {\em Commun. Contemp. Math.}, 18(3):1550043, 36, 2016.

\bibitem{gub.rab.wer:16}
W.~Gubler, J.~Rabinoff, and A.~Werner.
\newblock Skeletons and tropicalizations.
\newblock {\em Adv. Math.}, 294:150--215, 2016.

\bibitem{gub.rab.wer:15}
W.~Gubler, J.~Rabinoff, and A.~Werner.
\newblock Tropical skeletons.
\newblock {\em Ann. Inst. Fourier (Grenoble)}, 67(5):1905--1961, 2017.

\bibitem{hel:16}
P.~Helminck.
\newblock Tropical {I}gusa invariants and torsion embeddings.
\newblock \href{http://arxiv.org/abs/1604.03987}{\texttt{arXiv:1604.03987}}.

\bibitem{Igusa01}
J.~Igusa.
\newblock Arithmetic variety of moduli for genus two.
\newblock {\em Ann. of Math. (2)}, 72:612--649, 1960.

\bibitem{IMS09}
I.~Itenberg, G.~Mikhalkin, and E.~Shustin.
\newblock {\em Tropical algebraic geometry}, volume~35 of {\em Oberwolfach
  Seminars}.
\newblock Birkh\"auser Verlag, Basel, second edition, 2009.

\bibitem{gfan}
A.~N. Jensen.
\newblock Gfan, a software system for {G}r{\"o}bner fans and tropical
  varieties.
\newblock Available at
  \url{http://www.math.tu-berlin.de/~jensen/software/gfan/gfan.html}, 2009.

\bibitem{tropicalLib}
A.~N. Jensen, H.~Markwig, and T.~Markwig.
\newblock \texttt{tropical.lib}. {A} {SINGULAR 3.0} library for computations in
  tropical geometry.
\newblock \url{http://www.math.uni-tuebingen.de/user/keilen/en/tropical.html},
  2007.

\bibitem{KMM08}
E.~Katz, H.~Markwig, and T.~Markwig.
\newblock The {$j$}-invariant of a plane tropical cubic.
\newblock {\em J. Algebra}, 320(10):3832--3848, 2008.

\bibitem{ker.mar:09}
M.~Kerber and H.~Markwig.
\newblock Counting tropical elliptic plane curves with fixed {$j$}-invariant.
\newblock {\em Comment. Math. Helv.}, 84(2):387--427, 2009.

\bibitem{LY11}
K.~Lauter and T.~Yang.
\newblock Computing genus 2 curves from invariants on the {H}ilbert moduli
  space.
\newblock {\em J. Number Theory}, 131(5):936--958, 2011.

\bibitem{Liu93}
Q.~Liu.
\newblock Courbes stables de genre {$2$} et leur sch\'ema de modules.
\newblock {\em Math. Ann.} 295(2):201--222, 1993.

\bibitem{MSBook}
D.~Maclagan and B.~Sturmfels.
\newblock {\em Introduction to tropical geometry}, volume 161 of {\em Graduate
  Studies in Mathematics}.
\newblock American Mathematical Society, Providence, RI, 2015.

\bibitem{Mestre90}
J.-F. Mestre.
\newblock Construction de courbes de genre {$2$} \`a partir de leurs modules.
\newblock In {\em Effective methods in algebraic geometry ({C}astiglioncello,
  1990)}, volume~94 of {\em Progr. Math.}, pages 313--334. Birkh\"auser Boston,
  Boston, MA, 1991.

\bibitem{MikhalkinICM}
G.~Mikhalkin.
\newblock Tropical geometry and its applications.
\newblock In {\em International {C}ongress of {M}athematicians. {V}ol. {II}},
  pages 827--852. Eur. Math. Soc., Z\"urich, 2006.

\bibitem{mik:07}
G.~Mikhalkin.
\newblock Moduli spaces of rational tropical curves.
\newblock In {\em Proceedings of {G}\"okova {G}eometry-{T}opology {C}onference
  2006}, pages 39--51. G\"okova Geometry/Topology Conference (GGT), G\"okova,
  2007.

\bibitem{analytifAndTrop}
S.~Payne.
\newblock Analytification is the limit of all tropicalizations.
\newblock {\em Math. Res. Lett.}, 16(3):543--556, 2009.

\bibitem{ren.sam.stu:14}
Q.~Ren, S.~V. Sam, and B.~Sturmfels.
\newblock Tropicalization of classical moduli spaces.
\newblock {\em Math. Comput. Sci.}, 8(2):119--145, 2014.

\bibitem{FirstSteps}
J.~Richter-Gebert, B.~Sturmfels, and T.~Theobald.
\newblock First steps in tropical geometry.
\newblock In {\em Idempotent mathematics and mathematical physics}, volume 377
  of {\em Contemp. Math.}, pages 289--317. Amer. Math. Soc., Providence, RI,
  2005.

\bibitem{TropGrass}
D.~Speyer and B.~Sturmfels.
\newblock The tropical {G}rassmannian.
\newblock {\em Adv. Geom.}, 4(3):389--411, 2004.

\bibitem{sage}
W.~Stein.
\newblock {\em {S}age {M}athematics {S}oftware ({V}ersion 6.8)}, 2015.
\newblock \url{http://www.sagemath.org}.

\bibitem{CompactificationsTori}
J.~Tevelev.
\newblock Compactifications of subvarieties of tori.
\newblock {\em Amer. J. Math.}, 129(4):1087--1104, 2007.

\bibitem{tyo:12}
I.~Tyomkin.
\newblock Tropical geometry and correspondence theorems via toric stacks.
\newblock {\em Math. Ann.}, 353(3):945--995, 2012.

\bibitem{wag:17}
T.~Wagner.
\newblock Faithful tropicalization of {M}umford curves of genus two.
\newblock {\em Beitr. Algebra Geom.}, 58(1):47--67, 2017.

\end{thebibliography}
\end{document}